\documentclass[a4paper]{article}

\usepackage[utf8]{inputenc}
\usepackage[width=15cm,height=22cm]{geometry}
\usepackage{amsfonts,amsmath,amssymb,amsthm}
\usepackage[dvipsnames]{xcolor}
\usepackage[hidelinks,pdfusetitle]{hyperref}
\usepackage[capitalise]{cleveref}
\crefname{assumption}{Assumption}{Assumptions}

\newtheorem{proposition}{Proposition}[section]
\newtheorem{lemma}[proposition]{Lemma}
\newtheorem{assumption}[proposition]{Assumption}
\newtheorem{corollary}[proposition]{Corollary}

\newtheorem{example}[proposition]{Example}
\newtheorem{theorem}[proposition]{Theorem}
\newtheorem{open}[proposition]{Open Problem}
\newtheorem{remark}[proposition]{Remark}
\newtheorem{definition}[proposition]{Definition}
\numberwithin{equation}{section}

\title{Time-iteration methods for controllability}
\author{Frédéric Marbach}
\date{\today}

\newcommand{\C}{\mathbb{C}}
\newcommand{\R}{\mathbb{R}}
\newcommand{\N}{\mathbb{N}}
\newcommand{\T}{\mathbb{T}}
\newcommand{\Z}{\mathbb{Z}}
\newcommand{\dd}{\mathrm{d}}

\newcommand{\norm}[1]{\lVert#1\rVert}
\newcommand{\abs}[1]{\lvert#1\rvert}

\newcommand{\cO}{\mathcal{O}}
\newcommand{\cR}{\mathcal{R}}
\newcommand{\cL}{\mathcal{L}}
\newcommand{\cS}{\mathcal{S}}
\newcommand{\cX}{\mathcal{X}}
\newcommand{\cY}{\mathcal{Y}}
\newcommand{\Id}{\operatorname{Id}}
\newcommand{\Ran}{\operatorname{Ran}}
\newcommand{\vect}{\operatorname{span}}
\newcommand{\supp}{\operatorname{supp}}

\usepackage{etoolbox}
\newcounter{proofcounter}
\newcounter{stepcounter}[proofcounter]
\newcommand{\step}[1]{%
\ifnum\value{stepcounter}>0 \medskip\fi
\refstepcounter{stepcounter}\noindent\emph{\textbf{Step \thestepcounter:} #1.}}
\newcommand{\eqstep}[1]{%
\ifnum\value{stepcounter}>0 \medskip\fi
\refstepcounter{stepcounter}\noindent\emph{\textbf{Step \thestepcounter:} #1}}
\AtBeginEnvironment{proof}{\stepcounter{proofcounter} \setcounter{stepcounter}{0}}

\begin{document}

\maketitle

\begin{abstract}
    These notes are based on a short course delivered at the Summer School EUR MINT 2025 ``\emph{Control, Inverse Problems and Spectral Theory}'', held in June 2025 in Toulouse, France.
    The course presents three important strategies in control theory, formulated as time-iteration methods, where each time step brings the state of the system closer to the desired target.

    For linear PDEs, we survey the classical Lebeau--Robbiano method and its more recent developments.
    This method combines spectral inequalities and dissipation estimates to prove null controllability of a dissipative linear system.

    For nonlinear PDEs, we reinterpret the Liu--Takahashi--Tucsnak method, which establishes local controllability of a nonlinear system by analyzing the control cost of its linearization.
    We provide an easily applicable black-box formulation of their method.

    Finally, for nonlinear ODEs, we present the tangent vectors method, which establishes local exact controllability starting from approximately reachable directions.
\end{abstract}

\newpage

\setcounter{tocdepth}{2}
\tableofcontents

\newpage
\section{A prequel on the controllability of ODEs}
\label{sec:intro}

Controllability refers to the possibility of choosing a time-varying input in an evolution equation to drive the state towards a desired target.
This self-contained prequel highlights a few basic results on the controllability of ODEs.
It is included as a reference for newcomers, and to fix notations.

For broader context, we refer the readers to classical textbooks on mathematical control theory, for example \cite{BulloLewis2005,Jurdjevic1997,Sontag1998} for ODEs, \cite{Coron2007,TucsnakWeiss2009} for PDEs, or \cite{Liberzon2012} for optimal control.

\medskip

In this section, we fix $n, m \in \N^*$ and $p \in [1,\infty]$.

\subsection{The rank condition for linear systems}

We consider the linear time-invariant control system on $\R^n$:
\begin{equation}
    \label{eq:lti_system}
    \dot{y}(t) = A y(t) + B u(t),
\end{equation}
where $y(t) \in \R^n$ is the state, $u(t) \in \R^m$ the control, and $A \in \R^{n \times n}$, $B \in \R^{n \times m}$ constant matrices.

\begin{lemma}
    For any $T > 0$, any initial data $y_0 \in \R^n$ and any control $u \in L^p((0,T); \R^m)$, there exists a unique mild solution $y \in C^0([0,T];\R^n)$ to \eqref{eq:lti_system} with $y(0) = y_0$ given by:
    \begin{equation}
        \label{eq:variation_constants}
        y(t) = e^{tA} y_0 + \int_0^t e^{(t-s)A} B u(s) \dd s.
    \end{equation}
\end{lemma}

\begin{definition}[Controllability]
    Let $T > 0$.
    System \eqref{eq:lti_system} is said to be \emph{controllable in time~$T$} when, for any initial state $y_0 \in \R^n$, and any final state $y_T \in \R^n$, there exists $u \in L^p((0,T);\R^m)$ such that the solution to \eqref{eq:lti_system} with $y(0)=y_0$ satisfies $y(T)=y_T$.
\end{definition}

The fundamental result of control theory for linear time-invariant systems, dating back to 1960 in \cite[Corollary 5.5]{Kalman1960} or \cite[Theorem 6 and p.\ 15]{LaSalle1960}, is that controllability is equivalent to a simple algebraic condition on the matrices $A$ and $B$, which happens to be independent of $T$ (and $p$).

\begin{theorem}[Rank condition]
    \label{thm:kalman}
    Let $T > 0$.
    System \eqref{eq:lti_system} is controllable in time $T$ if and only if
    \begin{equation}
        \label{eq:kalman}
        \operatorname{rank} [B, AB, A^2B, \dots, A^{n-1}B] = n.
    \end{equation}
\end{theorem}

\begin{proof}
    Using \eqref{eq:variation_constants}, one checks that system \eqref{eq:lti_system} is controllable in time $T$ if and only if $\cR_T = \R^n$ where $\cR_T$ denotes the set of reachable states from the origin:
    \begin{equation}
        \cR_T := \{ y(T) \mid y \text{ is a solution to \eqref{eq:lti_system} with } y(0) = 0 \}.
    \end{equation}
    Thus $\cR_T$ is the image of the linear operator $\Phi_T : L^p((0,T);\R^m) \to \R^n$ defined by
    \begin{equation}
        \label{eq:lti-PhiT}
        \Phi_T(u) := \int_0^T e^{(T-t)A} B u(t) \dd t.
    \end{equation}
    \emph{First, assume \eqref{eq:kalman}.}
    Let us prove that $\cR_T = \R^n$.
    Let $z \in (\cR_T)^{\perp}$.
    Then, for any $u \in L^p((0,T);\R^m)$,
    \begin{equation}
        0
        = \langle z, \Phi_T(u) \rangle
        = z^\top \Phi_T(u)
        = \int_0^T (B^\top e^{(T-t) A^\top} z)^\top u(t) \dd t.
    \end{equation}
    Hence the analytic function $t \mapsto B^\top e^{(T-t) A^\top} z$ from $[0,T] \to \R^m$ is identically $0$.
    Differentiating $n-1$ times at $t = T$, we obtain that $(-1)^k B^\top (A^\top)^k z = 0$ for $k = 0, \dotsc, n-1$.
    Thus $z$ is orthogonal to the columns of $[B, A B, \dotsc, A^{n-1} B]$, which span $\R^n$.
    Hence $z = 0$, so $(\cR_T)^\perp = \{ 0 \}$.

    \medskip

    \noindent
    \emph{Second, assume that $\operatorname{rank} [B, A B, A^2 B, \dots, A^{n-1} B] < n$.}
    Then there exists a nonzero $z \in \R^n$, such that $z^\top A^k B = 0$ for all $k = 0, \dotsc, n-1$.
    By Cayley--Hamilton, this holds for all $k \geq 0$.
    Then, for any $u \in L^p((0,T);\R^m)$,
    \begin{equation}
        \langle z, \Phi_T(u) \rangle = \sum_{k=0}^\infty z^\top A^k B \int_0^T \frac{(T-t)^k}{k!} u(t) \dd t = 0.
    \end{equation}
    Thus $z \in (\cR_T)^\perp \neq \{ 0 \}$.
\end{proof}

Let us give two archetypal examples of controllable linear systems with $m = 1$.

\begin{example}[Integrator chain]
    \label{ex:chain}
    Consider the system $\dot{y}_1 = u$, $\dot{y}_2 = y_1$, $\dot{y}_3 = y_2$, ..., $\dot{y}_n = y_{n-1}$.
    This system is of the form $\dot{y} = Ay + Bu$ with
    \begin{equation}
        A =
        \begin{pmatrix}
            0 & 0 & \dotsb & 0 & 0 \\
            1 & 0 & \dotsb & 0 & 0 \\
            0 & 1 & \dotsb & 0 & 0 \\
            \vdots & \vdots & \ddots & \vdots & \vdots \\
            0 & 0 & \dotsb & 1 & 0
        \end{pmatrix}
        \quad \text{and} \quad
        B =
        \begin{pmatrix}
            1 \\ 0 \\ \vdots \\ 0
        \end{pmatrix}
        .
    \end{equation}
    Hence $B = e_1$ the first unit vector and $A^k B = e_{k+1}$ for $0 \leq k \leq n-1$.
    Thus \eqref{eq:kalman} is satisfied and the system is controllable in any time $T$.
\end{example}

\begin{remark}
    Though it may seem like a particular case, it can be proved that any system satisfying the rank condition \eqref{eq:kalman} can be reduced to the form of \cref{ex:chain} by \emph{static state feedback} and \emph{linear change of coordinates}.
    This reduction is called the \emph{Brunovsk\'y normal form} of the system (see \cite[Theorem 2]{Brunovsky1970} or \cite[Chapter 5, Theorem 2]{Jurdjevic1997}).
\end{remark}

\begin{example}[Diagonal system]
    \label{ex:diag}
    Consider the system $\dot{y}_j = - \lambda_j y_j + j u(t)$ where $\lambda_j = j^2$.
    This system is of the form $\dot{y} = Ay + Bu$ with
    \begin{equation}
        A = \operatorname{diag} (-\lambda_1, \dotsc, -\lambda_n)
        \quad \text{and} \quad
        B =
        \begin{pmatrix}
            1 \\ \vdots \\ n
        \end{pmatrix}
        .
    \end{equation}
    Thus one checks that $[B, AB, \dotsc, A^{n-1}B] = B^\top V$ where $V$ is the Vandermonde matrix associated with $-\lambda_1, \dotsc, -\lambda_n$, which is nonsingular since these values are distinct.
    Thus \eqref{eq:kalman} is satisfied and the system is controllable in any time $T > 0$.

    We will investigate the null-control cost of this system in \cref{s:large-systems} and use it as a finite approximation of a 1D heat equation controlled from the boundary in \cref{s:heat-boundary}.
\end{example}

\subsection{About the cost of fast controls}
\label{sec:seidman}

Assume that the linear time-invariant system \eqref{eq:lti_system} is controllable.
Hence, for all $T > 0$ and~$y_0 \in \R^n$, there exists at least one control $u$ such that the corresponding solution satisfies $y(T) = 0$.

In this paragraph, we quantify the cost of such null controllability.
We are interested in a worst-case measure: the largest, over all $y_0$ in the unit ball, of the minimal size of a control needed to drive the state to zero at time $T$.
It is natural to expect that this cost increases as $T \to 0$.

Understanding the asymptotic behavior of this cost as $T \to 0$ is crucial for transferring controllability properties from finite-dimensional ODEs to dissipative PDEs (see \cref{sec:LR}), and subsequently from linear PDEs to nonlinear ones (see \cref{sec:lin-2-nonlin}).

\bigskip

Let $\abs{\cdot}$ be the Euclidean norm on $\R^n$.
For $T > 0$, define the \emph{null-control cost} as
\begin{equation}
    \mathcal{K}^p(T) := \sup_{y_0 \in \R^n, \, \abs{y_0} = 1} \inf \Big\{ \| u \|_{L^p} \mid u \in L^p((0,T);\R^m), \enskip y(0) = y_0 \text{ and } y(T) = 0 \Big\}.
\end{equation}
The following result is due to Seidman and Yong \cite{SeidmanYong1996} (see \cite{Seidman1988} for $p=2$).

\begin{proposition}
    \label{p:seidman}
    Let $0 \leq K \leq n-1$ be the minimal integer such that
    \begin{equation}
        \label{eq:seidman-rank}
        \operatorname{rank} [B, A B, A^2 B, \dotsc, A^K B] = n.
    \end{equation}
    There exists $c > 0$ such that, as $T \to 0$,
    \begin{equation}
        \label{eq:seidman}
        \mathcal{K}^p(T) \sim c T^{-K-1+\frac 1 p}.
    \end{equation}
\end{proposition}

\begin{proof}
    \step{Time scaling}
    Given $u \in L^p((0,T);\R^m)$, we introduce $v(s) := u(s T)$.
    From the Duhamel formula \eqref{eq:variation_constants} at $t = T$, the condition $y(T) = 0$ is equivalent to $F_T(v) = - y_0 / T$ where
    \begin{equation}
        \label{eq:FTv}
        F_T(v) := \int_0^1 e^{-s (T A)} B v(s) \dd s.
    \end{equation}
    Since $\norm{u}_{L^p(0,T)} = T^{\frac 1 p} \norm{v}_{L^p(0,1)}$, we obtain by homogeneity that
    \begin{equation}
        \label{eq:TildeK-2-K}
        \mathcal{K}^p(T) = T^{\frac 1 p - 1} \widetilde{\mathcal{K}}(T)
        \quad \text{where} \quad
        \widetilde{\mathcal{K}}(T) := \sup_{z \in \R^n, \, \abs{z} = 1} \inf \big\{ \norm{v}_{L^p} \mid F_T(v) = z \big\}.
    \end{equation}

    \step{Decomposition of the state space}
    Write $B = (b_1, \dotsc, b_m)$.
    For $0 \leq k \leq K$, define $S_k := \operatorname{span} \{ A^j b_l \mid 0 \leq j \leq k, 1 \leq l \leq m \} \subset \R^n$ and $E_k$ the orthogonal complement of $S_{k-1}$ in~$S_k$, with the convention that $S_{-1} = \{0\}$.
    By assumption $E_0 \oplus \dotsb \oplus E_K = \R^n$ and $E_K \neq \{ 0 \}$.
    Let $\mathbb{P}_k$ denote the orthogonal projection on $E_k$.
    Thus $\mathbb{P}_k A^j b_l = 0$ for any $j < k$.

    Finally, for $z \in \R^n$, define $\Delta_T(z) := \sum_{k=0}^K T^{-k} \mathbb{P}_k z$ with inverse $\Delta_T^{-1} (z) := \sum_{k=0}^K T^k \mathbb{P}_k z$.

    \step{Asymptotic expansion of $F_T$}
    Expanding the exponential in \eqref{eq:FTv}, we see that $F_T$ scales like $T^k$ on each $E_k$.
    More precisely, $\Delta_T F_T = \Lambda + R_T$ where, for $v \in L^p((0,1);\R^m)$,
    \begin{align}
        \label{eq:seidman-Lambda}
        \Lambda(v) & := \sum_{k=0}^K \sum_{l = 1}^m \mathbb{P}_k (A^k b_l) \int_0^1 \frac{(-s)^k}{k!} v_l(s) \dd s \\
        R_T(v) & := \sum_{k=0}^K \sum_{l = 1}^m \sum_{j=k+1}^\infty \frac{\mathbb{P}_k (A^j b_l)}{j!} T^{j-k} \int_0^1 (-s)^j v_l(s) \dd s.
    \end{align}
    In particular, for $T \leq 1$,
    \begin{equation}
        \abs{R_T(v)} \leq (K+1) \sum_{l=1}^m \sum_{j=0}^{+\infty} \frac{\norm{A}^j \abs{b_l}}{j!} T \norm{v_l}_{L^1} \leq T C_R \norm{v}_{L^p}.
    \end{equation}

    \step{The map $\Lambda$ is onto from $L^p((0,1);\R^m)$ to $\R^n$}
    Each $E_k$ is spanned by the $\mathbb{P}_k (A^k b_l)$.
    Thus, for each $z \in \R^n$, there exist $\alpha_0, \dotsc, \alpha_K \in \R^m$ such that $z = \sum \mathbb{P}_k (A^k B \alpha_k)$.
    Moreover, for each $0 \leq k \leq K$, there exists $\varphi_k \in C^\infty([0,1];\R)$ such that $\int_0^1 (-s)^j \varphi_k(s) \dd s = j! \mathbf{1}_{j=k}$ for all $0 \leq j \leq K$.
    Then $\Lambda(v) = z$ where $v(s) := \sum \varphi_k (s) \alpha_k$.
    So $\Lambda$ is onto.

    In particular, there exists $U \in \mathcal{L}(\R^n;L^p((0,1);\R^m))$ such that $\Lambda \circ U = \Id$ and $C_\Lambda := \norm{U} < \infty$.

    \step{Definition of the constant}
    Since $\Lambda$ is bounded and $E_K \neq \{ 0 \}$,
    \begin{equation}
        \label{eq:def-c-seidman}
        c := \sup_{|z|=1, \, z \in E_K} \inf \big\{ \norm{v}_{L^p} \mid \Lambda (v) = z \big\} \in (0,C_\Lambda].
    \end{equation}
    For all $\varepsilon > 0$ and $z \in E_K$, there exists $v := V_\varepsilon(z)$ such that $\norm{V_\varepsilon(z)} \leq (c+\varepsilon) \abs{z}$ and $\Lambda(V_\varepsilon(z)) = z$.

    \step{Upper bound on the cost}
    For all $\varepsilon > 0$, there exists $T_\varepsilon > 0$ such that, for all $T \in (0,T_\varepsilon]$ and all $z \in \R^n$, there exists $v \in L^p((0,1);\R^m)$ such that $F_T(v) = z$ and $\norm{v}_{L^p} \leq (c + 2 \varepsilon) T^{-K} \abs{z}$.

    Let $T, \varepsilon > 0$ and $z \in \R^n$.
    Decompose $z = z_K + z'$ where $z_K \in E_K$ and $z' \in E_K^\perp$.
    To solve $F_T (v) = z$, we set $v := v_0 + v_1 + \dotsb$ where we define $v_0 := V_\varepsilon(\Delta_T(z_K)) + U(\Delta_T(z'))$ and $v_{i+1} := - U(R_T (v_i))$.
    Thus $\norm{v_0} \leq (c + \varepsilon) T^{-K} \abs{z_K} + C_\Lambda T^{-K+1} \abs{z'} \leq (c + \varepsilon + C_\Lambda T) T^{-K} \abs{z}$.
    Moreover $\norm{v_{i+1}} \leq C_\Lambda T C_R \norm{v_i}$.
    Thus the series converges normally and $\norm{v} \leq (c + \varepsilon + T C') T^{-K} \abs{z}$.
    By a telescoping sum argument, $F_T(v) = z$.
    Choosing $T_\varepsilon$ small enough, we obtain $\norm{v} \leq (c + 2\varepsilon) T^{-K} \abs{z}$.

    \step{Lower bound on the cost}
    For all $\varepsilon > 0$, there exists $T_\varepsilon > 0$ such that, for all $T \in (0,T_\varepsilon]$, there exists $z \in \R^n$ with $z \neq 0$ such that for all $v \in F_T^{-1}(z)$, one has $\norm{v}_{L^p} \geq (c - 2 \varepsilon) T^{-K} \abs{z}$.

    Let $\varepsilon > 0$.
    By definition \eqref{eq:def-c-seidman} of $c$, there exists $z^* \in E_K$ with $z^* \neq 0$ such that $\Lambda(v) = z^*$ implies $\norm{v}_{L^p} \geq (c - \varepsilon) \abs{z^*}$.
    For $T > 0$, set $z_T := \Delta_T^{-1}(z^*) = T^K z^*$.
    Let $v_T$ such that $F_T(v_T) = z_T$.
    Thus $(\Lambda + R_T)(v_T) = z^*$.
    Let $w_T := U (R_T(v_T))$ so that $\Lambda (w_T) = R_T (v_T)$.
    Then $\Lambda (v_T+w_T) = z^* \in E_K$.
    Hence $\norm{v_T+w_T} \geq (c-\varepsilon) \abs{z^*}$.
    But $\norm{w_T} \leq C_\Lambda T C_R \norm{v_T}$.
    Hence $(1 + T C_\Lambda C_R) \norm{v_T} \geq (c - \varepsilon) \abs{z^*}
    = (c - \varepsilon) T^{-K} \abs{z_T}$.
    For $T \in (0,T_\varepsilon]$ where $T_\varepsilon := \varepsilon / (C_\Lambda C_R (c - 2 \varepsilon))$, we have found $z_T \in \R^n$ with $z_T \neq 0$ such that any $v_T$ satisfying $F_T(v_T) = z_T$ satisfies $\norm{v_T} \geq
    (c - 2\varepsilon) T^{-K} \abs{z_T}$.

    \step{Conclusion}
    We proved that $\widetilde{\mathcal{K}}(T) \sim c T^{-K}$.
    Recalling \eqref{eq:TildeK-2-K} proves \eqref{eq:seidman}.
\end{proof}

\begin{example}[Integrator chain]
    Recall \cref{ex:chain}.
    Since $B = e_1$ and $A^k B = e_{k+1}$, we have $K = n-1$ in \eqref{eq:seidman-rank}.
    With the notations of the proof of \cref{p:seidman}, recalling \eqref{eq:seidman-Lambda},
    \begin{equation}
        \Lambda(v) = \sum_{k=0}^{n-1} e_{k+1} \int_0^1 \frac{(-s)^k}{k!} v(s) \dd s
    \end{equation}
    and, as in \eqref{eq:def-c-seidman}, since $E_K = \R e_n$, by homogeneity,
    \begin{equation}
        c = \inf \big\{ \norm{v}_{L^p} \mid \Lambda(v) = e_n \big\}.
    \end{equation}
    Take $p = 2$ to fix ideas.
    Then $c$ is the minimal $L^2$ norm of a function on $[0,1]$ orthogonal to $1, s, s^2, \dotsc, s^{n-2}$ and such that $\int_0^1 s^{n-1} v(s) \dd s = (-1)^{n-1} (n-1)!$.
    Using the properties of Legendre polynomials, one can prove that $c = \sqrt{2n-1} \frac{(2n-2)!}{(n-1)!}$.
\end{example}

\subsection{About the cost of large systems}
\label{s:large-systems}

In the previous paragraph, we computed the asymptotic behavior as $T \to 0$ of the null-control cost for a fixed system of a given size $n \in \N^*$.
When moving on to PDEs in \cref{sec:LR}, we will approximate a PDE by its (increasingly larger) finite dimensional projections.

As an example, consider again the diagonal system \cref{ex:diag}.
Equip $\R^n$ with the norm $\norm{y}_h^2 := \sum y_j^2 / j^2$ (reminiscent of the $H^{-1}$ norm for Fourier series).
Define
\begin{equation}
    \mathcal{K}_n(T) := \sup_{y_0 \in \R^n, \, \norm{y}_h=1} \inf_{u \in L^2((0,T);\R)} \Big\{ \norm{u}_{L^2} \mid y(0) = y_0 \text{ and } y(T) = 0 \Big\}.
\end{equation}
For a fixed $n$, we know by \cref{p:seidman} that $\mathcal{K}_n(T) \sim c_n T^{-n+\frac 12}$ for some constant $c_n$.
This estimate is only valid for $T \leq T_n$.
In \cref{s:heat-boundary}, we will use the following estimate.

\begin{proposition}
    \label{thm:fin-cost}
    For all $T > 0$ and $n \geq 1$,
    \begin{equation}
        \mathcal{K}_n(T) \leq \sqrt{\frac{n^2}{T}} \exp \left( 6 \sqrt{\frac n T} \right).
    \end{equation}
\end{proposition}

\begin{proof}
    Let $T > 0$, $n \geq 1$ and $y_0 \in \R^n$ with $\norm{y_0}_h = 1$.
    By the Duhamel formula \eqref{eq:variation_constants} for the system of \cref{ex:diag}, we have $y(T) = 0$ if and only if, for all $1 \leq j \leq n$,
    \begin{equation}
        \label{eq:moments}
        \int_0^T u(t) e^{j^2 t} \dd t = - \frac{y_0^j}{j} =: b_j \in \R.
    \end{equation}
    Hence $b \in \R^n$ with $\abs{b}_2 = \norm{y_0}_h = 1$ is given, and we are looking for a $u \in L^2((0,T);\R)$ of minimal norm satisfying these conditions.
    Such a problem is called a \emph{moment problem}, and, starting with \cite{FattoriniRussell1971,FattoriniRussell1974}, the mathematical literature contains many works on such problems where one tries to solve \eqref{eq:moments} simultaneously for all $j \in \N^*$.
    Here, we focus on the question (and the cost) of the finite moment problem for $1 \leq j \leq n$.

    By projection on the subspace of $L^2((0,T);\R)$ spanned by $e^{t}, \dotsc, e^{n^2 t}$, the control of minimal norm can be written $u(t) := \sum_{j=1}^n a_j e^{j^2 t}$ for some $a \in \R^n$.
    Let us introduce the Gram matrix $M \in \R^{n \times n}$ defined as
    \begin{equation}
        \label{eq:def-Mij}
        M_{ij} := \int_0^T e^{(i^2+j^2) t} \dd t.
    \end{equation}
    With this notation, solving \eqref{eq:moments} for $1 \leq j \leq n$ is equivalent to the equation $M a = b$.
    To obtain a cost estimate, write
    \begin{equation}
        \norm{u}_{L^2}^2 = a^\top M a = (M^{-1} b)^\top M (M^{-1} b)
        = b^\top M^{-1} b.
    \end{equation}
    Since $\abs{b}_2 = \norm{y_0}_h$, this proves that $\mathcal{K}_n(T) = \norm{M^{-1}}^{\frac 12}$ where, for a matrix $A$, we use the spectral norm $\norm{A} := \max_{\abs{x}_2=1} \abs{Ax}_2$.

    Starting from \eqref{eq:def-Mij}, the idea of the proof is to approximate $M$ by a Riemann sum:
    \begin{equation}
        \int_0^T e^{(i^2+j^2)t} \dd t \approx \frac{T}{n} \sum_{k=1}^n e^{(i^2+j^2) \frac{(k-1) T}{n}}.
    \end{equation}
    Heuristically, we thus have $M \approx \frac{T}{n} U U^\top$ where we define
    \begin{equation}
        \label{eq:Uij}
        U_{ij} = e^{i^2 (j-1) \frac{T}{n}}.
    \end{equation}
    By \cref{lem:M-1-U-1}, $\norm{M^{-1}} \leq \frac{n}{T} \norm{U^{-1}}^2$.
    By \cref{lem:U-1-exp}, $\norm{U^{-1}} \leq \sqrt{n} \exp ( 6 \sqrt{n/T} )$.
\end{proof}

\begin{lemma}
    \label{lem:M-1-U-1}
    For any $n \geq 1$ and $T > 0$, for the matrices of \eqref{eq:def-Mij} and \eqref{eq:Uij},
    \begin{equation}
        \norm{M^{-1}} \leq \frac{n}{T} \norm{U^{-1}}^2.
    \end{equation}
\end{lemma}

\begin{proof}
    Let $D(t) := \operatorname{diag} (e^{i^2 t})$ and $U(t) := D(t) U$.
    Then $M = \int_0^{\frac T n} U(t) U(t)^\top \dd t$.
    Indeed, for $1 \leq i, j \leq n$, letting $\sigma := i^2 + j^2 > 0$, we have
    \begin{equation}
        \int_0^{\frac T n} (U(t) U(t)^\top)_{ij} \dd t
        = \sum_{k=1}^n e^{\sigma (k-1)\frac{T}{n}} \int_0^{\frac T n} e^{\sigma t} \dd t
        = \frac{e^{\sigma T}-1}{e^{\sigma T / n}-1} \frac{e^{\sigma T/n}-1}{\sigma}
        = M_{ij}.
    \end{equation}
    For a real symmetric matrix $A$ one has $\norm{A^{-1}}^{-1} = \lambda_{\min}(A) = \inf_{\abs{x}_2 = 1} x^\top A x$.

    For any $x \in \R^n$, we have
    \begin{equation}
        x^\top M x
        = \int_0^{\frac T n} x^\top D(t) U U^\top D(t) x \dd t
        \geq \int_0^{\frac T n} \lambda_{\min}(U U^\top) \abs{D(t) x}_2^2 \dd t
        \geq \frac T n \lambda_{\min}(U U^\top) \abs{x}_2^2
    \end{equation}
    since $\abs{D(t) x}_2 \geq \abs{x}_2$.
    Thus $\norm{M^{-1}} \leq \frac n T \norm{(U U^\top)^{-1}} \leq \frac{n}{T} \norm{U^{-1}}^2$.
\end{proof}

\begin{lemma}
    \label{lem:U-1-exp}
    For any $n \geq 1$ and $T > 0$, for the matrix $U$ of \eqref{eq:Uij},
    \begin{equation}
        \norm{U^{-1}} \leq \sqrt{n} \exp \left( 6 \sqrt{n/T} \right).
    \end{equation}
\end{lemma}

\begin{proof}
    Let $h := \sqrt{T/n}$.
    We see $U$ as the Vandermonde matrix for the nodes $z_j := e^{j^2 T / n} = e^{y_j^2}$ where $y_j := j h$ for $1 \leq j \leq n$.
    By a classical estimate by Gautschi (see \cref{lem:gautschi}),
    \begin{equation}
        \label{eq:Qj-1}
        \norm{U^{-1}} \leq \sqrt{n} \max_{1 \leq j \leq n} e^{Q_j}
        \quad \text{where} \quad
        Q_j := \ln \prod_{k\neq j}\frac{1+|z_k|}{|z_j-z_k|}
        = \sum_{k\neq j} \ln \frac{1+e^{y_k^2}}{|e^{y_j^2}-e^{y_k^2}|}.
    \end{equation}
    We claim that
    \begin{equation}
        \label{eq:Qj-2}
        Q_j = \sum_{k \neq j} H_{y_j}(y_k) \leq \frac 1 h \int_h^{n h} H_{y_j}(y) \dd y
        \quad \text{where} \quad
        H_x(y) := \ln \frac{1+e^{y^2}}{|e^{x^2}-e^{y^2}|}.
    \end{equation}
    Indeed, for a fixed $x \geq 0$, for $y \in [0,x)$ and $y \in (x,\infty)$, explicitly differentiating $H_x$ yields
    \begin{equation}
        H_x'(y) = \frac{2 y e^{y^2} (1+e^{x^2})}{(1+e^{y^2})(e^{x^2}-e^{y^2})}.
    \end{equation}
    Thus $H_x$ is increasing on $[0,x)$ and decreasing on $(x,\infty)$.
    Hence, for each $k<j$ we have $h H_{y_j}(y_k) \leq \int_{y_k}^{y_{k+1}} H_{y_j}$, and for each $k>j$, $h H_{y_j}(y_k) \leq \int_{y_{k-1}}^{y_k} H_{y_j}$.
    Summing over $k \neq j$ yields \eqref{eq:Qj-2}.

    For $x, y \geq 0$, using $|x^2-y^2| \geq (x-y)^2$, we have
    \begin{equation}
        H_x(y) = \ln \frac{1+e^{-y^2}}{|1-e^{x^2-y^2}|}
        \leq \ln \frac{1+e^{-y^2}}{1-e^{-|x^2-y^2|}}
        \leq \ln \frac{1+e^{-y^2}}{1-e^{-(x-y)^2}}.
    \end{equation}
    Thus
    \begin{equation}
        \int_h^{n h} H_x(y) \leq \int_h^{n h} \ln \frac{1+e^{-y^2}}{1-e^{-(x-y)^2}} \dd y
        \leq \int_0^{\infty} \ln (1+e^{-y^2}) \dd y + \int_{\R} -\ln (1-e^{-z^2}) \dd z < 6.
    \end{equation}
    Gathering with \eqref{eq:Qj-1} and \eqref{eq:Qj-2} concludes the proof.
\end{proof}

\subsection{Small-time local controllability of control-affine systems}
\label{sec:affine}

We consider the time-invariant control-affine system on $\R^n$:
\begin{equation}
    \label{eq:affine}
    \dot{y}(t) = f_0(y(t)) + u_1(t) f_1(y(t)) + \dotsb + u_m(t) f_m(y(t)),
\end{equation}
where $y(t) \in \R^n$ is the state, $u(t) \in \R^m$ the control, and $f_0, f_1, \dotsc, f_m \in C^1(\R^n;\R^n)$ vector fields.

We assume that $f_0(0) = 0$, so that $0$ is an equilibrium of the free system (with null control).
We are interested in the local controllability at this equilibrium.
Since this is a local property, it only depends on the values of $f_0, f_1, \dotsc, f_m$ in a neighborhood of the origin.
Hence, we can assume without loss of generality that they are globally Lipschitz.
This will lighten the proofs.
In particular, this prevents finite-time blow-up in the nonlinear ODE \eqref{eq:affine}.

\begin{lemma}
    \label{lem:WP}
    Let $f_0, f_1, \dotsc, f_m \in C^1(\R^n;\R^n)$, globally Lipschitz, with $f_0(0) = 0$.
    For any $T > 0$, $u \in L^p((0,T);\R^m)$ and $y_0 \in \R^n$, there exists a unique mild solution $y \in C^0([0,T];\R^n)$ to \eqref{eq:affine} with initial data $y(0) = y_0$.
    When needed, we will refer to this solution as $y(t;u,y_0)$.
\end{lemma}

\begin{proof}
    Let $T > 0$, $u \in L^p((0,T);\R^m)$ and $y_0 \in \R^n$.
    We say that $y \in C^0([0,T];\R^n)$ is a \emph{mild solution} to \eqref{eq:affine} with initial data $y_0$ when it satisfies the integral formulation:
    \begin{equation}
        \label{eq:mild}
        \forall t \in [0,T], \quad y(t) = y_0 + \int_0^t (f_0 + u_1(s) f_1 + \dotsb + u_m(s) f_m)(y(s)) \dd s,
    \end{equation}
    which makes sense for $u \in L^p((0,T);\R^m)$.
    Let us prove that such a solution exists.

    Define a map $\Theta$ from the Banach space $C^0([0,T];\R^n)$ to itself by
    \begin{equation}
        (\Theta(y))(t) := y_0 + \int_0^t (f_0 + u_1(s) f_1 + \dotsb + u_m(s) f_m)(y(s)) \dd s.
    \end{equation}
    Let $L_0, L_1, \dotsc, L_m$ be the global Lipschitz constants of $f_0, f_1, \dotsc, f_m$ and $L := \max L_i$.
    One proves by induction on $k \in \N$ that, for all $y, \tilde{y} \in C^0([0,T];\R^n)$ and $t \in [0,T]$,
    \begin{equation}
        | (\Theta^k(y) - \Theta^k(\tilde{y}))(t) | \leq \frac{L^k (t + \int_0^t |u|)^k}{k!} \norm{y - \tilde{y}}_{C^0}.
    \end{equation}
    Since $T$ and $u$ are fixed, for $k$ large enough, $\Theta^k$ is a contraction and admits a unique fixed-point $y^\star$ by the Banach fixed-point theorem.
    Then $\Theta^k(\Theta(y^\star)) = \Theta(\Theta^k(y^\star)) = \Theta(y^\star)$, so $\Theta(y^\star)$ is a fixed point of $\Theta^k$, so is equal to $y^\star$.
    So $y^\star$ satisfies \eqref{eq:mild}.

    Uniqueness follows immediately.
    If $y$ and $\tilde{y}$ are two solutions, then $y = \Theta^k(y)$ and $\tilde{y} = \Theta^k(\tilde{y})$ which implies that $y = \tilde{y}$ since $\Theta^k$ is a contraction.
\end{proof}

\begin{remark}
    The proof given above relying on a Banach fixed-point argument in the Picard iterates is the same as the one of the Cauchy--Lipschitz / Picard--Lindelöf theorem.
    Given $T > 0$ and $u \in L^p((0,T);\R^m)$, one could have wanted to apply this abstract theorem to the time-dependent vector field $f(t,y) := f_0(y) + u_1(t) f_1(y) + \dotsb + u_m(t) f_m(y)$.
    However, since $u$ is not assumed to be continuous, $f$ is not continuous in time, which is usually assumed in these theorems.
    The role of this assumption is to obtain classical solutions in $C^1([0,T];\R^n)$, not only mild solutions as here.
\end{remark}

There are many (non-equivalent) local controllability notions in the literature, with subtle differences\footnote{See \cite[Table 2]{BoscainCannarsaFranceschiSigalotti2021}, \cite[Section 1.2, Appendix A.1]{BeauchardMarbach2022} or \cite[Section 8.2]{BeauchardMarbach2018} for some examples of definitions with variations on what exactly needs to be small or local and how.}.
We use the following definition.

\begin{definition}[$L^p$-STLC]
    \label{def:STLC}
    We say that \eqref{eq:affine} is \emph{small-time locally controllable} when, for every $T, \eta > 0$, there exists $\delta > 0$ such that, for every initial and final data $y_0, y_T \in B(0,\delta)$, there exists $u \in L^p((0,T);\R^m)$ with $\norm{u}_{L^p} \leq \eta$ such that $y(T;u,y_0) = y_T$.
\end{definition}

\subsection{The linear test for control-affine systems}
\label{sec:linear-test-ODEs}

Since small-time local controllability is a local notion, the first natural approach is to attempt to linearize the problem, at the equilibrium state $0$.
Given $\varphi \in C^1(\R;\R)$ with $\varphi(0) = 0$ and $\varphi'(0) \neq 0$, it is well-known that $\varphi$ is invertible near $0$, thanks to the inverse function theorem.
The \emph{linear test} is an instance of this result in control theory.
We here state it for systems governed by ODEs, but we will state a version valid for PDEs in \cref{sec:lin-2-nonlin-exact}.

Given $f_0, f_1, \dotsc, f_m \in C^1(\R^n;\R^n)$ globally Lipschitz with $f_0(0) = 0$, define
\begin{equation}
    \label{eq:AB-linearized}
    A := D f_0(0) \in \R^{n \times n}
    \quad \text{and} \quad
    B := [f_1(0), \dotsc, f_m(0)] \in \R^{n \times m}.
\end{equation}
We call $\dot{z} = A z + B v$ the \emph{linearized system at $0$} of the control-affine system \eqref{eq:affine}.

\begin{lemma}
    \label{lem:ODE-C1}
    Let $T > 0$.
    The solution map $\cS : \R^n \times L^p((0,T);\R^m) \to C^0([0,T];\R^n)$ defined by $\cS(y_0,u) := y(\cdot;u,y_0)$ is $C^1$ in a neighborhood of $(0,0)$ and, for all $(z_0,v) \in \R^n \times L^p((0,T);\R^m)$, $D \cS (0,0) \cdot (z_0,v) = z$ where $z \in C^0([0,T];\R^n)$ is the solution to $\dot{z} = Az + Bv$ and $z(0) = z_0$.
\end{lemma}

\begin{proof}
    Fix $T > 0$.
    Define $\cX := \R^n \times L^p((0,T);\R^m)$, $\cY := C^0([0,T];\R^n)$ and $G : \cX \times \cY \to \cY$ as
    \begin{equation}
        G((y_0,u),y) := t \mapsto y(t) - y_0 - \int_0^t (f_0 +u_1(s) f_1 + \dotsb + u_m(s) f_m)(y(s)) \dd s.
    \end{equation}
    Thus, $G((y_0,u),y) = 0$ iff $y = \cS(y_0,u)$.
    Let $\mathbf{0} := ((0,0),0)$.
    Since $f_0(0) = 0$, $G(\mathbf{0}) = 0$.
    Since the vector fields are $C^1$, $G$ is $C^1$ too.
    We compute its partial derivatives at $\mathbf{0}$.
    One has
    \begin{equation}
        \label{eq:D12G}
        D_1 G(\mathbf{0}) \cdot (z_0,v) = t \mapsto - z_0 - B \int_0^t v
        \quad \text{and} \quad
        D_2 G(\mathbf{0}) \cdot z = t \mapsto z(t) - A \int_0^t z.
    \end{equation}
    One checks that $D_2 G(\mathbf{0}) \cdot z = h$ iff $z(t) = h(t) + A \int_0^t e^{(t-s)A} h(s) \dd s$, so $D_2 G(\mathbf{0})$ is a bounded linear isomorphism of $\cY$.
    By the implicit function theorem, there exists a neighborhood $\Omega$ of $(0,0)$ in $\cX$ and a $C^1$ map $\cS : \Omega \to \cY$ such that $G((y_0,u),\cS(y_0,u)) = 0$ for all $(y_0,u) \in \Omega$.
    By the chain rule, one has $D \cS(0,0) = - (D_2 G(\mathbf{0}))^{-1} (D_1 G(\mathbf{0}))$.
    Using \eqref{eq:D12G}, one finds that $D \cS(0) \cdot (z_0,v) = z$ where $z(t) = e^{tA} z_0 + \int_0^t e^{(t-s)A} B v(s) \dd s$ (so the unique solution to $\dot{z} = Az + Bv$ with $z(0) = z_0$).
\end{proof}

As noted in the rank condition of \cref{thm:kalman}, the controllability of the linearized system is independent of $T$.
Thus, one obtains the classical result, dating back at least to \cite[Theorem~3]{Markus1965}.

\begin{theorem}[Linear test]
    \label{thm:ODE-linear}
    Assume that linearized system $\dot{z} = A z + B v$ with $A, B$ as in \eqref{eq:AB-linearized} is controllable.
    Then the control-affine system \eqref{eq:affine} is $L^p$-STLC in the sense of \cref{def:STLC}.
\end{theorem}

\begin{proof}
    Let $T, \eta > 0$.
    Since the linearized system $\dot{z} = Az + Bv$ is controllable, there exist $n$ controls $v^1, \dotsc, v^n \in L^p((0,T);\R^m)$ such that $\Phi_T(v^j) = \int_0^T e^{(T-t)A} B v^j(t) \dd t = e_j$ for $1 \leq j \leq n$.

    For $\lambda \in \R^n$, let $u_\lambda := \lambda_1 v^1 + \dotsb + \lambda_n v^n$.
    Define $\Psi : \R^{n+n} \times \R^n \to \R^n$ by
    \begin{equation}
        \Psi( (y_0,y_T), \lambda) := y(T;u_\lambda, y_0) - y_T.
    \end{equation}
    By \cref{lem:ODE-C1}, $\Psi$ is $C^1$ and $D_\lambda \Psi((0,0),0) = \Id_{\R^n}$ thanks to our choice of the $v^j$.
    By the $C^1$ implicit function theorem, there exists $\delta > 0$ and a $C^1$ map $\Lambda : B(0,\delta) \times B(0,\delta) \to \R^n$ with $\Lambda(0,0) = 0$ such that, for all $y_0, y_T \in B(0,\delta)$, $\Psi((y_0,y_T),\Lambda(y_0,y_T)) = 0$.
    Since $\Lambda$ is $C^1$, there exists $C_\Lambda > 0$ such that $\norm{u_{\Lambda(y_0,y_T)}}_{L^p} \leq C_\Lambda (|y_0| + |y_T|)$.
    Thus, up to reducing $\delta$, one has $\norm{u_{\Lambda(y_0,y_T)}} \leq \eta$ and $L^p$-STLC in the sense of \cref{def:STLC}.
\end{proof}

\begin{example}
    Let $n = 3$ and $m = 1$.
    Consider the system
    \begin{equation}
        \begin{cases}
            \dot{y}_1 = u \\
            \dot{y}_2 = y_1 + y_3^2 \\
            \dot{y}_3 = y_2 + y_1^3.
        \end{cases}
    \end{equation}
    It is of the form \eqref{eq:affine} with $f_0(y) = (0, y_1+y_3^2, y_2+y_1^3)$ and $f_1(y) = (1,0,0)$.
    Thus $A = Df_0(0)$ and $B = f_1(0)$ are of the integrator chain form given in \eqref{ex:chain}.
    So the linearized system is controllable by the rank condition, and the nonlinear one is $L^p$-STLC by the linear test.
\end{example}

\subsection{Nonlinear motions and obstructions}
\label{sec:ode-nonlinear}

When the linearized system at an equilibrium is not controllable, one must study higher order-terms to determine the controllability of the nonlinear system.

The following example is one where controllability can be obtained using a nonlinear term, using an entirely elementary argument.
The quest for increasingly general sufficient conditions for STLC of ODEs has a long history; see Hermes \cite[Theorem~3.2]{Hermes1982}, Sussmann \cite[Theorem~7.3]{Sussmann1987}, Agrachev and Gamkrelidze \cite[Theorem~4]{AgrachevGamkrelidze1993_Semigroups} or Krastanov \cite[Theorem 2.7]{Krastanov2009}.

\begin{example}
    \label{ex:cubic}
    Let $n = 2$ and $m = 1$.
    Consider the system:
    \begin{equation}
        \label{eq:ex-cubic}
        \begin{cases}
            \dot{y}_1 = u \\
            \dot{y}_2 = y_1^3.
        \end{cases}
    \end{equation}
    The linearized system is $\dot{y}_1 = u$ and $\dot{y}_2 = 0$, so is not controllable.
    Nevertheless, one can prove that~\eqref{eq:ex-cubic} is $L^p$-STLC.
    Fix $T > 0$ and $\varphi \in C^1_c((0,T);\R)$ such that $\int_0^T \varphi^3 = 1$.
    For $y_0, y_T \in \R^2$, consider $u(t) := (y_{T,1}-y_{0,1})/T + \mu \varphi'(t)$ for $\mu \in \R$ to be chosen later.
    One has $y_1(T) = y_{T,1}$ and $y_2(T)$ is a cubic polynomial of $\mu$ with leading coefficient $1$.
    Hence, by the intermediate value theorem, there exists $\mu$ such that $y_2(T) = y_{T,2}$.
\end{example}

As shown by the next example, the opposite situation can also occur, where higher-order terms prevent the controllability of the nonlinear system.
Although the following example is caricatural, many more general obstruction results are known; see Sussmann \cite[Proposition 6.3]{Sussmann1983}, Stefani \cite[Theorem 1]{Stefani1986}, Kawski \cite[Theorem~1]{Kawski1987_Necessary} or Beauchard and Marbach \cite{BeauchardMarbach2018,BeauchardMarbach2022}.

\begin{example}
    \label{ex:obs-W1}
    Let $n = 2$ and $m = 1$.
    Consider the system:
    \begin{equation}
        \label{eq:ex-quad}
        \begin{cases}
            \dot{y}_1 = u \\
            \dot{y}_2 = y_1^2.
        \end{cases}
    \end{equation}
    The linearized system is the same as the one of \eqref{eq:ex-cubic}.
    It is clear that the nonlinear system \eqref{eq:ex-quad} is not $L^p$-STLC since $\dot{y}_2 \geq 0$ (hence, starting from the initial state $y_0 = 0$, no final state of the form $- \delta e_2$ is reachable).
\end{example}

As can be expected, most difficulties arise when a system involves a competition between a nonlinear term allowing controllability such as $y_1^3$ in \eqref{eq:ex-cubic} and one yielding an obstruction such as $y_1^2$ in \eqref{eq:ex-quad}.
Such competitions are at the heart of the literature on sufficient or necessary STLC conditions.
We study one such example in \cref{sec:jakub} using the theory of tangent vectors.

\newpage
\section{From linear ODEs to linear dissipative PDEs}
\label{sec:LR}

In this section, we now consider control systems whose state is infinite dimensional.
Assuming that some projections of the state are controllable, and that the system is dissipative, we use a time-iteration argument to deduce that the infinite-dimensional state is small-time null controllable.

In contrast with the ODE theory of \cref{sec:intro} or with non-dissipative PDEs such as transport, wave or Schrödinger equations, in this section, we only consider the \emph{null controllability} problem, that is where the target state is $0$.
See \cref{sec:why-null} for more comments on this topic.

Heuristically, one can think that we consider systems of the form:
\begin{equation}
    \label{eq:yt=Ay+Bu}
    \dot{y} = A y + B u
\end{equation}
where $A$ and $B$ are operators generalizing the matrices of \eqref{eq:lti_system} and we aim for $y(T) = 0$.

\subsection{Input-output formalism}
\label{sec:linear-io}

To avoid technicalities linked with the theory of operator semigroups and admissibility of control operators, we will present the time-iteration argument using an input-output formalism, without referring to the underlying evolution equation \eqref{eq:yt=Ay+Bu}.

Let $Y$ be a Banach space (the \emph{state space}) describing the possible values for $y(t) \in Y$ and $U$ be a Banach space (the \emph{control space}) describing the possible values for $u(t) \in U$.
Fix $p \in [1,\infty]$.
The following definitions are adapted\footnote{
\cref{def:linear-io} slightly differs from \cite{Weiss1989}.
See \cref{sec:weiss} for a short discussion.} from \cite{Weiss1989}.

\begin{definition}
    For $u,v \in L^p(\R_+;U)$, let $u \underset{\tau}{\diamond} v \in L^p(\R_+;U)$ denote their $\tau$-concatenation:
    \begin{equation}
        (u \underset{\tau}{\diamond} v)(s) :=
        \begin{cases}
            u(s) & \text{for } s \in [0,\tau), \\
            v(s-\tau) & \text{for } s \geq \tau.
        \end{cases}
    \end{equation}
\end{definition}

\begin{definition}
    \label{def:linear-io}
    We say that $\Sigma$ is an \emph{abstract linear control system} when it is a family $(\Sigma_t)_{t \geq 0}$ of continuous linear maps from $Y \times L^p(\R_+;U)$ to $Y$ such that:
    \begin{itemize}
        \item for all $y_0 \in Y$ and $u \in L^p(\R_+;U)$, $t \mapsto \Sigma_t(y_0,u)$ is continuous on $\R_+$ with $\Sigma_0(y_0,u) = y_0$,
        \item for all $t, \tau \geq 0$, $y_0 \in Y$ and $u,v \in L^p(\R_+;U)$,
        \begin{equation}
            \label{eq:Sigma-compo}
            \Sigma_{t+\tau}(y_0,u \underset{\tau}{\diamond} v) = \Sigma_t(\Sigma_\tau(y_0,u), v).
        \end{equation}
    \end{itemize}
\end{definition}

Therefore, one should interpret $\Sigma_t(y_0,u)$ as $y(t) \in Y$, the state at time $t \geq 0$ of the solution to the considered controlled PDE with initial state $y_0 \in Y$ and control $u \in L^p(\R_+;U)$.
The requirements of \cref{def:linear-io} correspond to asking that $y(t)$ depends continuously on $t$, $y_0$ and $u$, along with the \emph{composition property} \eqref{eq:Sigma-compo}, which reflects that one obtains the same final state by using some intermediate state as initial data and propagating for the remaining time.

One has the following natural estimate.

\begin{lemma}
    \label{lem:ubp}
    Let $T > 0$.
    There exists $M_T > 0$ such that, for all $t \in [0,T]$, $\norm{\Sigma_t} \leq M_T$.
\end{lemma}

\begin{proof}
    The family $(\Sigma_t)_{t \in [0,T]}$ is a family of bounded linear maps from $Y \times L^p(\R_+;U)$ to $Y$.
    For each fixed $(y_0,u) \in Y \times L^p(\R_+;U)$, $t \mapsto \Sigma_t(y_0,u)$ is continuous from $[0,T]$ to $Y$ by the first item of \cref{def:linear-io}, so $\sup_{t \in [0,T]} \norm{\Sigma_t(y_0,u)}_Y$ is bounded.
    Hence the conclusion follows from the uniform boundedness principle (Banach--Steinhaus theorem).
\end{proof}

\begin{remark}
    For pedagogical purposes, we consider here time-invariant systems.
    However, it is easy to adapt \cref{def:linear-io} to more general families of propagators, say $\Sigma_{t,s}$ for $t \geq s$.
\end{remark}

\subsection{A general time iteration}
\label{sec:LR-iteration}

In 1995, Lebeau and Robbiano introduced in \cite{LebeauRobbiano1995} a time-iteration argument which establishes the small-time null controllability of a dissipative PDE from the balance between the cost of the small-time null controllability of its projections (see \eqref{eq:ass.approx}) and its dissipation rate (see \eqref{eq:ass.dissip}).

We present here an abstract version of their argument, directly applicable as a black box, which encompasses recent progress made by other authors (see \cref{sec:LR-biblio}).
A pedagogical way to state the assumptions of this time iteration is as follows (see below for enhancements).

\begin{assumption}
    \label{ass:proj+diss}
    There exists a family $(P_k)_{k \in \N}$ of projections on $Y$ and there exist constants $C_1, C_2, c_1, c_2, a, b, m > 0$ with $a < b$ such that, for all $k \in \N$,
    \begin{itemize}
        \item \textbf{(control of low modes)} for all $T \in (0,1]$ and $y_0 \in Y$, there exists $u \in L^p(\R_+;U)$ such that
        \begin{equation}
            \label{eq:ass.approx}
            P_k \Sigma_T(y_0,u) = 0
            \quad \text{and} \quad
            \norm{u}_{L^p} \leq C_1 e^{c_1 k^a} \norm{y_0}_Y,
        \end{equation}

        \item \textbf{(dissipation of high modes)} for all $T \in (0,1]$ and $y_0 \in Y$,
        \begin{equation}
            \label{eq:ass.dissip}
            P_k y_0 = 0
            \quad \Longrightarrow \quad
            \norm{\Sigma_T(y_0,0)}_Y \leq C_2 e^{- c_2 T^m k^b} \norm{y_0}_Y.
        \end{equation}
    \end{itemize}
\end{assumption}

During the main time iteration, within each time step of duration $\tau$, we choose a threshold $k \in \N$, control the low modes to zero on $[0,\tau/2]$, then let the energy dissipate on $[\tau/2,\tau]$.
More precisely, the main time iteration relies on the following one-step contraction result.

\begin{lemma}
    \label{lem:LR-contract}
    If \cref{ass:proj+diss} holds, then there exist $C_3,c_3 > 0$ such that, for all $k \in \N$, for all $\tau \in (0,1]$ and $y_0 \in Y$, there exists $u \in L^p(\R_+;U)$ such that
    \begin{align}
        \label{eq:ass.contract-u}
        \norm{u}_{L^p} & \leq C_1 e^{c_1 k^a} \norm{y_0}_Y, \\
        \label{eq:ass.contract-y}
        \norm{\Sigma_\tau(y_0,u)}_Y & \leq C_3 e^{c_1 k^a} e^{- c_3 \tau^m k^b} \norm{y_0}_Y.
    \end{align}
\end{lemma}

\begin{proof}
    Let $C_1, C_2, c_1, c_2, a, b, m > 0$ with $a < b$ be given by \cref{ass:proj+diss}.
    Let $\tau > 0$, $k \in \N$ and $y_0 \in Y$.
    Define a control $u$ on $[0,\tau]$ as being given by the first item of \cref{ass:proj+diss} for a time~$\frac \tau 2$, extended by $0$ on $[\frac \tau 2, \tau]$.
    Let $y_1 := \Sigma_{\frac \tau 2}(y_0,u)$.
    By \cref{lem:ubp} and \eqref{eq:ass.approx},
    \begin{equation}
        \label{eq:2ass-1}
        \norm{y_1}_Y
        \leq M_1 \left(\norm{y_0}_Y + \norm{u}_{L^p}\right)
        \leq M_1 (1 + C_1 e^{c_1 k^a}) \norm{y_0}_Y.
    \end{equation}
    Since $P_k y_1 = 0$, \eqref{eq:ass.dissip} implies that
    \begin{equation}
        \label{eq:2ass-2}
        \norm{\Sigma_\tau(y_0,u)}_Y
        = \norm{\Sigma_{\frac \tau 2}(y_1,0)}
        \leq C_2 e^{-c_2 (\tau -\frac \tau 2)^m k^b} \norm{y_1}_Y
    \end{equation}
    We deduce the claimed estimates, with $c_3 = c_2 2^{-m}$ and $C_3 = M_1 C_2 (1 + C_1)$.
\end{proof}

\begin{theorem}
    \label{thm:LR}
    Under \cref{ass:proj+diss}, the abstract linear system $\Sigma$ is small-time null-controllable.
    More precisely, there exist constants $C_0, c_0 > 0$ such that, for every $T \in (0,1]$ and $y_0 \in Y$, there exists $u \in L^p((0,T);U)$ such that $\Sigma_T(y_0,u) = 0$ and
    \begin{equation}
        \label{eq:LR-cost}
        \norm{u}_{L^p} \leq C_0 \exp \left( c_0 T^{-\sigma} \right) \norm{y_0}_Y
        \quad \text{where} \quad
        \sigma := \frac{am}{b-a}.
    \end{equation}
\end{theorem}

\begin{proof}
    \step{Construction of the time and frequency sequences}
    We define a sequence of time intervals and frequency cutoffs.
    Let $r>1$ and $q>1$ be two real numbers to be chosen later.
    We decompose the interval $[0,T]$ into a series of subintervals.
    Let
    \begin{equation}
        \label{eq:LR-tau_j}
        \alpha := T \left(1- \frac 1 r\right) \quad \text{and} \quad \tau_j := \frac{\alpha}{r^j} \quad \text{for } j \in \N.
    \end{equation}
    Hence $T = \sum_{j=0}^\infty \tau_j$.
    We define the sequence of time instants $T_0=0$ and $T_{j+1} = T_j + \tau_j$.
    The $j$-th step of our procedure will happen on the interval $[T_j, T_{j+1}]$.

    Let $\beta > 0$ be another parameter to be chosen.
    We define the sequence of frequency cutoffs as
    \begin{equation}
        \label{eq:LR-k_j}
        k_j := \lfloor \beta q^j \rfloor \quad \text{for } j \in \N.
    \end{equation}

    \step{The iterative scheme}
    We define a sequence of states $(y_j)_{j \in \N}$ by $y_0$ the given initial state and $y_{j+1} := \Sigma_{\tau_j}(y_j, u_j)$, where $u_j$ is the control used on $[T_j, T_{j+1}]$, given by \cref{lem:LR-contract}.
    Hence, by \eqref{eq:ass.contract-u} and \eqref{eq:ass.contract-y},
    \begin{align}
        \label{eq:vj_bound}
        \norm{u_j}_{L^p} & \leq C_1 e^{c_1 k_j^a} \norm{y_j}_Y, \\
        \label{eq:recurrence}
        \norm{y_{j+1}}_Y & \leq C_3 e^{- \gamma_j} \norm{y_j}_Y,
    \end{align}
    where we introduce the notation
    \begin{equation}
        \gamma_j := c_3 \tau_j^m k_j^b - c_1 k_j^a.
    \end{equation}

    \step{Choice of the parameters $r, q, \beta$}
    Our goal is to ensure that $\norm{y_j}_Y \to 0$ sufficiently fast.
    For $\beta \geq 1$, by \eqref{eq:LR-k_j}, $\frac 12 \beta q^j \leq k_j \leq \beta q^j$.
    Recalling \eqref{eq:LR-tau_j} for $\tau_j$, we obtain
    \begin{equation}
        \label{eq:gamma_j-low}
        \gamma_j \geq (q^a)^j \left( 2^{-b} c_3 \alpha^m \beta^b \left( q^{b-a} r^{-m} \right)^j - c_1 \beta^a \right).
    \end{equation}
    We let $q := 2^{\frac 1 a}$ so that $q^a = 2$ and $r := q^{\frac{b-a}{m}}$ so that $q^{b-a} r^{-m} = 1$.
    Also, we tune $\beta$ such that $2^{-b} c_3 \alpha^m \beta^b = 3 c_1 \beta^a$.
    This corresponds to setting
    \begin{equation}
        \label{eq:LR-beta}
        \beta := \left(\frac{2^b 3 c_1}{c_3 \alpha^m} \right)^{\frac 1 {b-a}}.
    \end{equation}
    With these choices $\beta \to +\infty$ as $T \to 0$ and one has:
    \begin{equation}
        \gamma_j \geq 2 c_1 \beta^a 2^j.
    \end{equation}

    \step{Stepwise convergence of the state}
    Iterating \eqref{eq:recurrence}, we obtain, for any $j \geq 1$,
    \begin{equation}
        \label{eq:yj-geom}
        \norm{y_j}_Y \leq \exp \left( j \ln C_3 - 2 c_1 \beta^a (2^j-1) \right) \norm{y_0}_Y.
    \end{equation}
    Therefore, $\norm{y_j}_Y \to 0$ as $j \to \infty$.

    \step{Bounding the total control cost}
    The overall control $u$ is the concatenation of the controls $u_j$ on $[T_j, T_{j+1}]$.
    Using the bounds \eqref{eq:vj_bound} and \eqref{eq:yj-geom}, we have:
    \begin{equation}
        \begin{split}
            \norm{u_j}_{L^p}
            & \leq C_1 e^{c_1 k_j^a} \norm{y_j}_Y \\
            & \leq C_1 \exp (c_1 \beta^a 2^j) \exp (j \ln C_3 - 2 c_1 \beta^a (2^j-1)) \norm{y_0}_Y \\
            & \leq C_1 e^{c_1 \beta^a} \exp \left(j (\ln C_3 - c_1 \beta^a) \right) \norm{y_0}_Y \\
            & \leq C_1 e^{c_1 \beta^a} 2^{-j} \norm{y_0}_Y
        \end{split}
    \end{equation}
    using $2^j-1 \geq j$ and $T$ small enough so that $\ln C_3 - c_1 \beta^a \leq - \ln 2$.

    Thus, there exists $T_0 > 0$ such that, for all $0 < T \leq T_0$,
    \begin{equation}
        \norm{u}_{L^p((0,T);U)}
        \leq \sum_{j=0}^{\infty} \norm{u_j}_{L^p((0,\tau_j);U)}
        \leq 2 C_1 e^{c_1 \beta^a} \norm{y_0}_Y.
    \end{equation}
    Recalling \eqref{eq:LR-tau_j} and \eqref{eq:LR-beta}, this entails the cost estimate \eqref{eq:LR-cost} since $\beta^a \approx (\alpha^{-\frac{m}{b-a}})^a \approx T^{-\sigma}$.

    \step{Conclusion}
    Since $\Sigma$ is an abstract linear control system, the associated solution $y(t) := \Sigma_t(y_0,u)$ belongs to $C^0([0,T];Y)$.
    By the composition property \eqref{eq:Sigma-compo}, $y(T_j) = y_j$.
    We proved that $y_j \to 0$.
    Since $T_j \to T$, by continuity, we conclude that $\Sigma_T(y_0,u) = y(T) = 0$.
\end{proof}

\begin{remark}[Cost optimization]
    In some contexts, one might be interested in optimizing the constant $c_0$ appearing inside the exponential of the cost estimate \eqref{eq:LR-cost}.
    Many parameters of the time iteration can be further tuned for this purpose (see \cite[Section 2.3]{Miller2010} or \cite[Section 2.2.3]{Gallaun2022}).
\end{remark}

\begin{remark}[Time-dependent systems]
    \label{rk:LR-time-dep}
    The proof given above also applies to time-dependent systems, say $\Sigma_{t,s}$ for $t \geq s$, provided that the assumptions hold uniformly with respect to $t$ and $s$ (see e.g.~\cite[Theorem 3.2]{BeauchardEgidiPravdaStarov2020}).
    We present the time-invariant version to lighten the notations.
\end{remark}

\begin{remark}[Relaxed assumptions on constants]
    \label{rk:relax}
    The form of the exponential constants in the assumed estimates \eqref{eq:ass.approx} and \eqref{eq:ass.dissip} is given to fix ideas, and because they match some frequent use cases.
    However, the time iteration is robust to many variations of these estimates.
    For example, it is common to have polynomial prefactors of the form $T^{-h}$ in \eqref{eq:ass.approx} (see e.g.\ \cite[Prop.~IV.2.24]{Boyer2022} and \cref{sec:spectral}) or in \eqref{eq:ass.dissip} (see e.g.\ \cite[eq.~(3.2)]{BeauchardEgidiPravdaStarov2020}).

    In fact, the proof can accommodate for much larger modifications.
    Since the time schedule works with $T^m k^b \approx k^a$, the result still holds if the estimates \eqref{eq:ass.approx} and \eqref{eq:ass.dissip} are respectively replaced by
    \begin{align}
        \label{eq:ass-approx-u-relax}
        \norm{u}_{L^p} & \leq C_1 \exp\left(c_1 k^a + c_1' T^{-\theta} \right) \norm{y_0}_Y, \\
        \label{eq:ass-dissip-relax}
        \norm{\Sigma_T(y_0,0)}_{Y} & \leq C_2 \exp\left( -c_2 T^m k^b + c_2' k^a + c_2'' T^{-\theta}\right) \norm{y_0}_Y.
    \end{align}
    Then the cost estimate \eqref{eq:LR-cost} holds with $\sigma = \max \{ \frac{am}{b-a}, \theta \}$.
    In fact, up to worsening $b$ or $\theta$ in the assumptions \eqref{eq:ass-approx-u-relax} and \eqref{eq:ass-dissip-relax}, one can always assume that $\sigma = \frac{am}{b-a} = \theta$.

    One can carry out the same proof, starting with adapting \cref{lem:LR-contract}.
    In the time schedule, one can keep $q = 2^{\frac 1a}$ and $r = q^{\frac{b-a}{m}}$ and look for $\beta = \rho \alpha^{-\frac{m}{b-a}}$.
    The lower-bound \eqref{eq:gamma_j-low} is replaced with
    \begin{equation}
        \gamma_j \geq 2^j \beta^a \left( \frac{c_2}{2^{b+m}} \rho^{b-a} - (c_1+c_2') - 2 \theta (c_1' + c_2'') \rho^{-a} \right).
    \end{equation}
    Choosing $\rho$ large enough, $\gamma_j \geq 2 (c_1 + c_1' 2^{-\sigma} \rho^{-a}) 2^j \beta^a$, and the end of the proof is unchanged.
\end{remark}

\begin{remark}
    Estimate \eqref{eq:ass.approx} states that the cost of the control of the first $k$ modes does not blow up as $T \to 0$.
    This can seem to be in sharp contrast with our finite-dimensional estimates of \cref{sec:seidman}.
    To understand this apparent paradox, one must recall that we will obtain such estimates only in situations where the integer $K$ of \cref{p:seidman} is equal to $0$ (i.e.\ the control operator $B$ directly spans the full space); see e.g.\ \cref{sec:spectral-commutative}.
\end{remark}

\subsection{The role of spectral inequalities}
\label{sec:spectral}

The time-iteration argument presented in \cref{sec:LR-iteration} is independent of the way that one proves that its assumptions are satisfied.
However, starting with \cite{LebeauRobbiano1995}, most works using this method also rely on a so-called \emph{spectral inequality} to check that the assumptions are satisfied.
Some authors also say \emph{uncertainty principle}, since it is not required that the projections $P_k$ for \cref{ass:proj+diss} are orthogonal projections on spectral subspaces of the considered self-adjoint operator.

\medskip

In this section, we assume that $Y$ and $U$ are Hilbert spaces, and we consider a family $(P_k)_{k \in \N}$ of orthogonal projections on $Y$.
For $t \geq 0$, we also define $\T_t \in \cL(Y)$ by $\T_t y_0 := \Sigma_t(y_0,0)$ ($\T$ is the strongly continuous semigroup associated with our system; see \cref{sec:weiss}).
Finally, we assume that there exists an operator $B \in \cL(U, Y)$ such that, for all $u \in L^p(\R_+;U)$ and $T \geq 0$,
\begin{equation}
    \label{eq:Sigma_T=B}
    \Sigma_T(0,u) = \int_0^T \T_{T-t} B u(t) \dd t.
\end{equation}
The existence of an integral representation of $u \mapsto \Sigma_T(0,u)$ is always true (at least when $p < + \infty$), but $B$ is not always bounded from $U$ to $Y$ (typically not for boundary control); see \cite[Theorem 3.9, Problem 3.10]{Weiss1989}.
In such a case, there is a notion of admissibility of $B$ with respect to $A$ (see \cite[Section 4]{Weiss1989}) and some of the techniques presented here still hold (see \cite[Theorem 1.2]{TenenbaumTucsnak2011}).

\medskip

The constructions of this section are based on the following \emph{static} assumption (see below for examples), in which neither the time nor the operator $A$ (or the semigroup $\T_t$) play a role.

\begin{assumption}[Spectral inequality]
    \label{ass:spectral}
    There exist $C_1, c_1, a > 0$ such that, for all $k \in \N$,
    \begin{equation}
        \label{eq:BPk}
        \forall y \in Y, \quad
        \norm{P_k y}_Y \leq C_1 e^{c_1 k^a} \norm{B^\star P_k y}_U
    \end{equation}
    or, equivalently (by a consequence of the closed-graph theorem, see \cite[Proposition 12.1.2]{TucsnakWeiss2009}),
    \begin{equation}
        \exists D_k \in \cL(Y, U), \text{ such that } P_k = P_k B D_k \text{ and } \norm{D_k} \leq C_1 e^{c_1 k^a}.
    \end{equation}
\end{assumption}

\subsubsection{The commutative case}
\label{sec:spectral-commutative}

In this paragraph, assuming that the projections $P_k$ commute with the dynamics (i.e.\ all $t \geq 0$, $P_k \T_t = \T_t P_k$), we show that the spectral inequality entails the controllability of the low modes.

\begin{proposition}
    \label{p:dual-easy}
    Under \cref{ass:spectral}, if the projections $P_k$ commute with the dynamics, then, for all $k \in \N$, $T \in (0,1]$ and $y_0 \in Y$, there exists $u \in L^p((0,T);U)$ such that
    \begin{equation}
        P_k \Sigma_T(y_0,u) = 0
        \quad \text{and} \quad
        \norm{u}_{L^p} \leq T^{\frac 1 p - 1} M_1 C_1 e^{c_1 k^a} \norm{y_0}_Y
    \end{equation}
    where $M_1$ is the constant from \cref{lem:ubp}.
\end{proposition}

\begin{proof}
    Let $k \in \N$ and $D_k \in \cL(Y,U)$ be given by \cref{ass:spectral}.
    For $T \in (0,1]$ and $y_0 \in Y$, set
    \begin{equation}
        u(t) := - \frac{1}{T} D_k \T_t y_0.
    \end{equation}
    By \cref{lem:ubp}, one has
    \begin{equation}
        \norm{u}_{L^p} \leq T^{\frac 1 p - 1} \norm{D_k} \cdot \sup_{t \in [0,1]} \norm{\T_t} \cdot \norm{y_0}_Y
        \leq T^{\frac 1 p - 1} M_1 C_1 e^{c_1 k^a} \norm{y_0}_Y.
    \end{equation}
    By linearity and \eqref{eq:Sigma_T=B},
    \begin{equation}
        \label{eq:duhamel-BDK}
        \Sigma_T(y_0, u) = \T_T y_0 - \frac{1}{T} \int_0^T \T_{T-t} B D_k \T_t y_0 \dd t.
    \end{equation}
    By assumption, the projection $P_k$ commutes with the semigroup, so, in particular
    \begin{equation}
        \label{eq:PTP}
        \forall \tau \geq 0, \quad
        P_k \T_\tau P_k = P_k \T_\tau
    \end{equation}
    Applying $P_k$ to \eqref{eq:duhamel-BDK}, using $P_k B D_k = P_k$, the semigroup property and \eqref{eq:PTP},
    \begin{equation}
        \begin{split}
            P_k \Sigma_T(y_0, u)
            & = P_k \T_T y_0 - \frac{1}{T} P_k \int_0^T \T_{T-t} P_k B D_k \T_t y_0 \dd t \\
            & = P_k \T_T y_0 - \frac{1}{T} P_k \int_0^T \T_{T-t} P_k \T_t y_0 \dd t \\
            & = P_k \T_T y_0 - \frac{1}{T} P_k \int_0^T \T_{T-t} \T_t y_0 \dd t = 0.
        \end{split}
    \end{equation}
    So we have found a control of the expected size, bringing the projection $P_k$ of the state to $0$.
\end{proof}

\begin{remark}
    \label{rk:PKP}
    The above proof only uses the relation \eqref{eq:PTP}, which is equivalent to the condition that the range of $P_k^\bot = \Id - P_k$ is stable by the semigroup (and so slightly weaker than requiring that $P_k$ commutes with the semigroup; one must only prevent ``high-to-low'' energy transfers).
\end{remark}

\subsubsection{The non-commutative case}
\label{sec:spectral-no-commute}

When the projections $P_k$ don't commute with the dynamics, the construction of \cref{p:dual-easy} no longer yields a null-control for the low modes.
However, the lack of commutation can be compensated by the dissipation to derive the one-step contraction result of \cref{lem:LR-contract}, thus allowing to execute the time iteration of \cref{thm:LR}.

\begin{proposition}
    \label{p:dual-hard}
    Assume that there exist $C_1, C_2, c_1, c_2, a, b, m > 0$ with $a < b$ such that, for all $k \in \N$, the spectral \cref{ass:spectral} and the dissipation condition \eqref{eq:ass.dissip} hold.
    Then there exist $C_3,c_3 > 0$ such that, for all $k \in \N$, for all $\tau \in (0,1]$ and $y_0 \in Y$, there exists $u \in L^p(\R_+;U)$ with
    \begin{align}
        \label{eq:ass.contract-u-bis}
        \norm{u}_{L^p} & \leq 2 \tau^{\frac 1 p - 1} M_1 C_1 e^{c_1 k^a} \norm{y_0}_Y, \\
        \label{eq:ass.contract-y-bis}
        \norm{\Sigma_\tau(y_0,u)}_Y & \leq C_3 e^{c_1 k^a} e^{- c_3 \tau^m k^b} \norm{y_0}_Y,
    \end{align}
    where $M_1$ is the constant from \cref{lem:ubp}.
\end{proposition}

\begin{proof}
    Let $k \in \N$ and $D_k \in \cL(Y,U)$ given by \cref{ass:spectral}.
    For $\tau \in (0,1]$ and $y_0 \in Y$, we use the same control as in the proof of \cref{p:dual-easy} on $(0,\frac \tau 2)$, then a null control on $(\frac \tau 2, \tau)$:
    \begin{equation}
        u(t) := - \frac{2}{\tau} \mathbf{1}_{(0,\frac \tau 2)}(t) D_k \T_t y_0 .
    \end{equation}
    By \cref{lem:ubp}, we obtain \eqref{eq:ass.contract-u-bis}.
    By linearity and \eqref{eq:Sigma_T=B},
    \begin{equation}
        \Sigma_\tau(y_0, u) = \T_\tau y_0 - \frac{2}{\tau} \int_0^{\frac \tau 2} \T_{\tau-t} B D_k \T_t y_0 \dd t.
    \end{equation}
    We now write the following decomposition:
    \begin{equation}
        B D_k
        = \Id + (B D_k - \Id)
        = \Id + P_k (B D_k - \Id) + (\Id - P_k) (B D_k - \Id).
    \end{equation}
    By construction $P_k (B D_k - \Id) = 0$.
    By linearity and the semigroup property,
    \begin{equation}
        \Sigma_\tau(y_0, u) = - \frac 2 \tau \int_0^{\frac \tau 2} \T_{\tau - t} (\Id - P_k) (B D_k - \Id) \T_t y_0 \dd t.
    \end{equation}
    Hence, by the dissipation condition \eqref{eq:ass.dissip},
    \begin{equation}
        \begin{split}
            \norm{\Sigma_\tau(y_0,u)}_Y & \leq \frac 2 \tau \int_0^{\frac \tau 2} \norm{\T_{\tau - t} (\Id - P_k)} (\norm{B} \norm{D_k} + 1) \norm{\T_t} \norm{y_0}_Y \dd t \\
            & \leq C_2 e^{-c_2 (\tau/2)^m k^b} (\norm{B} C_1 e^{c_1 k^a} + 1) M_1 \norm{y_0}_Y.
        \end{split}
    \end{equation}
    Thus we obtain \eqref{eq:ass.contract-y-bis} with $C_3 = C_2 (C_1 \norm{B} + 1) M_1$ and $c_3 = c_2 2^{-m}$.
\end{proof}

\subsection{Historical perspective}
\label{sec:LR-biblio}

We briefly recall some contributions regarding the method exposed above, introduced in \cite{LebeauRobbiano1995}, and thus commonly referred to as the ``Lebeau--Robbiano method''.

\paragraph{The 1995 paper of Lebeau and Robbiano.}

In \cite{LebeauRobbiano1995}, to prove the small-time null controllability of the heat equation posed on a smooth compact Riemannian manifold $\mathcal{M}$, the strategy is as follows.
\begin{itemize}
    \item Consider $e_k$ the eigenfunctions of the operator $A$ of \eqref{eq:yt=Ay+Bu} (here $A = - \Delta$ with Dirichlet boundary conditions on $\mathcal{M}$) and $\lambda_k$ the associated eigenvalues.
    Consider $P_k$ to be the projection on the finite-dimensional space spanned by $e_1, \dotsc, e_k$.

    \item
    This automatically ensures that $\norm{\Sigma_T(y_0,0)}_Y \leq e^{-\lambda_k T} \norm{y_0}_Y$ when $P_k y_0 = 0$.
    By Weyl's law, this entails the dissipation property \eqref{eq:ass.dissip} with $b = \frac 2 d$ and $m = 1$.

    \item Using Carleman estimates inspired by \cite{Robbiano1995,Russell1973} for the elliptic operator $\partial_t^2 +\Delta$, prove that, for any $\theta > 1$, for any $T \in (0,1]$ and $y_0 \in Y$, there exists $u \in L^\infty(0,T)$ such that $P_k \Sigma_T(y_0,u) = 0$ with a cost estimate of the form
    \begin{equation}
        \label{eq:LR95-lossy}
        \norm{u}_{L^\infty} \leq C_1 \exp \left( c_1 \lambda_k^{\frac 12} + c_1 T^{-\theta} \right) \norm{y_0}_Y.
    \end{equation}
    This corresponds to our generalized cost assumption \eqref{eq:ass-approx-u-relax} with $a = \frac{1}{d}$ by Weyl's law.

    \item Use a geometric time and frequency schedule as exposed above to deduce the small-time null controllability of the full PDE.
    By \cref{rk:relax}, since $\theta > 1$ and $\frac{ma}{b-a} = 1$, their proof concludes to a cost estimate of the form $C_0 \exp(c_0 T^{-\theta})$ with $\theta > 1$, falling just short of the now-known optimal exponent $1$ (see below).
\end{itemize}

\paragraph{More recent variations and enhancements.}
The Lebeau--Robbiano method has been applied in countless settings, and many authors have proposed improvements.
We unfortunately only refer to a few of them, relevant to our exposition, and of which we are aware.

\begin{itemize}

    \item In \cite{JerisonLebeau1999} and \cite{LebeauZuazua1998}, Jerison, Lebeau and Zuazua started to isolate and study precisely static spectral inequalities of the form of \cref{ass:spectral}, avoiding the lossy estimate \eqref{eq:LR95-lossy}.

    \item In \cite{Miller2006}, Miller started a first abstraction of the previous results which concerned the heat equation to self-adjoint operators $A$ with admissible control operators $B$, obtaining the qualitatively optimal cost as $T \to 0$ for the control of a family of diffusion equations.

    \item Multiple subsequent works \cite{Miller2010,TenenbaumTucsnak2007,TenenbaumTucsnak2011} improved the estimates of the cost of controllability, or reduced the assumptions on the control operator $B$.
    Miller also remarked in \cite{Miller2010} that commutation is not necessary, only stability of the high modes as in \cref{rk:PKP}.

    \item In \cite{BeauchardPravdaStarov2018}, Beauchard and Pravda-Starov, in collaboration with Miller, introduced the key novelty that the projections $P_k$ don't need to be spectral projections onto the eigenspaces of the operator $A$ of \eqref{eq:yt=Ay+Bu} (which is thus not required to be self-adjoint).
    This important improvement allows to reuse well-known spectral inequalities for a reference operator and apply it to prove the controllability of another PDE (see \cref{sec:kolmogorov} where we apply a spectral inequality for the Laplacian to a Kolmogorov equation).

    As we prove in \cref{p:dual-hard}, thanks to the dissipation estimate, the lack of commutation between $P_k$ and the semigroup $\T_t$ does not prevent to obtain the one-step contraction of \cref{lem:LR-contract} which is used in the main time iteration.

    \item Later, it was identified that the method can be extended to time-dependent systems (with time-dependent operators \cite{BeauchardPravdaStarov2017}, or time-dependent control domains \cite{BeauchardEgidiPravdaStarov2020}).
    As mentioned in \cref{rk:LR-time-dep}, the same proofs apply, provided that the dissipation and spectral assumptions hold uniformly in time.

    \item The recent PhD thesis by Gallaun \cite{Gallaun2022} generalizes these results in multiple directions.
    In particular, in the spirit of \cref{rk:relax}, he proves that the structure of the time iteration is robust with respect to very general forms of the cost function and dissipation estimate.
    See also \cite{GallaunSeifertTautenhahn2020} for a generalization of \cref{sec:spectral} to the context of Banach spaces.

    \item All the references mentioned above work on the dual problem of observability (see \cite{TucsnakWeiss2009}).
    In particular, the proofs of \cite{Miller2010} or \cite{BeauchardPravdaStarov2018} rely on a telescoping series argument.
    In \cref{sec:LR-iteration,sec:spectral}, we obtain the same results while remaining in the control viewpoint.
\end{itemize}

Eventually, we refer to \cite[Section IV.2]{Boyer2022} for a nice textbook on the Lebeau--Robbiano method containing the proofs of the Carleman inequalities required for the control of the heat equation in bounded domains of $\R^d$.

\subsection{Example: interior control of the heat equation on \texorpdfstring{$[0,1]$}{[0,1]}}
\label{sec:heat-1D-interior}

In this paragraph, we investigate the small-time null controllability of the heat equation on the segment $[0,1]$, with an interior control.
More precisely, for $T > 0$, we consider the system:
\begin{equation}
    \label{eq:heat-1D-internal}
    \begin{cases}
        y_t - y_{xx} = u \mathbf{1}_\omega & \text{in } (0,T) \times (0,1), \\
        y(t,0) = y(t,1) = 0 & \text{in } (0,T), \\
        y(0,x) = y_0(x) & \text{in } (0,1),
    \end{cases}
\end{equation}
where $y_0 \in L^2(0,1)$ is the initial data and $u \in L^2((0,T) \times \omega)$ is a control, supported in a fixed measurable subset $\omega \subset (0,1)$.
We will consider solutions to this system in the space\footnote{By \cite[Chapter 1, Theorem 3.1]{LionsMagenes1972_Vol1}, $L^2((0,T);H^1_0(0,1)) \cap H^1((0,T);H^{-1}(0,1)) \hookrightarrow C^0([0,T];L^2(0,1))$ so one could omit it from the definition of $\mathcal{Y}$.}
\begin{equation}
    \mathcal{Y} := C^0([0,T];L^2(0,1)) \cap L^2((0,T);H^1_0(0,1)) \cap H^1((0,T);H^{-1}(0,1))
\end{equation}
equipped with the canonical norm.
We start with the following general well-posedness result.

\begin{proposition}
    \label{prop:heat-f-WP}
    For any $T > 0$, $y_0 \in L^2(0,1)$ and $f \in L^2((0,T);H^{-1}(0,1))$, there exists a unique solution $y \in \cY$ to
    \begin{equation}
        \label{eq:heat-f}
        \begin{cases}
            y_t - y_{xx} = f & \text{in } (0,T) \times (0,1), \\
            y(t,0) = y(t,1) = 0 & \text{in } (0,T), \\
            y(0,x) = y_0(x) & \text{in } (0,1).
        \end{cases}
    \end{equation}
    Moreover, there exists a constant $C > 0$ such that, for every $T$, $f$, $y_0$,
    \begin{equation}
        \label{eq:energy-heat-f-WP}
        \norm{y}_{C^0 L^2} + \norm{y}_{L^2 H^1} + \norm{y}_{H^1 H^{-1}} \leq C \left( \norm{y_0}_{L^2} + \norm{f}_{L^2 H^{-1}} \right).
    \end{equation}
\end{proposition}

\begin{proof}
    The existence and uniqueness of such a solution to \eqref{eq:heat-f} is classical textbook knowledge for linear parabolic PDEs (see for example \cite[Section 7.1.3, Theorems 3 and 4]{Evans2010}).
    The claimed inequality is a classical energy estimate, obtained by multiplying the PDE and integrating by parts.

    Multiplying the PDE by $y$ and integrating for $x \in (0,1)$, for almost every $t \in (0,T)$,
    \begin{equation}
        \frac{1}{2} \frac{\dd}{\dd t} \int_0^1 y^2(t,x) \dd x - \int_0^1 (y y_{xx})(t,x) \dd x = \langle f(t,\cdot), y(t,\cdot) \rangle_{H^{-1},H^1_0}.
    \end{equation}
    Integrating by parts in the second term,
    \begin{equation}
        \frac12 \frac{\dd}{\dd t} \int_0^1 y^2(t,x) \dd x + \int_0^1 y_x^2(t,x) \dd x
        \leq \norm{f(t,\cdot)}_{H^{-1}} \norm{y_x(t,\cdot)}_{L^2}
        \leq \frac 12 \norm{f(t,\cdot)}_{H^{-1}}^2 + \frac 12 \norm{y_x(t,\cdot)}_{L^2}^2.
    \end{equation}
    Hence, integrating in time leads to \eqref{eq:energy-heat-f-WP}.
\end{proof}

\subsubsection{Energy dissipation}
\label{sec:heat-1D-dissip}

Let $(\varphi_j)_{j \in \N^*}$ be the orthonormal basis of $L^2(0,1)$ of eigenfunctions of $- \partial_{xx}$ with Dirichlet boundary conditions, associated with the eigenvalues $\lambda_j = (j \pi)^2$.
More explicitly, $\varphi_j(x) := \sqrt{2} \sin (j \pi x)$.
Decomposing $y(t)$ on this basis, the explicit resolution of \eqref{eq:heat-1D-internal} with $u = 0$ leads to
\begin{equation}
    \langle y(t), \varphi_j \rangle = e^{- (j \pi)^2 t} \langle y_0, \varphi_j \rangle.
\end{equation}
In particular, if $\langle y_0, \varphi_j \rangle = 0$ for all $1 \leq j \leq k$, one has
\begin{equation}
    \label{eq:heat-1D-dissip}
    \norm{y(t)}_{L^2} \leq e^{- (k \pi)^2 t} \norm{y_0}_{L^2}.
\end{equation}

\subsubsection{Spectral inequality}

To control the low frequencies, we will rely on the following spectral inequality.
Tenenbaum and Tucsnak identified in \cite{TenenbaumTucsnak2011} that, for rectangular domains (including our segment $[0,1]$), one could replace the usual sophisticated Carleman inequalities by the following simple result by Tur\'an.
We give in \cref{s:turan} a self-contained one-page proof (see \cite{Turan1947} for Tur\'an's work with arbitrary frequencies, but where $\omega$ is an interval, and \cite{Nazarov1993} for measurable $\omega$).

\begin{proposition}
    \label{p:turan}
    Let $\omega \subset [0,1]$ be a measurable set with $|\omega| > 0$.
    There exists $C_\omega > 0$ such that, for any $n \geq 1$ and $a_1, \dotsc, a_n \in \R$,
    \begin{equation}
        \norm{f}_{L^2(0,1)} \leq C_\omega^n \norm{f}_{L^2(\omega)}
        \quad \text{where} \quad
        f(x) := \sum_{j=1}^n a_j \sin (j \pi x).
    \end{equation}
\end{proposition}

\subsubsection{Small-time null controllability}

We recover the following very well-known result (dating back to \cite{FattoriniRussell1971} or \cite{Seidman1984} in 1D, \cite{LebeauRobbiano1995} for higher dimension, and to \cite{ApraizEscauriaza2013} for measurable control set $\omega$).

\begin{theorem}
    \label{prop:heat-control}
    Assume that $|\omega| > 0$.
    System \eqref{eq:heat-1D-internal} is small-time null controllable and there exists $C > 0$ such that, for all $T > 0$ and $y_0 \in L^2(0,1)$, there exists $u \in L^2((0,T) \times \omega)$ such that the associated solution to \eqref{eq:heat-1D-internal} satisfies $y(T) = 0$ and one has the estimate $\norm{u}_{L^2} \leq C e^{C/T} \norm{y_0}_{L^2}$.
\end{theorem}

\begin{proof}
    Applying \cref{prop:heat-f-WP} with $f = u \mathbf{1}_\omega$ entails that \eqref{eq:heat-1D-internal} is an abstract linear control system with state space $Y = L^2(0,1)$, control space $U = L^2(\omega)$ and $p=2$.

    Let $P_k$ denote the projection on $\vect \{ \varphi_1, \dotsc, \varphi_k \}$ in $L^2(0,1)$.
    By \eqref{eq:heat-1D-dissip}, the dissipation assumption \eqref{eq:ass.dissip} holds with $b = 2$ and $m = 1$.
    With the notations of \cref{sec:spectral}, $P_k$ commutes with $\T_t$ and $B \in \mathcal{L}(L^2(\omega);L^2(0,1))$ is the extension by $0$ operator, so $B^\star$ is merely the restriction to $\omega$ operator.
    Hence \cref{p:turan} proves that, for any $y \in L^2(0,1)$,
    \begin{equation}
        \norm{P_k y}_Y \leq e^{c_\omega k} \norm{B^\star P_k y}_U
    \end{equation}
    for some constant $c_\omega > 0$.
    Thus we can apply \cref{p:dual-easy}, which implies that the control cost estimate \eqref{eq:ass.approx} holds with $a = 1$.
    Thus \cref{ass:proj+diss} holds and \cref{thm:LR} applies.

    Since $p = 2$, \cref{p:dual-easy} leads to a singular $T^{-\frac12}$ prefactor in estimate \eqref{eq:ass.approx}.
    However, as noted in \cref{rk:relax}, such a weak degeneracy can be absorbed with the same time iteration.
\end{proof}

\subsection{Example: boundary control of the heat equation on \texorpdfstring{$[0,\pi]$}{[0,pi]}}
\label{s:heat-boundary}

In this paragraph, we investigate the small-time null controllability of the heat equation on the segment $[0,\pi]$, with a boundary control.
More precisely, for $T > 0$, we consider the system:
\begin{equation}
    \label{eq:heat-1D-boundary}
    \begin{cases}
        y_t - y_{xx} = 0 & \text{in } (0,T) \times (0,\pi), \\
        y(t,0) = u(t) & \text{in } (0,T), \\
        y(t,\pi) = 0 & \text{in } (0,T), \\
        y(0,x) = y_0(x) & \text{in } (0,\pi),
    \end{cases}
\end{equation}
where $y_0 \in H^{-1}(0,\pi)$ is the initial data and $u \in L^2((0,T);\R)$ is a control.
One has the following well-posedness result (see \cite[Proposition 7.1.3]{TucsnakWeiss2009})

\begin{proposition}
    For any $T > 0$, $y_0 \in H^{-1}(0,\pi)$ and $u \in L^2((0,T);\R)$, there exists a unique weak solution $y \in C^0([0,T];H^{-1}(0,\pi)) \cap L^2((0,T);L^2(0,\pi))$ (defined by transposition) to \eqref{eq:heat-1D-boundary}.
\end{proposition}

For $j \in \N^*$, let $\varphi_j(x) := \sin (jx)$.
Then $(\varphi_j)_{j \in \N^*}$ is an orthogonal basis of $L^2(0,\pi)$ of eigenfunctions of $- \partial_{xx}$ with Dirichlet boundary conditions, associated with the eigenvalues $\lambda_j = j^2$.
Testing equation \eqref{eq:heat-1D-boundary} against $\varphi_j$ leads to
\begin{equation}
    \frac{\dd}{\dd t} \langle y, \varphi_j \rangle
    = \langle \partial_{xx} y, \varphi_j \rangle
    = [\partial_x y \varphi_j - y \partial_x \varphi_j]_0^\pi + \langle y, \partial_{xx} \varphi_j \rangle
    = j u(t) - j^2 \langle y, \varphi_j \rangle.
\end{equation}
Thus, we have the Duhamel representation formula
\begin{equation}
    \langle y(T), \varphi_j \rangle = e^{-j^2 T} \langle y_0,\varphi_j \rangle + j \sqrt{2/\pi} \int_0^T e^{-j^2(T-t)} u(t) \dd t.
\end{equation}

\subsubsection{Energy dissipation}

In particular, as in \cref{sec:heat-1D-dissip}, with a null control, if $\langle y_0, \varphi_j \rangle = 0$ for all $1 \leq j \leq k$, one has
\begin{equation}
    \label{eq:heat-1D-dissip-bis}
    \norm{y(t)}_{L^2} \leq e^{- k^2 t} \norm{y_0}_{L^2}.
\end{equation}

\subsubsection{Cost of control of the low frequencies}

To control the low frequencies, we will rely on the analysis performed in \cref{s:large-systems}.
In particular, given $k \in \N^*$, $T > 0$ and $y_0 \in H^{-1}(0,\pi)$, finding $u \in L^2((0,T);\R)$ such that $\langle y(T), \varphi_j \rangle = 0$ for all $1 \leq j \leq k$ is equivalent to the null controllability of the finite-dimensional system \cref{ex:diag}.
By \cref{thm:fin-cost}, we know that there exists such a $u$ with
\begin{equation}
    \norm{u}_{L^2} \leq \sqrt{\frac{k^2}{T}} \exp \left( 6 \sqrt{\frac{k}{T}} \right) \norm{y_0}_{H^{-1}}.
\end{equation}
Using $x \leq e^x$ and Young's inequality, we obtain, for $T \in (0,1]$,
\begin{equation}
    \label{eq:heat-1D-BC-cost}
    \norm{u}_{L^2} \leq e^{4 (k - T^{-1})} \norm{y_0}_{H^{-1}}.
\end{equation}

\subsubsection{Small-time null controllability}

We recover the following very well-known result.
Usually, it is proved using the \emph{moment method}, as in the original proof of \cite{FattoriniRussell1971}.
Typically, the Lebeau--Robbiano method is not used for the boundary controllability of \eqref{eq:heat-1D-boundary} because the spectral inequality required in \cref{sec:spectral} is not satisfied (see \cite[Remark IV.2.34]{Boyer2022}).
However, here, we use the time iteration of the Lebeau--Robbiano method, but we prove the controllability of the low frequencies using another technique (see \cref{s:large-systems}).

\begin{theorem}
    \label{prop:heat-1D-control-boundary}
    System \eqref{eq:heat-1D-boundary} is small-time null controllable.
\end{theorem}

\begin{proof}
    Let $P_k$ be the projection on $\vect \{ \varphi_1, \dotsc, \varphi_k \}$ in $L^2(0,\pi)$.
    By \eqref{eq:heat-1D-dissip-bis}, the dissipation assumption \eqref{eq:ass.dissip} holds with $b = 2$ and $m = 1$.
    By \eqref{eq:heat-1D-BC-cost}, the relaxed cost estimate \eqref{eq:ass-approx-u-relax} of \cref{rk:relax} holds with $a = 1$ and $\sigma = 1$.
    Thus \cref{ass:proj+diss} is satisfied and \cref{thm:LR} applies.
\end{proof}

\subsection{Example: control of the heat equation in \texorpdfstring{$\R^d$}{Rd} from a thick set}
\label{sec:heat-Rd}

In this paragraph, we investigate the small-time null controllability of the heat equation set in the whole space $\R^d$ (for $d \geq 1$).
We consider
\begin{equation}
    \label{eq:heat-Rd}
    \begin{cases}
        y_t - \Delta y = u \mathbf{1}_\omega & \text{in } (0,T) \times \R^d, \\
        y(0,x) = y_0(x) & \text{in } \R^d
    \end{cases}
\end{equation}
where $y_0 \in L^2(\R^d)$ is the initial data and $u \in L^1((0,T);L^2(\omega))$ a control, supported in a fixed measurable subset $\omega \subset \R^d$.
This system is well-posed (see your favorite textbook, e.g.\ \cite[Sec.\ 9]{Golse2020}).
In particular, it defines an abstract linear control system with $Y = L^2(\R^d)$, $U = L^2(\omega)$ and $p = 1$.

\begin{lemma}
    For any $T > 0$, $y_0 \in L^2(\R^d)$ and $u \in L^1((0,T);L^2(\omega))$, there exists a unique solution $y \in C^0([0,T];L^2(\R^d))$ to \eqref{eq:heat-Rd}.
\end{lemma}

\subsubsection{Energy dissipation and projections}

In the absence of control, the Fourier transform of the solution is given for any $t \geq 0$ and $\xi \in \R^d$ by
\begin{equation}
    \label{eq:heat-fourier}
    \widehat{y}(t,\xi) = e^{-t |\xi|^2} \widehat{y_0}(\xi).
\end{equation}
In particular, if $\supp \widehat{y_0} \subset B_{\R^d}(0,k)$ for some $k \geq 0$, one has
\begin{equation}
    \label{eq:heat-etk2}
    \norm{y(t)}_{L^2} \leq e^{-t k^2} \norm{y_0}_{L^2}.
\end{equation}

\subsubsection{Spectral inequality}
\label{sec:heat-Rd-spectral}

Hence, by \cref{p:dual-easy}, we must find an inequality of the form
\begin{equation}
    \norm{P_k y}_{L^2(\R^d)} \leq C_k \norm{P_k y}_{L^2(\omega)}.
\end{equation}

For \eqref{eq:heat-Rd} to be small-time controllable, it is natural to require that the control domain $\omega$ is well spread-out all over $\R^d$.
The modern theory involves the following definition which is now well understood in the control theory community (see e.g.\ \cite{BeauchardEgidiPravdaStarov2020,EgidiVeselic2018,WangWangZhangZhang2019}).

\begin{definition}
    Let $\gamma \in (0,1]$ and $a > 0$.
    We say that a measurable $\omega \subset \R^d$ is $(\gamma,a)$-\emph{thick} when
    \begin{equation}
        \label{eq:thick}
        \forall x \in \R^d, \quad
        \left| \omega \cap (x + [0,a]^d) \right| \geq \gamma a^d.
    \end{equation}
\end{definition}

It is well-known that band-limited functions cannot have compact support (uncertainty principle).
More quantitatively, for such thick control domains, one has the following important inequality, initially due to Logvinenko and Sereda \cite{LogvinenkoSereda1974} with a constant $C^{Na/\gamma}$ and enhanced as below by Kovrijkine \cite{Kovrijkine2001}.
We give a short self-contained proof in \cref{sec:LS-proof}.
The heuristic idea of the proof is to split $\R^d$ into cells of size $a$.
Within each cell, band-limited functions are well-approximated by their Taylor expansion, to which we can apply a Remez inequality.

\begin{proposition}
    \label{thm:LS}
    There exists a universal constant $C > 0$ such that, for any $\gamma \in (0,1]$, $a > 0$ and any measurable subset $\omega \subset \R^d$ which is $(\gamma,a)$-thick, one has
    \begin{equation}
        \label{eq:logvinenko-sereda}
        \norm{f}_{L^2(\R^d)} \leq \left(\frac{C}{\gamma}\right)^{C (N a + 1)} \norm{f}_{L^2(\omega)}
    \end{equation}
    for all $N \geq 0$ and $f \in L^2(\R^d;\C)$ such that $\supp \widehat{f} \subset B_{\R^d}(0,N)$.
\end{proposition}

\subsubsection{Small-time null controllability}

We recover the following result, due to \cite[Theorem 3]{EgidiVeselic2018} and \cite[Theorem 1.1]{WangWangZhangZhang2019} (which moreover prove that thickness is in fact also a necessary condition for small-time null controllability).

\begin{theorem}
    Let $\gamma \in (0,1]$ and $a > 0$.
    Assume that $\omega$ is $(\gamma,a)$-thick.
    Then the heat equation~\eqref{eq:heat-Rd} is small-time null controllable.
\end{theorem}

\begin{proof}
    Define $P_k$ as the projection on the Fourier modes $|\xi| \leq k$.
    More precisely, set
    \begin{equation}
        P_k y := \mathcal{F}^{-1} \left( \xi \mapsto \widehat{y}(\xi) \mathbf{1}_{|\xi| \leq k} \right).
    \end{equation}
    First, by \eqref{eq:heat-etk2}, the dissipation assumption \eqref{eq:ass.dissip} holds immediately with $m = 1$ and $b = 2$.

    Second, since $\omega$ is $(\gamma,a)$-thick, by \cref{thm:LS}, for any $y \in L^2(\R^d)$,
    \begin{equation}
        \norm{P_k y}_{L^2(\R^d)} \leq C_1 e^{c_1 k} \norm{P_k y}_{L^2(\omega)}
    \end{equation}
    for some constants $C_1, c_1 > 0$ depending on $a$ and $\gamma$.
    By \eqref{eq:heat-fourier}, it is clear that $P_k$ commutes with the semigroup.
    Written in the Duhamel form \eqref{eq:Sigma_T=B}, the operator $B : L^2(\omega) \to L^2(\R^d)$ is the extension by $0$.
    Thus $B^\star$ is merely the restriction to $\omega$ operator and $\norm{P_k y}_{L^2(\omega)} = \norm{B^\star P_k y}_{L^2(\omega)}$.
    Thus, by \cref{p:dual-easy}, for any $T \in (0,1]$ and $y_0 \in L^2(\R^d)$, there exists $u \in L^1((0,T);L^2(\omega))$ such that $P_k \Sigma_T(y_0,u) = 0$ and $\norm{u}_{L^1} \leq M_1 C_1 e^{c_1 k} \norm{y_0}_{L^2(\R^d)}$.
    Hence \eqref{eq:ass.approx} holds with $a = 1$.

    Thus, \cref{ass:proj+diss} holds, and the result follows from \cref{thm:LR}.
\end{proof}

\subsection{Example: 2D Kolmogorov equation}
\label{sec:kolmogorov}

In this paragraph, we investigate the small-time null controllability of a 2D Kolmogorov equation (also called Fokker--Planck equation), based on the hypoelliptic operator $z \partial_x - \partial_{zz}$.
PDEs based on this operator are ubiquitous in kinetic theory or boundary layer theory.
We consider:
\begin{equation}
    \label{eq:kolmogorov}
    \begin{cases}
        y_t + z y_x - y_{zz} = u \mathbf{1}_\omega & \text{in } (0,T) \times \R^2, \\
        y(0,x,z) = y_0(x,z) & \text{in } \R^2,
    \end{cases}
\end{equation}
where $y_0 \in L^2(\R^2)$ is the initial data and $u \in L^1((0,T);L^2(\omega))$ is a control, supported in a fixed measurable subset $\omega \subset \R^2$.
We have the following well-posedness result (see \cite[Section 2]{LeRousseauMoyano2016}).

\begin{proposition}
    For any $T > 0$, $y_0 \in L^2(\R^2)$ and $u \in L^1((0,T);L^2(\omega))$, there exists a unique solution $y \in C^0([0,T];L^2(\R^2))$ to \eqref{eq:kolmogorov}.
\end{proposition}

Thus, we have an abstract linear control system with $Y = L^2(\R^2)$, $U = L^2(\omega)$ and $p = 1$.

\medskip

To illustrate that the theory exposed in \cref{sec:LR-iteration} can be applied with projections which are not on eigenspaces of the considered operator, we will here use the same projections on Fourier modes as for the heat equation in \cref{sec:heat-Rd}, thus re-using the spectral theory of \cref{sec:heat-Rd-spectral} involving thick sets and Logvinenko--Sereda inequality.

\subsubsection{Energy dissipation}

The following decay estimate is also proved in \cite[Section 3]{Zhang2016}.

\begin{proposition}
    \label{p:kolmogorov-dissip}
    There exists $c_0 := (4-\sqrt{13})/6 > 0$ such that, for any $k \geq 1$ and $y_0 \in L^2(\R^2)$ such that $\widehat{y_0} = 0$ on $B(0,k)$, we have, for any $t \geq 0$,
    \begin{equation}
        \norm{y(t)}_{L^2} \leq e^{-c_0 \min \{ t, t^3 \} k^2} \norm{y_0}_{L^2}.
    \end{equation}
\end{proposition}

\begin{proof}
    We perform a change of variables to follow the characteristics of the Kolmogorov equation.
    Set $\phi(t,x,z) := y(t,x+z t,z)$ (or conversely $y(t,x,z) = \phi(t,x-z t,z)$).
    In particular, $\phi(0) = y_0$.

    Then, since $\partial_t y = \partial_t \phi - z \partial_x \phi$, $\partial_x y = \partial_x \phi$ and $\partial_{zz} y = \partial_{zz} \phi - 2 t \partial_{xz} \phi + t^2 \partial_{xx} \phi$, we get
    \begin{equation}
        (\partial_t - \partial_{zz} + 2 t \partial_{xz}- t^2 \partial_{xx}) \phi = 0.
    \end{equation}
    Hence we have a PDE with time-varying but constant in space coefficients.
    Taking the Fourier transform in both directions $(x,z) \to (\xi,\zeta)$, we get
    \begin{equation}
        (\partial_t + \zeta^2 + 2 t \xi \zeta + t^2 \xi^2) \hat{\phi} = 0.
    \end{equation}
    Explicit integration leads to
    \begin{equation}
        \label{eq:phi-explicit}
        \hat{\phi}(t,\xi,\zeta) = \exp \left( - t \zeta^2 - t^2 \xi \zeta - \frac{t^3}{3} \xi^2 \right) \hat{\phi_0}(\xi,\zeta).
    \end{equation}
    Writing
    \begin{equation}
        | t^2 \xi \zeta | \leq \frac{\alpha}{2} t \zeta^2 + \frac{1}{2\alpha} t^3 \xi^2
    \end{equation}
    and optimizing with respect to $\alpha := (2+\sqrt{13})/3$, we obtain that
    \begin{equation}
        t \zeta^2 + t^2 \xi \zeta + \frac{t^3}{3} \xi^2 \geq \frac{4-\sqrt{13}}{6} (t\zeta^2 + t^3\xi^2).
    \end{equation}
    Since the initial data is supported for $\zeta^2 +\xi^2 \geq k^2$, we obtain that
    \begin{equation}
        \norm{\phi(t)}_{L^2} \leq e^{-c_0 t \min \{1,t^2\} k^2} \norm{\phi(0)}_{L^2},
    \end{equation}
    which is the claimed estimate since, for any $t \geq 0$, $\norm{\phi(t)}_{L^2} = \norm{y(t)}_{L^2}$.
\end{proof}

\subsubsection{Small-time null controllability}

We now recover the following theorem, due to \cite{LeRousseauMoyano2016} (when $\omega$ has a particular product structure $\omega = \omega_x \times \omega_z$) and \cite{Zhang2016} for general $\omega$.
Both of these references make the stronger assumption on the control set that it contains open balls of fixed radius near all points.

\begin{theorem}
    \label{thm:kolmogorov}
    Let $\gamma \in (0,1]$ and $a > 0$.
    Assume that $\omega$ is $(\gamma,a)$-thick.
    Then the Kolmogorov equation \eqref{eq:kolmogorov} is small-time null controllable.
\end{theorem}

\begin{proof}
    Define $P_k$ as the projection on the Fourier modes $|\xi|^2+|\zeta|^2 \leq k^2$.
    First, by \cref{p:kolmogorov-dissip}, the dissipation condition \eqref{eq:ass.dissip} holds with $m = 3$ and $b = 2$.

    Second, since $\omega$ is $(\gamma,a)$-thick, by \cref{thm:LS}, for any $y \in L^2(\R^2)$,
    \begin{equation}
        \norm{P_k y}_{L^2(\R^2)} \leq C_1 e^{c_1 k} \norm{P_k y}_{L^2(\omega)}
    \end{equation}
    for some constants $C_1, c_1 > 0$ depending on $a$ and $\gamma$.
    Written in the Duhamel form \eqref{eq:Sigma_T=B}, the operator $B : L^2(\omega) \to L^2(\R^2)$ is the extension by $0$.
    Thus $B^\star$ is merely the restriction to $\omega$ operator and $\norm{P_k y}_{L^2(\omega)} = \norm{B^\star P_k y}_{L^2(\omega)}$.

    Contrary to the heat case, here $P_k$ does not commute with the semigroup.
    Nevertheless, by \cref{p:dual-hard}, the conclusion of \cref{lem:LR-contract} holds with $a = 1$.

    Thus the result follows from \cref{thm:LR}.
\end{proof}

\begin{remark}
    The Kolmogorov equation \eqref{eq:kolmogorov} is a particular case of a more general class of hypoelliptic equations, the Ornstein--Uhlenbeck equations, studied in \cite{BeauchardPravdaStarov2018}, of the form
    \begin{equation}
        y_t - \frac 12 \operatorname{Tr} (Q \nabla^2 y) - \langle Bx, \nabla \cdot y \rangle = u \mathbf{1}_\omega.
    \end{equation}
    Our 2D Kolmogorov case is
    \begin{equation}
        Q =
        \begin{pmatrix}
            0 & 0 \\ 0 & 2
        \end{pmatrix}
        \quad \text{and} \quad
        B =
        \begin{pmatrix}
            0 & -1 \\ 0 & 0
        \end{pmatrix}
    \end{equation}
    which satisfies the condition $\operatorname{rank} (Q^{\frac 12}, B Q^{\frac 12}) = 2$.
    In particular, the dissipation estimate \cref{p:kolmogorov-dissip} is included in \cite[Lemma 3.1]{BeauchardPravdaStarov2018} and \cref{thm:kolmogorov} in \cite[Theorem 1.3]{BeauchardPravdaStarov2018}.
\end{remark}

\subsection{Comments on our formalism}
\label{sec:weiss}

\subsubsection{Time continuity and comparison with the usual definition}

Our \cref{def:linear-io} of an \emph{abstract linear control system} slightly differs from the reference one given by \cite[Definition 2.1]{Weiss1989}.
Indeed, as in \cite[Section 3]{Salamon1988}, we require time continuity.
Given Banach spaces $Y$ and $U$, and $p \in [1,\infty]$, Weiss' definition of \cite{Weiss1989} is:

\begin{definition}
    \label{def:linear-io-weiss}
    An \emph{abstract linear control system} on $Y$ is a pair $\Sigma = (\T,\Phi)$, where $\T = (\T_t)_{t \geq 0}$ is a strongly continuous semigroup on $Y$ and $\Phi = (\Phi_t)_{t \geq 0}$ is a family of bounded operators from $L^p(\R_+;U)$ to $Y$ such that, for all $u,v \in L^p(\R_+;U)$ and $\tau,t \geq 0$,
    \begin{equation}
        \label{eq:composition-weiss}
        \Phi_{\tau+t}(u \underset{\tau}{\diamond} v)
        = \T_t \Phi_{\tau} u + \Phi_t v.
    \end{equation}
\end{definition}

Recall the definition of a strongly continuous semigroup on $Y$ (see \cite{Pazy1983} for more details).

\begin{definition}
    A \emph{strongly continuous semigroup on $Y$} is a family $\T = (\T_t)_{t \geq 0}$ of bounded linear operators from $Y$ to $Y$ such that $\T_0 = \Id$, $\T_{t+s} = \T_t \circ \T_s$ (for all $t,s \geq 0$), and $\lim_{t \to 0} \T_t y_0 = y_0$ for all fixed $y_0 \in Y$.
\end{definition}

It is almost straightforward to go back and forth between Weiss' definition and ours.

\begin{lemma}
    \cref{def:linear-io} and \cref{def:linear-io-weiss} are equivalent for $p \in [1,\infty)$.

    When $p = \infty$, \cref{def:linear-io} is equivalent to \cref{def:linear-io-weiss} with the added constraint that, for any $u \in L^p(\R_+;U)$, $t \mapsto \Phi_t u$ is continuous on $\R_+$.
\end{lemma}

\begin{proof}
    Given an abstract linear control system $\Sigma$ in the sense of \cref{def:linear-io}, it suffices to define $\T_t y := \Sigma_t(y,0)$ for any $y \in Y$ and $t \geq 0$ and $\Phi_t u := \Sigma_t(0,u)$ for any $u \in L^p(\R_+;U)$ and $t \geq 0$.
    One easily checks that $(\T,\Phi)$ then satisfies \cref{def:linear-io-weiss} (where \eqref{eq:composition-weiss} follows from \eqref{eq:Sigma-compo}).

    \medskip

    Conversely, given an abstract linear control system $(\T,\Phi)$ in the sense of \cref{def:linear-io-weiss}, it suffices to define $\Sigma_t(y_0,u) := \T_t y_0 + \Phi_t u$.
    One easily checks that $\Sigma$ is a family of continuous linear maps satisfying \eqref{eq:Sigma-compo} (thanks to \eqref{eq:composition-weiss}).
    The only subtlety is that \cref{def:linear-io} requires that $t \mapsto \Sigma_t(y_0,u)$ is continuous in time for any fixed $y_0 \in Y$ and $u \in L^p(\R_+;U)$.

    The continuity of $t \mapsto \T_t y_0$ from $\R_+$ to $Y$ is a basic property of strongly continuous semigroups (resulting from the required continuity at $t = 0$, the uniform boundedness principle and the semigroup property, see \cite[Chapter 1, Corollary 2.3]{Pazy1983}).

    The continuity of $t \mapsto \Phi_t u$ from $\R_+$ to $Y$ (for a fixed $u \in L^p(\R_+;U)$) is more subtle.
    When $p \in [1,\infty)$, one can use that $u \mathbf{1}_{(0,t)} \to 0$ in $L^p$ as $t \to 0$, and that translations are continuous in $L^p$ to prove that $t \mapsto \Phi_t u$ is indeed continuous (see \cite[Proposition 2.3]{Weiss1989}).
\end{proof}

\begin{remark}
    When $p = \infty$, it appears to be unknown whether the time continuity of $t \mapsto \Phi_t u$ follows from \cref{def:linear-io-weiss} (see \cite[Problem 2.4]{Weiss1989} for the initial statement, and \cite[p.\ 23]{MirochenkoPrieur2020} or \cite[Section 6]{JacobRobertPartingtonSchwenninger2018} recent accounts, and \cite[Theorem 4.3.1]{Staffans2005} for continuity with regulated controls).
\end{remark}

\subsubsection{Relation with the evolution equation}

The relation between a time-independent uncontrolled evolution equation of the form $\dot{y} = A y$ with initial data $y(0) = y_0$ and its integrated version $y(t) = \T_t y_0$ is well understood.
$A$ is called the \emph{generator} of the semigroup $\T_t$.
One knows how to define $A$ from $\T$, and conversely one knows many sufficient conditions on $A$ to ensure that it generates a strongly continuous semigroup $\T$.
We refer to the classical monograph \cite{Pazy1983}.

For a controlled equation of the form $\dot{y} = Ay + Bu$ and $u \in L^p(\R_+;U)$, many results are known.

Given $A$ and $B$, the concept of ``admissibility'' (of $B$ with respect to $A$) determines whether they will generate an abstract linear control system.
We refer to \cite{Weiss1989} or \cite{TucsnakWeiss2009}.

Conversely, given an abstract linear control system as in \cref{def:linear-io-weiss}, when $p < \infty$, one can construct an operator $B$ (see e.g.\ \cite[Thm.\ 3.9]{Weiss1989}).
When $p = \infty$, the existence is not guaranteed (see \cite[Remark 3.7]{Weiss1991} for a counter-example and \cite[Thm.\ 4.2.1]{Staffans2005} for regulated controls).

More generally, one can also consider abstract linear control systems along with observation operators (see e.g.\ \cite{TucsnakWeiss2014} and \cite{JacobPartington2004} for a survey of admissibility of observation operators).

\subsection{Why only null controllability?}
\label{sec:why-null}

For ODEs (such as in \cref{sec:intro}) and for some conservative PDEs (such as the semilinear wave equation in \cref{sec:wave}), one can often prove \emph{exact controllability}: for any initial state $y_0\in Y$ and any target $y_T\in Y$, there exists a control $u$ steering the system from $y_0$ to $y_T$ in time $T$.

However, for the dissipative PDEs studied in this section, one cannot hope for such a result, due to the regularizing effect of the dynamics.
Roughly speaking, as soon as $t>0$, solutions become smoother than the initial data, even with a non-zero control.
Consequently, for fixed $T>0$, the set of states that can be attained at time $T$ is usually contained in a proper subspace of the state space $Y$.
In particular, one cannot expect to reach an arbitrary target $y_T \in Y$.

\begin{example}
    Consider the 1D heat equation~\eqref{eq:heat-1D-internal} with internal control supported in $\omega=(\tfrac12,1)$.
    Let $y_0\in L^2(0,1)$, $T>0$, and $u\in L^2((0,T)\times\omega)$, and denote by $y$ the corresponding solution.
    On the uncontrolled region $(0,\tfrac12)$, the function $y$ satisfies the \emph{homogeneous} heat equation (since $u\equiv 0$ there).
    Standard interior regularity (which can be proved by energy estimates and bootstrapping) then yields $y(T) \in C^\infty(0,\tfrac 12)$.
    This rules out exact controllability in the state space $Y=L^2(0,1)$.
    Note that $y(T)$ is actually even smoother than $C^\infty$, since it typically enjoys analytic regularity.
\end{example}

\paragraph{Reachable set.}
Obtaining a precise description of the \emph{reachable set} at time $T$ (i.e.\ the range of the input-to-state map) is a delicate problem, even in very simple settings.
For instance, for the 1D heat equation controlled from both endpoints, the program initiated in \cite{FattoriniRussell1971} was only completed much more recently in \cite{HartmannKellayTucsnak2020,HartmannOrsoni2021}; see also \cite[Section~3]{ErvedozaLeBalchTucsnak2022} for a brief survey of the many intermediate contributions.
Beyond its intrinsic interest, a sharp characterization of reachable states unlocks new techniques for nonlinear control (see \cref{sec:reachable}).

\paragraph{Null controllability and controllability to trajectories.}
In more complex PDE models, a common and robust objective is \emph{controllability to trajectories}.
Fix a \emph{reference pair} $(\bar y_0,\bar u)$, and let $\bar y$ be the associated solution.
One asks whether, for any other initial state $y_0\in Y$, there exists a control $u\in L^p((0,T);U)$ such that $y(T) = \bar{y}(T)$.
The special case $\bar y_0=0$, $\bar u=0$ (hence $\bar y\equiv 0$) is the \emph{null controllability} problem:
given $y_0\in Y$, can one find $u$ such that $y(T)=0$?
This is also the most natural goal in many applications, where the objective is to drive the system to rest.

For \emph{linear} systems, controllability to trajectories is equivalent to null controllability: indeed, if $\bar y$ solves the system with data $(\bar y_0,\bar u)$, then the difference $z:=y-\bar y$ solves the same linear dynamics with initial condition $z(0)=y_0-\bar y_0$ and control $u-\bar u$; thus $y(T)=\bar y(T)$ is equivalent to $z(T)=0$.

\paragraph{Operator theoretic definition.}
We can formalize these notions using the input-output maps introduced in \cref{sec:weiss}.
For $T>0$, recall that the \emph{semigroup} $\T_T : Y \to Y$ and the \emph{control map} $\Phi_T : L^p((0,T);U) \to Y$ are defined by $\T_T y_0 := \Sigma_T(y_0, 0)$ and $\Phi_T u := \Sigma_T(0, u)$.
By linearity, the final state is given by $y(T) = \T_T y_0 + \Phi_T u$.
Thus, with these notations
\begin{itemize}
    \item \emph{Exact controllability} is the surjectivity condition $\Ran \Phi_T = Y$.
    \item \emph{Null controllability} is the inclusion $\Ran \T_T \subset \Ran \Phi_T$.
    \item \emph{Approximate controllability} is the density condition $\overline{\Ran \Phi_T}=Y$.
\end{itemize}
For dissipative systems, since $\T_T$ maps $Y$ into a smaller subspace of smooth functions, the inclusion $\Ran \T_T \subset \Ran \Phi_T$ is much easier to satisfy than $\Ran \Phi_T = Y$.

\newpage
\section{From linear to local nonlinear controllability}
\label{sec:lin-2-nonlin}

In this section, we consider a nonlinear control system.
Assuming that its linearization at the null equilibrium is controllable, we will present methods to establish the local controllability of the nonlinear system.
Heuristically, one can think that we consider systems of the form:
\begin{equation}
    \label{eq:nonlin-system}
    \dot{y} = A y + B u + N(u,y),
\end{equation}
where $N$ models a nonlinearity such that $N(0,0) = 0$ and $D N(0,0) = 0$.
For small data, one expects that the contribution of $N$ to the dynamics can be estimated perturbatively and that we can transfer the controllability of the linearized system $\dot{y} = Ay + Bu$ to \eqref{eq:nonlin-system}.

\subsection{Input-output formalism}
\label{sec:abstract-nonlinear}

As in \cref{sec:linear-io}, we will use an abstract ``input-output'' formalism which avoids talking about the objects $A$, $B$ and $N$ involved in the heuristic system \eqref{eq:nonlin-system}.
Related nonlinear definitions can for example be found in \cite[Definition 1.3]{MirochenkoPrieur2020}.
As in \cref{sec:linear-io}, fix $p \in [1,\infty]$.

\begin{definition}
    \label{def:nonlinear-io}
    We say that $\cS$ is an \emph{abstract nonlinear control system} when it is a family $(\cS_t)_{t \geq 0}$ of maps from $Y \times L^p(\R_+;U)$ to $Y$ such that:
    \begin{itemize}
        \item for all $y_0 \in Y$ and $u \in L^p(\R_+;U)$, $t \mapsto \cS_t(y_0,u)$ is continuous on $\R_+$ with $\cS_0(y_0,u) = y_0$,
        \item for all $t \geq 0$, $\cS_t(0,0) = 0$,
        \item for all $t, \tau \geq 0$, $y_0 \in Y$ and $u,v \in L^p(\R_+;U)$,
        \begin{equation}
            \label{eq:cS-compo}
            \cS_{t+\tau}(y_0,u \underset{\tau}{\diamond} v) = \cS_t(\cS_\tau(y_0,u), v).
        \end{equation}
    \end{itemize}
\end{definition}

Hence $\cS$ represents the nonlinear state transition map for the heuristic evolution equation \eqref{eq:nonlin-system}, while $\Sigma$ of \cref{def:linear-io} represents the linear state transition map for its linearized version \eqref{eq:yt=Ay+Bu}.
We give multiple concrete examples in the following subsections.

Typically, for well-posed systems in the Hadamard sense, the solution state $\cS_t(y_0,u)$ depends continuously on $y_0$ and $u$, but we will not need this property here.

\begin{remark}
    To avoid burdening the exposition, \cref{def:nonlinear-io} corresponds to globally well-posed systems.
    Since we will be considering only local properties near the equilibrium, it would be sufficient in this section that $\cS$ be defined only for small enough times, controls and initial states.
\end{remark}

\subsection{Transferring exact controllability}
\label{sec:lin-2-nonlin-exact}

We start with the easier case where the underlying linear system is exactly controllable.

\begin{definition}
    \label{def:lin-exact-cont}
    For $T > 0$, we say that the abstract linear control system $\Sigma$ is \emph{exactly controllable} in time $T$ when there exists a bounded linear map $L : Y \to L^p((0,T);U)$ such that, for all $y_1 \in Y$, the control $u := L(y_1)$ satisfies $\Sigma_T(0,u) = y_1$.
\end{definition}

Under a natural regularity assumption, the local controllability of the nonlinear system is a direct consequence of the controllability of its linearization.
To lighten the conditions to be checked, we use the notion of strong Fréchet-differentiability.
We refer the reader to \cite[Chapter 25]{Schechter1997}.

\begin{definition}
    \label{def:strong-Frechet}
    Let $E, F$ be Banach spaces and $f : E \to F$.
    We say that $f$ is \emph{strongly Fréchet-differentiable at $0 \in E$} when there exists a continuous linear map $Df(0) : E \to F$ such that
    \begin{equation}
        \norm{f(x_1) - f(x_2) - Df(0) (x_1 - x_2)}_F = \underset{x_1, x_2 \to 0}{o} \left( \norm{x_1 - x_2}_{E} \right).
    \end{equation}
\end{definition}

The following implicit function theorem is proved in \cite[Paragraph 25.13]{Schechter1997}.

\begin{lemma}
    \label{lem:TFI-Strong-Frechet}
    Let $E_1, E_2, F$ be Banach spaces and $f : E_1 \times E_2 \to F$ such that $f(0,0) = 0$.
    Assume that $f$ is strongly Fréchet-differentiable at $(0,0)$ and that $D_2 f(0,0) : E_2 \to F$ is a linear isomorphism.
    Then there exists a Lipschitz-continuous map $g:E_1\to E_2$ defined in a neighborhood of $0 \in E_1$ such that, for every $(x,y)$ in a neighborhood of $(0,0) \in E_1 \times E_2$, $f(x,y) = 0$ if and only if $y = g(x)$.
    Moreover, $g$ is strongly Fréchet-differentiable at $0$ and $Dg(0) = - (D_2 f(0,0))^{-1} D_1 f(0,0)$.
\end{lemma}

\begin{remark}
    Strong differentiability at $0$ is stronger than usual Fréchet-differentiability, but weaker than being of class $C^1$ in a neighborhood of $0$ (see \cite[end of Section 2]{Nijenhuis1974}).
    It is tailored for applying implicit function arguments without requiring the knowledge (or the existence) of the derivatives in a neighborhood of the point; only a single derivative at the point suffices.
\end{remark}

\begin{corollary}
    \label{cor:TFI-surjective}
    Let $E_1, E_2, F$ be Banach spaces and $f : E_1 \times E_2 \to F$ such that $f(0,0) = 0$.
    Assume that $f$ is strongly Fréchet-differentiable at $(0,0)$ and that $D_2 f(0,0) : E_2 \to F$ has a continuous right inverse $R : F \to E_2$.
    Then there exists a Lipschitz-continuous map $g:E_1 \to E_2$ defined in a neighborhood of $0 \in E_1$ with $g(0) = 0$ such that, for every $x$ in a neighborhood of $0 \in E_1$, $f(x,g(x)) = 0$.
    Moreover, $g$ is strongly Fréchet-differentiable at $0$ and $Dg(0) = - R \circ D_1 f(0,0)$.
\end{corollary}

\begin{proof}
    The strategy is to construct a subspace of $E_2$ on which the partial derivative becomes an isomorphism, and then apply \cref{lem:TFI-Strong-Frechet}.
    Let $L_2 := D_2 f(0,0) : E_2 \to F$.
    By assumption, $L_2$ admits a continuous right inverse $R : F \to E_2$.
    Let us define the operator $\pi := R \circ L_2 : E_2 \to E_2$.
    Then $\pi$ is a continuous projection with $\ker \pi = \ker L_2$.
    Thus, we can write, as a topological vector sum $E_2 = \ker L_2 \oplus \operatorname{im} \pi$.
    Moreover, $L_2$ is an isomorphism from $\operatorname{im} \pi$ to $F$.
\end{proof}

\begin{theorem}
    \label{thm:linear-to-nonlinear}
    Let $T > 0$, $\cS$ be an abstract nonlinear control system and $\Sigma$ be a linear one.
    Assume that:
    \begin{enumerate}
        \item $\Sigma$ is exactly controllable in time $T$;
        \item $\cS_T : Y \times L^p(\R_+;U) \to Y$ is strongly Fréchet-differentiable at $(0,0)$, with derivative~$\Sigma_T$.
    \end{enumerate}
    Then the nonlinear system $\cS$ is \emph{locally exactly controllable} in time $T$ at $(0,0)$.
    That is, there exists $\delta > 0$ such that for any $y_0, y_1 \in B_Y(0,\delta)$, there exists $u \in L^p((0,T);U)$ such that $\cS_T(y_0,u) = y_1$.
    Moreover, there exists $C > 0$ such that $\norm{u}_{L^p} \leq C (\norm{y_0}_Y + \norm{y_1}_Y)$.
\end{theorem}

\begin{proof}
    Let $E_1 := Y \times Y$, $E_2 := L^p((0,T);U)$ and $F := Y$.
    Define $f: E_1 \times E_2 \to F$ by
    \begin{equation}
        f((y_0, y_1), u) := \cS_T(y_0, u) - y_1.
    \end{equation}
    By \cref{def:nonlinear-io}, $f(0,0) = 0$.
    By the second assumption, $f$ is strongly Fréchet-differentiable at $(0,0)$.
    By the first assumption and our \cref{def:lin-exact-cont}, $D_2 f(0,0)$ has a bounded right inverse.
    Hence, by \cref{cor:TFI-surjective}, for $y_0, y_1 \in Y$ small enough, there exists $u \in L^p((0,T);U)$ such that $f((y_0,y_1),u) = 0$, i.e., $\cS_T(y_0,u) = y_1$.
    The Lipschitz-continuity of the inverse given by \cref{cor:TFI-surjective} gives the estimate of $\norm{u}_{L^p}$.
\end{proof}

\subsubsection{Example: linear test for ODEs}
\label{ex:ODE-abstract-nonlinear}

Let us revisit the linear test of \cref{sec:linear-test-ODEs} for ODEs using our abstract setting.

Let $n, m \in \N^*$ and $f_0, f_1, \dotsc, f_m \in C^1(\R^n;\R^n)$ globally Lipschitz and globally bounded vector fields with $f_0(0) = 0$.
We consider the control-affine system \eqref{eq:affine}.

Set $Y := \R^n$, $p \in [1,\infty]$ and $U := \R^m$.
For $t \geq 0$, $y_0 \in Y$ and $u \in L^p(\R_+;U)$, define $\cS_t(y_0,u) := y(t;u,y_0)$ with the notation of \cref{lem:WP}, which implies that $\cS$ satisfies \cref{def:nonlinear-io}.

Set $A := D f_0(0) \in \R^{n \times n}$ and $B := [f_1(0), \dotsc, f_m(0)] \in \R^{n \times m}$.
For $t \geq 0$, $y_0 \in Y$ and $u \in L^p(\R_+;U)$, define $\Sigma_t(y_0,u) := e^{tA} y_0 + \int_0^t e^{(t-s)A} B u(s) \dd s$.
It is straightforward to check that $\Sigma$ satisfies \cref{def:linear-io}.

Assume that the matrices $A$ and $B$ satisfy the rank condition \eqref{eq:kalman}.
Let $T > 0$.
By \cref{thm:kalman}, $\Sigma$ is exactly controllable in time $T$ in the sense of \cref{def:lin-exact-cont}.
By \cref{lem:ODE-C1}, $\cS_T$ is $C^1$ in a neighborhood of $(0,0)$ with $D\cS_T(0,0) = \Sigma_T$.

Thus \cref{thm:linear-to-nonlinear} applies and the control-affine system \eqref{eq:affine} is $L^p$-STLC in the sense of \cref{def:STLC} (as was proved in \cref{thm:ODE-linear}).

\subsubsection{Example: semilinear wave equation}
\label{sec:wave}

As a PDE example, let us study a 1D semilinear wave equation.
This is both an easy toy model, and relevant for many applications; see \cite[Chapter 14.1]{Whitham1999} or \cite[Section 2.1]{Rajaraman1987} for examples of various physical contexts in which such nonlinearities appear.
We consider:
\begin{equation}
    \label{eq:nonlinear-wave}
    \begin{cases}
        y_{tt} - y_{xx} = G(y) & \quad \text{in } (0,T) \times (0,\pi), \\
        y(t,0) = u(t) & \quad \text{on } (0,T), \\
        y(t,\pi) = 0 & \quad \text{on } (0,T)
    \end{cases}
\end{equation}
where the nonlinearity $G \in C^1(\R;\R)$ is such that $G(0) = G'(0) = 0$ and $|G'(s)| \leq |s|$ for all $s \in \R$.
We start with the linear theory before proving the exact controllability of \eqref{eq:nonlinear-wave} for $T \geq 2 \pi$.

Since \eqref{eq:nonlinear-wave} is a PDE of second order with respect to time, the state of the system is actually the pair $(y, y_t)$.
We therefore introduce the state space $Y := L^2(0,\pi) \times H^{-1}(0,\pi)$ and control space $U := \R$.
We will look for solutions in the solution space
\begin{equation}
    \label{eq:cY-wave}
    \cY := C^0([0,T];L^2(0,\pi)) \cap C^1([0,T];H^{-1}(0,\pi)).
\end{equation}
Linear well-posedness is standard (see \cite[Theorem 10.9.3 and Section 4.1]{TucsnakWeiss2009} or \cite[Theorem 6.6.1]{Raymond2016}).

\begin{lemma}
    \label{lem:linear-wave-WP}
    Let $T > 0$, $(y_0, y_0') \in Y$, $f \in L^1((0,T);H^{-1}(0,\pi))$ and $u \in L^2((0,T);\R)$.
    There exists a unique weak solution $y \in \cY$ to the following linear wave equation
    \begin{equation}
        \label{eq:linear-wave}
        \begin{cases}
            y_{tt} - y_{xx} = f & \quad \text{in } (0,T) \times (0,\pi), \\
            y(t,0) = u(t) & \quad \text{on } (0,T), \\
            y(t,\pi) = 0 & \quad \text{on } (0,T),
        \end{cases}
    \end{equation}
    with initial data $(y,y_t)(0) = (y_0,y_0')$.
    Moreover, there exists $C_T > 0$ such that
    \begin{equation}
        \label{eq:estimate-linear-wave}
        \norm{y}_{C^0 L^2} + \norm{y_t}_{C^0 H^{-1}} \leq C_T \left( \norm{f}_{L^1 H^{-1}} + \norm{y_0}_{L^2} + \norm{y_0'}_{H^{-1}} + \norm{u}_{L^2} \right).
    \end{equation}
    In particular, setting $\Sigma_s((y_0,y_0'),u) := (y,y_t)(s)$ for $s \in [0,T]$ defines an abstract linear control system in the sense of \cref{def:linear-io} with $p = 2$.
\end{lemma}

\begin{lemma}
    \label{lem:control-linear-wave}
    Let $T \geq 2 \pi$.
    There exists a bounded linear map $L : Y \to L^2((0,T);\R)$ such that, for all $(y_1, y_1') \in Y$, the control $u := L (y_1,y_1')$ is such that the solution to~\eqref{eq:linear-wave} with null initial data, control $u$ and null source term satisfies $(y,y_t)(T) = (y_1, y_1')$.
\end{lemma}

\begin{proof}
    We decompose the state on the Hilbert basis of $L^2(0,\pi)$ given by the eigenfunctions of the Dirichlet Laplacian: $e_k(x) := \sqrt{2/\pi} \sin(k x)$ for $k \in \N^*$.
    We write $y(t,x) = \sum_{k \geq 1} y_k(t) e_k(x)$.
    Taking the inner product of \eqref{eq:linear-wave} with $e_k$, we obtain the following cascade of independent harmonic oscillators:
    \begin{equation}
        \ddot{y}_k(t) + k^2 y_k(t) = b_k u(t), \quad \text{with } b_k := k \sqrt{\frac{2}{\pi}}.
    \end{equation}
    We introduce the complex unknown $z_k(t) := k y_k(t) + i \dot{y}_k(t)$ which satisfies $\dot{z}_k(t) = - i k z_k(t) + i b_k u(t)$ and $z_k(0) = 0$.
    Therefore
    \begin{equation}
        z_k(T) = i b_k \int_0^T e^{-ik(T-t)} u(t) \dd t.
    \end{equation}
    Let $T = 2\pi$ (if $T > 2\pi$, we extend $u$ by $0$ on $[0, T-2\pi]$).
    Hence, solving the control problem $y(T) = y_1$ and $y_t(T) = y_1'$ amounts to finding $u \in L^2((0,2\pi);\R)$ solving the moment problem:
    \begin{equation}
        \int_0^{2 \pi} e^{i k t} u(t) \dd t = \frac{e^{i k T}}{i b_k} (k y_{1,k} + i y'_{1,k}) =: a_k \in \C
    \end{equation}
    where $(y_{1,k})_k$ and $(y'_{1,k})_k$ are the coefficients of the target data $y_1$ and $y_1'$.
    Setting $a_0 := 0$ and $a_{-k} := \overline{a_k}$ for $k \in \N^*$ we can extend the moment problem to all $k \in \Z$.

    Using that the $(e^{i k t})_{k \in \Z}$ form an orthonormal basis of $L^2(0,2\pi)$, we set
    \begin{equation}
        u(t) := \frac{1}{2\pi} \sum_{j \in \Z} a_j e^{- i j t}.
    \end{equation}
    Then $u$ solves the moment problem by orthogonality of the basis.
    By Parseval's identity,
    \begin{equation}
        \norm{u}_{L^2}^2 = \frac{1}{2 \pi} \sum_{k \in \Z} |a_k|^2
        \leq C \norm{y_1}_{L^2}^2 + C \norm{y_1'}_{H^{-1}}^2
    \end{equation}
    which concludes the proof.
\end{proof}

\begin{remark}
    The proof of \cref{lem:control-linear-wave} relies on Parseval's identity because the PDE we are studying is exactly the wave equation $y_{tt} - y_{xx} = 0$, associated with eigenvalues $\lambda_k = k$.
    For more general hyperbolic equations, for example if one considers the ubiquitous Sine-Gordon model $G(y) = - \sin y$, one must study the linearized equation $y_{tt} - y_{xx} + y = 0$, associated with eigenvalues $\lambda_k = \sqrt{k^2+1}$.
    One loses the time-orthogonality of the $e^{i \lambda_k t}$ on $(0,2\pi)$, so one must replace Parseval's identity with more powerful tools (see e.g.\ \cite{Russell1967}), called Ingham inequalities (see \cite{Ingham1936}).
    A particularly useful version of Ingham's inequality in such contexts relating the minimal control time with the asymptotic gap of the eigenvalues is due to Haraux in \cite[Théorème 4]{Haraux1989}.
\end{remark}

We can now move to the nonlinear equation \eqref{eq:nonlinear-wave}.

\begin{lemma}
    Let $T > 0$.
    There exists $\delta > 0 $ such that, for all $(y_0, y_0') \in Y$ and $u \in L^2((0,T);\R)$ satisfying $\norm{y_0}_{L^2} + \norm{y_0'}_{H^{-1}} + \norm{u}_{L^2} \leq \delta$, there exists a unique weak solution $y \in \cY$ to \eqref{eq:nonlinear-wave} with initial data $(y,y_t)(0) = (y_0,y_0')$.
    It moreover satisfies
    \begin{equation}
        \label{eq:estimate-nonlinear-wave}
        \norm{y}_{C^0 L^2} + \norm{y_t}_{C^0 H^{-1}} \leq 2 C_T \left( \norm{y_0}_{L^2} + \norm{y_0'}_{H^{-1}} + \norm{u}_{L^2} \right).
    \end{equation}
\end{lemma}

\begin{proof}
    Fix $T > 0$.
    Let $\delta > 0$ to be chosen later.
    Let $(y_0, y_0') \in Y$ and $u \in L^2((0,T);\R)$ such that $\norm{y_0}_{L^2} + \norm{y_0'}_{H^{-1}} + \norm{u}_{L^2} \leq \delta$.
    We define a map $\Theta : \cY \to \cY$ as follows.
    Given $z \in \cY$, let $y \in \cY$ be the solution given by \cref{lem:linear-wave-WP} to \eqref{eq:linear-wave} for the source term $f := G(z)$.

    Using the assumptions on $G$, $|G(s)| \leq s^2$.
    Hence, one has
    \begin{equation}
        \norm{\Theta(z)}_{\cY} \leq C_T \left( \delta + \norm{G(z)}_{L^1 H^{-1}}\right)
        \leq C_T \left( \delta + T \norm{G(z)}_{L^\infty L^1} \right)
        \leq C_T \left( \delta + T \norm{z}_{L^\infty L^2}^2 \right).
    \end{equation}
    Thus, given $R > 0$, if $z \in B_{\cY}(0,R)$, $\norm{\Theta(z)}_{\cY} \leq C_T(\delta+TR^2)$.
    In particular, with $R := 1/(4 T C_T)$ and $\delta := R / (4 C_T)$, one has $\Theta(z) \in B_{\cY}(0,R/2)$, so $\Theta$ preserves $B_{\cY}(0,R)$.

    Given $z_1, z_2 \in \cY$, $\Theta(z_1) - \Theta(z_2)$ is a solution to \eqref{eq:linear-wave} with null initial data and null control, and a source term $G(z_1)-G(z_2)$.
    Using the assumptions on $G$, for any $s_1, s_2 \in \R$, one has $|G(s_1)-G(s_2)| \leq (|s_1|+|s_2|)|s_1-s_2|$.
    This entails that
    \begin{equation}
        \label{eq:G-H1}
        \begin{split}
            \norm{G(z_1)-G(z_2)}_{L^1(H^{-1})}
            & \leq T \norm{G(z_1)-G(z_2)}_{C^0(L^1)} \\
            & \leq T \sup_{t \in [0,T]} \int_0^\pi (|z_1(t)| + |z_2(t)|) |z_1(t)-z_2(t)| \\
            & \leq T \left( \norm{z_1}_{C^0 L^2} + \norm{z_2}_{C^0 L^2}\right) \norm{z_1-z_2}_{C^0 L^2}.
        \end{split}
    \end{equation}
    Hence $\norm{\Theta(z_1) - \Theta(z_2)}_{\cY} \leq 2 R T C_T \norm{z_1-z_2}_{\cY} \leq \frac 12 \norm{z_1-z_2}_{\cY}$ with our choice of $R$.

    Hence $\Theta$ is a contraction from $B_{\cY}(0,R)$ to itself.
    By the Banach fixed-point theorem, it admits a unique fixed-point $y \in \cY$, which is therefore the unique weak solution to \eqref{eq:nonlinear-wave}.

    Moreover, using $z_2 = 0$, we obtain from the previous estimate that, at the fixed-point $y = \Theta(y)$, $\norm{y-\Theta(0)}_{\cY} = \norm{\Theta(y)-\Theta(0)}_{\cY} \leq \frac 12\norm{y-0}_{\cY}$ so $\norm{y}_{\cY} \leq 2 \norm{\Theta(0)}_{\cY}$.
    Using the linear estimate \eqref{eq:estimate-linear-wave} for $\Theta(0)$ proves the nonlinear estimate \eqref{eq:estimate-nonlinear-wave}.
\end{proof}

\begin{lemma}
    \label{lem:frechet-nonlinear-wave}
    For $T > 0$, let $\Sigma_T := (y, y_t)(T)$ and $\cS_T := (y,y_t)(T)$ be the respective $Y$-valued solutions to \eqref{eq:linear-wave} and \eqref{eq:nonlinear-wave} at time $T$ (depending on $(y_0, y_0') \in Y$ and $u \in L^2((0,T);\R)$).
    Then $\cS_T$ is strongly Fréchet-differentiable at $0$ with derivative $\Sigma_T$.
\end{lemma}

\begin{proof}
    Recalling \cref{def:strong-Frechet} of strong Fréchet-differentiability, we need to prove that, as $(y_0,y_0')$, $(\tilde{y}_0,\tilde{y}_0') \in Y$ and $u, \tilde{u} \in L^2((0,T);\R)$ go to $0$,
    \begin{equation}
        \label{eq:wT-wave}
        \norm{w(T)}_{L^2 \times H^{-1}} = o \big(\norm{y_0-\tilde{y}_0}_{L^2} + \norm{y_0'-\tilde{y}_0'}_{H^{-1}} + \norm{u-\tilde{u}}_{L^2}\big)
    \end{equation}
    where $w(t) := \cS_t((y_0,y_0'),u) - \cS_t((\tilde{y}_0,\tilde{y}_0'),\tilde{u}) - \Sigma_t((y_0-\tilde{y}_0,y_0'-\tilde{y}_0'),u-\tilde{u})$.
    Thus $w \in \cY$ is a solution to \eqref{eq:linear-wave} with null initial and boundary data, and a source term $G(y) - G(\tilde{y})$, where we denote by $y, \tilde{y} \in \cY$ the two nonlinear solutions.
    As in \eqref{eq:G-H1}, this entails that
    \begin{equation}
        \norm{w}_{\cY}
        \leq T C_T (\norm{y}_{\cY} + \norm{\tilde{y}}_{\cY}) \norm{y-\tilde{y}}_{\cY}.
    \end{equation}
    Moreover, $y - \tilde{y}$ is a solution to \eqref{eq:linear-wave} with initial data $(y_0-\tilde{y}_0, y_0'-\tilde{y}_0')$, control $u - \tilde{u}$ and source term $G(y) - G(\tilde{y})$.
    Combining \eqref{eq:estimate-linear-wave} and \eqref{eq:G-H1},
    \begin{equation}
        \norm{y-\tilde{y}}_{\cY} \leq C_T \left( \norm{y_0-\tilde{y}_0}_{L^2} + \norm{y_0'-\tilde{y}_0'}_{H^{-1}} + \norm{u-\tilde{u}}_{L^2} + T (\norm{y}_{\cY} + \norm{\tilde{y}}_{\cY}) \norm{y-\tilde{y}}_{\cY} \right).
    \end{equation}
    From the nonlinear estimate \eqref{eq:estimate-nonlinear-wave}, $\norm{y}_{\cY}$ and $\norm{\tilde{y}}_{\cY}$ are small when $y_0, \tilde{y}_0, y_0', \tilde{y}_0', u, \tilde{u}$ go to $0$.
    Thus we can absorb the right-hand side in the left-hand side, and both estimates entail \eqref{eq:wT-wave}.
\end{proof}

We thus obtain the following local exact controllability result in time $T \geq 2 \pi$ for \eqref{eq:nonlinear-wave}.
Such a result is well-known.
Local controllability for general semilinear wave equations was proved in \cite{Fattorini1975} using the linear theory of \cite{Russell1967}, and the linear test inspired by \cite[Theorem~3]{Markus1965} for ODEs.
There are nonlinearities $G$ for which one can also prove global controllability results (see \cite{Zuazua1993}).

\begin{proposition}
    Let $T \geq 2 \pi$.
    There exists $\delta > 0$ such that, for any $(y_0, y_0'), (y_1,y_1') \in Y$ with $\norm{y_0}_{L^2} + \norm{y_0'}_{H^{-1}} + \norm{y_1}_{L^2} + \norm{y_1'}_{H^{-1}} \leq \delta$, there exists $u \in L^2((0,T);\R)$ such that the unique solution to \eqref{eq:nonlinear-wave} with initial data $(y_0,y_0')$ satisfies $(y,y_t)(T) = (y_1,y_1')$.
\end{proposition}

\begin{proof}
    We apply the general linear test of \cref{thm:linear-to-nonlinear} with $p = 2$.
    We proved that its two assumptions are satisfied in \cref{lem:control-linear-wave} and \cref{lem:frechet-nonlinear-wave}.
\end{proof}

\subsubsection{Other examples}

The linear test of \cref{thm:linear-to-nonlinear} (or variations of it) is ubiquitous in the local controllability literature, for many PDEs of different nature.
Some of the most frequent examples are the Korteweg--de Vries equation (see e.g.\ \cite{Rosier1997} for small-time local controllability when the length of the domain is not ``critical''), or the bilinear Schrödinger equation (see e.g.\ \cite{BeauchardLaurent2010}).

\newpage

\subsection{Transferring null controllability}
\label{sec:LTT}

We move on to the harder case where the underlying linear system is not exactly controllable as in \cref{sec:lin-2-nonlin-exact}, but only null-controllable.
We present a general abstract result.

\subsubsection{A general abstract result}

In 2013, Liu, Takahashi and Tucsnak introduced in \cite{LiuTakahashiTucsnak2013} a time-iteration argument which allows one to deduce the small-time local null controllability of a nonlinear system from the cost of the small-time null controllability of its linearization at the equilibrium.
This time iteration is also at the heart of a series of recent abstract results \cite{AlabauCannarsaUrbani2021,AlabauCannarsaUrbani2022,AlabauCannarsaUrbani2024} by Alabau-Boussouira, Cannarsa and Urbani concerning bilinear parabolic equations (see \cref{sec:LTT-biblio}).

We present here a streamlined reformulation of these general arguments, which bypasses the historical intermediate ``source term method'' and makes the approach directly applicable as a black box.
It is straightforward to use in concrete systems and avoids too many assumptions on the operators $A$ and $B$ of \eqref{eq:nonlin-system}: one only checks the two natural assumptions below.

\begin{assumption}
    \label{ass:LTT-cost}
    The abstract linear control system $\Sigma$ is \emph{small-time null controllable with exponential cost}: for all $T > 0$ and $y_0 \in Y$, there exists $u \in L^p((0,T);U)$ such that $\Sigma_T(y_0,u) = 0$ and
    \begin{equation}
        \label{eq:cost.linear}
        \norm{u}_{L^p} \leq C_L \exp \left( c T^{-\sigma} \right) \norm{y_0}_{Y}
    \end{equation}
    for some given constants $\sigma, c, C_L > 0$ (independent of $T$).
\end{assumption}

\begin{assumption}
    \label{ass:LTT-power}
    The abstract nonlinear control system $\cS$ is a \emph{nonlinear perturbation} of~$\Sigma$: there exist $\gamma > 1$ and $C_N > 0$ such that, for all $T > 0$, $y_0 \in Y$ and $u \in L^p((0,T);U)$,
    \begin{equation}
        \norm{\cS_T(y_0,u) - \Sigma_T(y_0,u)}_Y \leq C_N \left( \norm{y_0}_Y + \norm{u}_{L^p} \right)^{\gamma}.
    \end{equation}
\end{assumption}

\begin{theorem}
    \label{thm:LTT}
    Under \cref{ass:LTT-cost,ass:LTT-power}, the system $\cS$ is small-time locally null controllable: for every $T,\eta > 0$, there exists $\delta > 0$ such that, for all $y_0 \in Y$ with $\norm{y_0}_Y \leq \delta$, there exists $u \in L^p((0,T);U)$ with $\norm{u}_{L^p} \leq \eta$ such that $\cS_T(y_0,u) = 0$.
\end{theorem}

\begin{proof}
    The proof relies on a time iteration.
    Within each time interval, we apply the control obtained from the linear approximation.
    Reminiscent of Newton's method, the doubly-exponential convergence of the nonlinear error mitigates the small-time blow-up of the linear control cost.

    \medskip

    \step{Time iteration and linear step on each slice}
    Partition $[0,T]$ into subintervals $[T_j, T_{j+1}]$, where $T_0 = 0$ and $T_{j+1} = T_j + \tau_j$, with the constraint $\sum_{j \geq 0} \tau_j = T$.
    We choose
    \begin{equation}
        \label{eq:LTT-tau_j}
        \tau_j := T (1-\rho) \rho^j
        \quad \text{for} \quad
        \rho := \gamma^{-\frac{1}{\sigma+1}} \in (0,1).
    \end{equation}
    Given an initial state $y_0$, define inductively $y_{j+1} := \cS_{\tau_j}(y_j, u_j)$, where $u_j \in L^p((0,\tau_j);U)$ is the control given by \cref{ass:LTT-cost}, such that $\Sigma_{\tau_j}(y_j,u_j) = 0$ and
    \begin{equation}
        \label{eq:LTT-ujyj}
        \norm{u_j}_{L^p((0,\tau_j);U)} \leq C_L \exp (c \tau_j^{-\sigma}) \norm{y_j}_Y.
    \end{equation}

    \step{Stepwise nonlinear contraction}
    By \cref{ass:LTT-power} and \eqref{eq:LTT-ujyj},
    \begin{equation}
        \begin{split}
            \norm{y_{j+1}}
            = \norm{\cS_{\tau_j}(y_j,u_j) - \Sigma_{\tau_j}(y_j,u_j)}
            & \leq C_N (\norm{y_j}+\norm{u_j})^{\gamma} \\
            & \leq C_N (1+C_L)^{\gamma} \exp (c \gamma \tau_j^{-\sigma}) \norm{y_j}^{\gamma}.
        \end{split}
    \end{equation}
    Taking logs and dividing by $\gamma^{j+1}$,
    \begin{equation}
        \frac{\ln \norm{y_{j+1}}}{\gamma^{j+1}}
        \leq
        \frac{\ln \norm{y_j}}{\gamma^j} +
        \frac{\ln \left( C_N (1+C_L)^{\gamma} \right)}{\gamma^{j+1}} +
        \frac{c}{\gamma^j \tau_j^{\sigma}}.
    \end{equation}
    Summing this telescoping inequality over $j$ and recalling \eqref{eq:LTT-tau_j} leads to
    \begin{equation}
        \label{eq:yj-y0-sumtau}
        \frac{\ln \norm{y_j}}{\gamma^j}
        \leq \ln \norm{y_0} + C_T
        \quad \text{where} \quad
        C_T := \frac{\ln \left( C_N (1+C_L)^{\gamma} \right)}{\gamma-1} + \frac{c T^{-\sigma}}{(1-\rho)^{\sigma+1}}.
    \end{equation}
    Given a constant $M_T > 0$, when
    \begin{equation}
        \label{hyp:nl.local}
        \norm{y_0} \leq \exp \left(-M_T - C_T \right),
    \end{equation}
    estimate \eqref{eq:yj-y0-sumtau} yields
    \begin{equation}
        \label{eq:norm-yj-to-0}
        \norm{y_j} \leq \exp (- M_T \gamma^j) \underset{j \to \infty}{\longrightarrow} 0.
    \end{equation}

    \step{Conclusion}
    The overall control $u$ is the concatenation of the controls $u_j$ on $[T_j,T_{j+1}]$.
    We prove below that there exists $M_T > 0$ such that, under condition \eqref{hyp:nl.local}, for all $j \in \N$,
    \begin{equation}
        \label{eq:claim-LTT}
        \norm{u_j}_{L^p((0,\tau_j);U)} \leq 2^{-j} C_L \exp (C_T + c \tau_0^{-\sigma}) \norm{y_0}_Y.
    \end{equation}
    Hence, by the triangle inequality, $u \in L^p((0,T);U)$ and
    \begin{equation}
        \label{eq:cost-LTT}
        \norm{u}_{L^p((0,T);U)} \leq \sum_{j \geq 0} \norm{u_j}_{L^p((0,\tau_j);U)}
        \leq 2 C_L \exp (C_T + c \tau_0^{-\sigma}) \norm{y_0}_Y.
    \end{equation}
    Since $\cS$ is an abstract nonlinear control system, the associated solution $y(t) := \cS_t(y_0,u)$ belongs to $C^0([0,T];Y)$.
    By the composition property \eqref{eq:cS-compo}, $y(T_j) = y_j$.
    By \eqref{eq:norm-yj-to-0}, $y_j \to 0$.
    Since $T_j \to T$, by continuity, we conclude that $\cS_T(y_0,u) = y(T) = 0$.
    By \eqref{eq:cost-LTT}, choosing $\norm{y_0} \leq \delta$ small enough as in the theorem statement, we can ensure that the control is as small as desired.

    \step{Proof of \eqref{eq:claim-LTT}}
    We fix
    \begin{equation}
        M_T := \frac{\ln 2}{\gamma-1} + c \frac{\rho^{-\sigma}-1}{\gamma-1} \tau_0^{-\sigma} > 0.
    \end{equation}
    From \eqref{eq:LTT-ujyj} and \eqref{eq:yj-y0-sumtau}, under the smallness hypothesis \eqref{hyp:nl.local}, for all $j \geq 0$,
    \begin{equation}
        \label{eq:key-line}
        \begin{split}
            \ln \frac{\norm{u_j}}{\norm{y_0}}
            & \leq \ln C_L + c \tau_j^{-\sigma} + (\gamma^j-1)\ln\norm{y_0} + \gamma^j C_T \\
            & \leq \ln C_L + c \tau_j^{-\sigma} - (\gamma^j-1)M_T + C_T.
        \end{split}
    \end{equation}
    Then, using $\frac{\gamma^j-1}{\gamma-1}\geq j$ and the monotonicity of $s \mapsto \frac{s^j-1}{s-1}$ on $(1,\infty)$ (together with $\gamma>\rho^{-\sigma} = \gamma^{\frac{\sigma}{\sigma+1}}$),
    \begin{equation}
        (\gamma^j-1)M_T
        \geq j\ln2 + c \tau_0^{-\sigma} (\rho^{-\sigma j}-1)
        = j\ln2 + c (\tau_j^{-\sigma}-\tau_0^{-\sigma}).
    \end{equation}
    Plugging this into \eqref{eq:key-line} cancels the $c \tau_j^{-\sigma}$ term and yields
    \begin{equation}
        \ln\frac{\norm{u_j}}{\norm{y_0}} \leq - j \ln 2 + \ln C_L + c \tau_0^{-\sigma} + C_T,
    \end{equation}
    which proves \eqref{eq:claim-LTT} and completes the proof.
\end{proof}

\begin{remark}
    \label{rk:LTT}
    To avoid burdening the exposition, we stated \cref{thm:LTT} for globally well-posed time-invariant control systems.
    Nevertheless, it is clear from the proof that
    \begin{itemize}
        \item the exact same result holds even if the nonlinear system $\cS$ is only locally well-posed (for example for small enough times, or controls, or initial states),
        \item an analogous result holds \emph{mutatis mutandis} for time-dependent systems provided that \cref{ass:LTT-cost} and \cref{ass:LTT-power} hold uniformly in the considered time interval.
    \end{itemize}
\end{remark}

\begin{remark}
    As highlighted in \cite{AlabauCannarsaUrbani2022,AlabauCannarsaUrbani2024}, this construction is entirely explicit (no abstract fixed-point argument is used).
    Moreover, tracking the constants in the proof, one obtains a lower bound for the size of the neighborhood $\delta$ (which is $\exp(-M_T-C_T)$ in our proof) and an upper bound for the control cost $\norm{u} / \norm{y_0}$ (which is $2 C_L \exp(C_T + c \tau_0^{-\sigma})$ in our proof).
\end{remark}

\subsubsection{Historical perspective}
\label{sec:LTT-biblio}

We recall the origins of the method and its recent enhancements.

\paragraph{Source term method.}
The approach introduced in \cite{LiuTakahashiTucsnak2013} has been extensively applied to many PDEs and is commonly referred to as the ``source term method''.
In that paper, the strategy can be summarized as follows:
\begin{enumerate}
    \item \label{it:ltt-1}
    Starting from the small-time null controllability of a linear system $\dot{y} = Ay + Bu$, the authors use the abstract time-iteration argument described above to deduce that any system of the form $\dot{y} = Ay + Bu + f$, where $f$ is a given \emph{source term}, is also small-time null controllable, provided that $f \in \mathcal{F}$, a weighted functional space ensuring that $f$ decays very fast as $t \to T$.
    The associated control and state then also decay very fast as $t \to T$.
    Typically, the weights for the source term, the state and the control are of the form $\exp(1/(T-t))$.

    \item \label{it:ltt-2}
    For a sufficiently small initial state $y_0$, they define a map $\Theta : \mathcal{F} \to \mathcal{F}$ by $\Theta(f) := N(y,u)$, where $u$ is the control steering $y_0$ to $0$ in the presence of the source term $f$, $y$ is the associated solution and $N$ is the nonlinearity of \eqref{eq:nonlin-system}.
    They prove that $\Theta$ is a contraction that preserves a small ball in $\mathcal{F}$, so that the Banach fixed-point theorem provides a fixed point $f^\star$.
    This yields a solution of the nonlinear system $\dot{y} = Ay + Bu + N(y,u)$ with $y(0)=y_0$ and $y(T) = 0$.
\end{enumerate}

The idea of obtaining local controllability for a nonlinear system via a fixed-point argument based on linear controllability in the presence of rapidly decaying source terms had already been used in \cite{Imanuvilov2001,ImanuvilovTakahashi2007}.
In these works, the first step was carried out by means of Carleman estimates directly handling the source terms, without resorting to time iteration.

A key contribution of \cite{LiuTakahashiTucsnak2013} is therefore the time-iteration argument, which is independent of the specific method used to establish controllability of $\dot{y} = Ay + Bu$, and allows to automatically derive from it the controllability of $\dot{y} = Ay + Bu + f$, for rapidly decaying $f$.

Although \cite{LiuTakahashiTucsnak2013} has been widely used, its application can be somewhat cumbersome: for each concrete PDE system, one must introduce suitable weighted functional spaces (for the source terms, the controls, and the states), re-prove controllability in the presence of source terms, and re-establish the properties of the map $\Theta$ (for example, in \cref{sec:ex-bilin-heat}, we explain how we can avoid 4 pages of boilerplate of \cite{BeauchardMarbach2020}).

\medskip
The source term method has been very successfully used in many contexts.
To cite a few: for a Burgers equation with non-local viscosity in \cite{MicuTakahashi2018}; for systems of parabolic equations coupled by nonlinear boundary conditions coming from biology \cite{BhandariBoyer2024}; for free boundary problems, such as the Stefan problem in \cite{GeshkovskiMaity2023} or a nonlinear Burgers equation with free endpoint \cite{GeshkovskiZuazua2021}.
Outside of null controllability (recall \cref{rk:LTT}), it has also been adapted in \cite{DuprezLissy2022} to prove controllability to non-zero trajectories of a Fokker--Planck equation.

\paragraph{Direct time iteration for bilinear parabolic equations.}
The time iteration presented above and dating back to the source term method is also at the heart of the series of recent papers \cite{AlabauCannarsaUrbani2021,AlabauCannarsaUrbani2022,AlabauCannarsaUrbani2024} by Alabau-Boussouira, Cannarsa and Urbani.
These articles concern abstract parabolic systems of the form $\dot{y} + A y + u(t) B y = 0$, studied in the vicinity of non-zero ground states or eigenfunctions.
Up to a change of variables, they are included in the class \eqref{eq:nonlin-system}.

\begin{itemize}
    \item First, in \cite{AlabauCannarsaUrbani2021}, the authors construct controls on $t \in [0,\infty)$ such that the state converges superexponentially towards the target, by working on a sequence of uniform time intervals $[j T_0, (j+1)T_0]$.
    As in \cref{it:ltt-1} above, on each time interval, they apply the control given by the linear theory to the nonlinear system.
    By a nonlinear estimate similar to \cref{ass:LTT-power}, one roughly obtains $\norm{y((j+1)T_0)} \leq \norm{y(jT_0)}^2$, which entails the superexponential convergence towards zero.

    \item Then in \cite{AlabauCannarsaUrbani2022} for bounded $B$ and \cite{AlabauCannarsaUrbani2024} for unbounded $B$, they adapt this strategy to obtain small-time local null controllability.
    Here, as in \cref{it:ltt-1}, the sequence of time intervals is geometrically decreasing.
    The increasing cost of linear controllability of \cref{ass:LTT-cost} is mitigated by the superexponential quadratic decay.
\end{itemize}

An advantage of their presentation is that the Banach fixed-point argument of \cref{it:ltt-2} above is not necessary.
One way to understand it is that, as in our proof of \cref{thm:LTT}, the time iteration is performed directly with $f = N(y,u)$ on each time interval.

Their formulation has been adapted to a nonlinear heat equation in \cite[Section~4.2]{DucaPozzoliUrbani2025}.

\bigskip

By bypassing the intermediate source term argument altogether, and avoiding any direct assumption on the operators $A$ and $B$ of \eqref{eq:nonlin-system}, our reformulation of the Liu--Takahashi--Tucsnak and Alabau-Boussouira--Cannarsa--Urbani time iteration stated in \cref{thm:LTT} can be applied directly, requiring only the verification of the minimal and natural \cref{ass:LTT-cost,ass:LTT-power}.

We provide examples below.

\subsubsection{Example: a bilinear heat equation}
\label{sec:ex-bilin-heat}

In this paragraph we investigate the small-time local controllability of a bilinear heat equation introduced in \cite{BeauchardMarbach2020} (to which we refer for more details).
Using our abstract result, we give a much shorter proof of its controllability.
We consider the system
\begin{equation}
    \label{eq:heat-bilin}
    \begin{cases}
        y_t(t,x) - y_{xx}(t,x) = u(t) \Gamma[y(t)](x) & \text{in } (0,T) \times (0,\pi) \\
        y_x(t,0) = y_x(t,\pi) = 0 & \text{in } (0,T), \\
        y(0,x) = y_0(x) & \text{in } (0,\pi).
    \end{cases}
\end{equation}
where the nonlinearity $\Gamma \in C^1(H^1_N;H^{-1}_N)$ is globally Lipschitz, $u \in L^\infty((0,T);\R)$ and $y_0 \in L^2(0,\pi)$.
Using a Banach fixed-point argument, one easily proves (see \cite[Lemma 2.2]{BeauchardMarbach2020}) that system \eqref{eq:heat-bilin} is locally well-posed in the following sense.

\begin{lemma}
    For every $T > 0$, $u \in L^\infty((0,T);\R)$ and $y_0 \in L^2(0,\pi)$ small enough, system~\eqref{eq:heat-bilin} admits a unique solution $y \in C^0([0,T];L^2(0,\pi))$ which moreover satisfies
    \begin{equation}
        \label{eq:heat-bilin-est}
        \| y \|_{C^0(L^2)} + \| y \|_{L^2(H^1_N)} \leq C \| y_0 \|_{L^2} + C \|u \|_{L^\infty}.
    \end{equation}
\end{lemma}

Hence, setting $\cS_t := (y_0,u) \mapsto y(t)$ locally defines an abstract nonlinear control system in the sense of \cref{def:nonlinear-io} with $Y = L^2(0,\pi)$, $U = \R$ and $p = \infty$.
Its linearization at the null equilibrium is
\begin{equation}
    \label{eq:heat-bilin-lin}
    \begin{cases}
        z_t(t,x) - z_{xx}(t,x) = u(t) \mu(x) & \text{in } (0,T) \times (0,\pi) \\
        z_x(t,0) = z_x(t,\pi) = 0 & \text{in } (0,T), \\
        z(0,x) = z_0(x) & \text{in } (0,\pi),
    \end{cases}
\end{equation}
where $\mu := \Gamma[0] \in H^{-1}_N$.
Let $(\varphi_k)_{k \in \N}$ be the eigenfunctions of the Neumann Laplacian on $(0,\pi)$.
Under well-understood classical assumptions on the coefficients $\langle \mu, \varphi_k \rangle$, one knows from techniques dating back to \cite{FattoriniRussell1971} that \eqref{eq:heat-bilin-lin} is small-time null controllable (see \cite[Proposition 2.5]{BeauchardMarbach2020}).

\begin{lemma}
    \label{lem:muk-control}
    Assume that, for all $k \in \N$, $\langle \mu, \varphi_k \rangle \neq 0$ and that
    \begin{equation}
        \label{eq:muk-bilinear}
        \limsup_{k \to +\infty} \frac{- \log |\langle \mu, \varphi_k \rangle|}{k} < + \infty.
    \end{equation}
    Then there exists $C > 0$ such that, for any $T > 0$ and $z_0 \in L^2(0,\pi)$, there exists $u \in L^\infty(0,T)$ such that the associated solution to \eqref{eq:heat-bilin-lin} satisfies $z(T) = 0$ and one has
    \begin{equation}
        \norm{u}_{L^\infty} \leq C e^{C/T} \norm{z_0}_{L^2}.
    \end{equation}
\end{lemma}

It is easy to estimate the discrepancy between the nonlinear system and its linearization.

\begin{lemma}
    \label{lem:bilin-gap}
    There exists $C > 0$ such that, for all $T > 0$, $u \in L^\infty((0,T);\R)$ and $y_0 \in L^2(0,\pi)$ small enough, the unique solution $y$ to \eqref{eq:heat-bilin} satisfies
    \begin{equation}
        \norm{y(T)-z(T)}_{L^2} \leq C \left( \norm{y_0}_{L^2} + \norm{u}_{L^\infty} \right)^2,
    \end{equation}
    where $z$ denotes the unique solution to \eqref{eq:heat-bilin-lin}.
\end{lemma}

\begin{proof}
    Let $y, z \in C^0([0,T];L^2(0,\pi))$ be the solutions to \eqref{eq:heat-bilin} and \eqref{eq:heat-bilin-lin} with initial data $y_0$ and control $u$.
    Then $w := y - z$ satisfies
    \begin{equation}
        \begin{cases}
            w_t(t,x) - w_{xx}(t,x) = u(t) (\Gamma[y]-\Gamma[0]) & \text{in } (0,T) \times (0,\pi) \\
            w_x(t,0) = w_x(t,\pi) = 0 & \text{in } (0,T), \\
            w(0,x) = 0 & \text{in } (0,\pi).
        \end{cases}
    \end{equation}
    By classical estimates for the heat equation (see \cite[Lemma 2.1]{BeauchardMarbach2020}) analogue to \cref{prop:heat-f-WP} for Dirichlet boundary conditions, and the Lipschitz constant for $\Gamma$,
    \begin{equation}
        \| w \|_{C^0(L^2)}
        \leq C \| u (\Gamma[y]-\Gamma[0]) \|_{L^2(H^{-1}_N)}
        \leq C C_\Gamma \|u\|_{L^\infty} \|y\|_{L^2(H^1_N)}.
    \end{equation}
    Thus, recalling \eqref{eq:heat-bilin-est}, we obtain
    \begin{equation}
        \| y(T) - z(T) \|_{L^2} \leq \|w\|_{C^0(L^2)} \leq C' \| u \|_{L^\infty} \left( \|y_0\|_{L^2}+\|u\|_{L^\infty} \right),
    \end{equation}
    which is the claimed estimate.
\end{proof}

We thus recover the following theorem, already contained in \cite[Theorem 2]{BeauchardMarbach2020}.
In particular, we avoid the 4 pages of computations of \cite[Sections 2.4 and 2.5]{BeauchardMarbach2020} which correspond to the rather heavy setup of the typical implementations of the Liu--Takahashi--Tucsnak method (definition of weighted functional spaces, controllability despite a source term, Banach fixed-point argument).

\begin{theorem}
    Assume that $\mu = \Gamma[0]$ is such that, for all $k \in \N$, $\langle \mu, \varphi_k \rangle \neq 0$ and \eqref{eq:muk-bilinear}.
    Then system \eqref{eq:heat-bilin} is small-time locally null controllable: for all $T, \eta > 0$, there exists $\delta > 0$ such that, for all $y_0 \in L^2(0,\pi)$ with $\norm{y_0}_{L^2} \leq \delta$, there exists $u \in L^\infty(0,T)$ with $\norm{u}_{L^\infty} \leq \eta$ such that the unique associated solution to \eqref{eq:heat-bilin} satisfies $y(T) = 0$.
\end{theorem}

\begin{proof}
    Let $Y := L^2(0,\pi)$ be the state space, $U := \R$ the control space and $p := \infty$.
    By \cref{lem:muk-control}, the linearized system is small-time null controllable with exponential cost, so \cref{ass:LTT-cost} is satisfied with $\sigma = 1$.
    By \cref{lem:bilin-gap}, the nonlinear state is close to the linear one, so \cref{ass:LTT-power} is satisfied with $\gamma = 2$.
    Hence the result follows from \cref{thm:LTT}.
\end{proof}

\subsubsection{Example: the viscous Burgers equation}

In this paragraph, we investigate the small-time local null controllability of the viscous Burgers equation.
More precisely, for $T > 0$, we consider the system:
\begin{equation}
    \label{eq:Burgers}
    \begin{cases}
        y_t - y_{xx} + y y_x = u \mathbf{1}_\omega & \text{in } (0,T) \times (0,1), \\
        y(t,0) = y(t,1) = 0 & \text{in } (0,T), \\
        y(0,x) = y_0(x) & \text{in } (0,1),
    \end{cases}
\end{equation}
where $y_0 \in L^2(0,1)$ is the initial data and $u \in L^2((0,T) \times \omega)$ is a control, supported in a small fixed open set $\omega \subset (0,1)$.
In this paragraph, we will consider solutions to this system in the space
\begin{equation}
    \mathcal{Y} := L^2((0,T);H^1_0(0,1)) \cap C^0([0,T];L^2(0,1)) \cap H^1((0,T);H^{-1}(0,1)).
\end{equation}
By a Banach fixed-point argument, one can prove that this system is locally well-posed.
In fact, classical compactness arguments (see \cref{sec:WP-Burgers}) entail that it is globally well-posed, so it can be interpreted as an abstract nonlinear control system with $Y = L^2(0,1)$, $U = L^2(\omega)$ and $p = 2$.

\begin{proposition}
    \label{prop:WP-Burgers}
    For any $T > 0$, $y_0 \in L^2(0,1)$ and $u \in L^2((0,T) \times \omega)$, there exists a unique solution $y \in \mathcal{Y}$ to \eqref{eq:Burgers}.
\end{proposition}

The controllability of this system, which is a nice toy model for fluid mechanics, has been much studied.
In particular, despite it being globally well-posed, it is known that it is not small-time globally null controllable \cite{Diaz1996,FernandezCaraGuerrero2007,GuerreroImanuvilov2007}.
It is however known that it is small-time locally controllable to trajectories (including to the null equilibrium) since \cite[Theorem 5.1]{FursikovImanuvilov1995}, which relied on a different method (see \cref{sec:all-linear}).
We give a short proof as a direct application of our abstract result.

\begin{theorem}
    \label{thm:Burgers}
    System \eqref{eq:Burgers} is small-time locally null controllable: for all $T, \eta > 0$, there exists $\delta > 0$ such that, for all $y_0 \in L^2(0,1)$ with $\norm{y_0}_{L^2} \leq \delta$, there exists $u \in L^2((0,T)\times\omega)$ with $\norm{u}_{L^2} \leq \eta$ such that the unique associated solution to \eqref{eq:Burgers} satisfies $y(T) = 0$.
\end{theorem}

\begin{proof}
    Let $Y := L^2(0,1)$ be the state space and $U := L^2(\omega)$, the control space and $p := 2$.
    The linearized system at the null equilibrium is the controlled heat equation \eqref{eq:heat-1D-internal}.
    By \cref{prop:heat-control}, the heat equation is small-time null controllable with exponential cost, so the linearized system satisfies \cref{ass:LTT-cost} (with $\sigma = 1$).
    By \cref{lem:Burgers-quad} below, the discrepancy between the nonlinear system and its linearized version is quantified quadratically, so \cref{ass:LTT-power} is satisfied (with $\gamma = 2$).
    Hence, the conclusion follows directly from our black box statement of \cref{thm:LTT}.
\end{proof}

\begin{lemma}
    \label{lem:Burgers-quad}
    There exists $C > 0$ such that, for any $T > 0$, $u \in L^2((0,T) \times \omega)$ and $y_0 \in L^2(0,1)$, the unique solution $y$ to \eqref{eq:Burgers} satisfies
    \begin{equation}
        \label{eq:Burgers-quad-estimate}
        \norm{y(T)-z(T)}_{L^2} \leq C \left( \norm{y_0}_{L^2} + \norm{u}_{L^2} \right)^2,
    \end{equation}
    where $z$ denotes the unique solution to the linear heat equation \eqref{eq:heat-1D-internal}.
\end{lemma}

\begin{proof}
    The proof consists in classical PDE energy estimates: computing the equation satisfied by the difference $w(t) := y(t) - z(t)$, integrations by parts lead to the desired estimates.
    In fact, the claimed estimate is a direct consequence of the estimates used to prove the well-posedness of the system.
    More precisely, by \cref{lem:Burgers-a-priori} with $\lambda = 1$, we have
    \begin{equation}
        \norm{y(T)-z(T)}_{L^2}
        \leq \norm{y-z}_{\mathcal{Y}}
        \leq C \left(\norm{y_0}_{L^2} + \norm{u}_{L^2} \right)^2,
    \end{equation}
    which is the desired estimate.
\end{proof}

\subsection{Other approaches}

We give a short overview of two other methods which allow to prove the small-time local null controllability of a nonlinear system: an historic one, and a recent perspective.

\subsubsection{Uniform controllability of all linearizations}
\label{sec:all-linear}

Historically, the first results on the small-time local null controllability of nonlinear parabolic equations relied on the uniform controllability of a family of linearizations.

As an example, consider again the Burgers system \eqref{eq:Burgers}.
Now, instead of considering only its linearization at $0$ (which corresponds to the controlled heat equation \eqref{eq:heat-1D-internal}), fix a given trajectory $\bar{y}$ and consider its linearization at $\bar{y}$:
\begin{equation}
    \label{eq:Burgers-bary}
    \begin{cases}
        y_t - y_{xx} + \bar{y} y_x + \bar{y}_x y = u \mathbf{1}_\omega & \text{in } (0,T) \times (0,1), \\
        y(t,0) = y(t,1) = 0 & \text{in } (0,T), \\
        y(0,x) = y_0(x) & \text{in } (0,1).
    \end{cases}
\end{equation}
This is a family of linear time-dependent\footnote{A key advantage of the method exposed in \cref{sec:LTT} is that it avoids introducing such time-dependent systems, for which some techniques coming from spectral theory may be harder to apply.} control systems (depending on $\bar{y}$).

Given $T > 0$, using techniques first introduced in \cite[Section 4]{FursikovImanuvilov1995} based on Carleman estimates, one can then try to find a functional space $E_T$ and constants $C_T, \varepsilon_T > 0$ such that, for any $\bar{y} \in E_T$ with $\norm{\bar{y}}_{E_T} \leq \varepsilon_T$, for any initial state $y_0 \in L^2(0,1)$, there exists $u \in L^2((0,T) \times \omega)$ such that the associated solution to \eqref{eq:Burgers-bary} satisfies $y(T) = 0$ and $\norm{u}_{L^2} \leq C_T \norm{y_0}_{L^2}$.
A key point is that the control cost $C_T$ should be independent of the trajectory $\bar{y}$ around which the linearization is taken.

Provided that the solution space for the nonlinear system \eqref{eq:Burgers} is compactly embedded in $E_T$, one can then try to set up a fixed-point argument, associating with any given linearization $\bar{y}$ a trajectory of \eqref{eq:Burgers-bary} driving $y_0$ to $0$.

This technique was introduced in \cite{Imanuvilov1995} for semilinear heat equations and in \cite[Section 5]{FursikovImanuvilov1995} to prove \cref{thm:Burgers}, and extended in \cite{FernandezCara1997} to obtain a global result for semilinear heat equations.

\subsubsection{Exact controllability on the reachable space}
\label{sec:reachable}

In \cite{ErvedozaLeBalchTucsnak2022}, Ervedoza, Le Balc'h and Tucsnak give a modern perspective on the controllability of dissipative systems, by viewing them as being exactly controllable on their reachable space.

To present their ideas, we use the notations of \cref{sec:weiss} for abstract linear control systems.
For $T>0$, recall that the \emph{semigroup} $\T_T : Y \to Y$ and the \emph{control map} $\Phi_T : L^p((0,T);U) \to Y$ are defined by $\T_T y_0 := \Sigma_T(y_0, 0)$ and $\Phi_T u := \Sigma_T(0, u)$.
By linearity, $y(T) = \T_T y_0 + \Phi_T u$.

Fix $p = 2$, introduce the \emph{reachable space} $\Ran \Phi_T$ and endow it with the following norm:
\begin{equation}
    \label{eq:norm-EBT}
    \norm{z}_{\Ran \Phi_T} := \inf \left\{ \norm{u}_{L^2} \mid u \in L^2((0,T);U) \text{ s.t. } \Phi_T u = z \right\}.
\end{equation}
With this choice, $\Ran \Phi_T$ becomes a Hilbert space, and, when $\Sigma$ is small-time null controllable, $\Ran \Phi_T$ does not depend on $T$, and the associated norms are equivalent (see \cite[Section 2.1]{ErvedozaLeBalchTucsnak2022}).

Fix any $\tau > 0$ and set $Y_R := \Ran \Phi_\tau$ with the norm \eqref{eq:norm-EBT}.
The key perspective change of \cite{ErvedozaLeBalchTucsnak2022} is that, if one studies $\Sigma$ as an abstract linear control system with state space $Y_R$ instead of $Y$, it becomes small-time exactly controllable.
The key point is that the restriction of $\T$ to $Y_R$ forms a strongly continuous semigroup on $Y_R$ (see \cite[Theorem 2.1]{ErvedozaLeBalchTucsnak2022} or \cite[Lemma 2.2]{VanNeerven2005}).

This viewpoint is powerful for perturbation theory (see \cite[Section~4]{ErvedozaLeBalchTucsnak2022}).
In particular, it provides a natural route to treat nonlinear perturbations.
Indeed, if the nonlinearity is sufficiently mild, one can hope that the nonlinear system $\cS_T$ actually defines a $C^1$ map from $Y_R \times L^p((0,T);U)$ to $Y_R$, with derivative $\Sigma_T$ at $(0,0)$.
Then the small-time local exact controllability of the nonlinear system follows from the usual linear test argument of \cref{sec:lin-2-nonlin-exact}.

The main difficulty in implementing this program is that one needs a sufficiently explicit characterization of $Y_R = \Ran\Phi_\tau$ and of its norm in order to estimate the nonlinear terms in this topology.
This can already be delicate for toy linear parabolic models (see \cref{sec:why-null}).
Nevertheless, \cite[Section~10.2]{ErvedozaLeBalchTucsnak2022} successfully carries out this strategy for a semilinear heat equation, proving small-time exact controllability on the reachable space (see also \cite{ErvedozaTendaniSoler2025,LaurentRosier2020} for nonlinearities of Burgers type $y y_x$).

\newpage
\section{The tangent vectors method for truly nonlinear motions}

In this section, we investigate the local controllability of nonlinear systems at an equilibrium when the linearized system is not controllable.
The nonlinear dynamics play a crucial role, and one encounters phenomena sketched in \cref{sec:ode-nonlinear}.
The core idea of the \emph{tangent vectors method} is to show that one can generate infinitesimal approximate displacements in enough directions to span the whole space.
We focus on ODEs, though some ideas can be extended to PDEs (see \cref{sec:tangent-PDEs}).

Let $n \in \N^*$.
As in \cref{sec:affine}, we consider a control-affine system:
\begin{equation}
    \label{eq:affine-f1}
    \dot{y}(t) = f_0(y(t)) + u(t) f_1(y(t)),
\end{equation}
where $y(t) \in \R^n$ is the state, $u(t) \in \R$ the control, and $f_0, f_1 \in C^1(\R^n;\R^n)$ are globally Lipschitz, globally bounded vector fields with $f_0(0) = 0$.
We consider a single scalar control $u \in L^1((0,T);\R)$ for notational simplicity ($m = 1$ and $p = 1$ with the notations of \cref{sec:intro}).

We lighten the exposition, we will use the following notion of controllability\footnote{Here, we ask to bring any initial state back to the equilibrium.
One could ask the converse: go from the equilibrium towards any desired target.
Or both as in \cref{def:STLC}.
Typically, these properties are equivalent for ODEs (see \cite[Theorem 5.3]{Grasse1992} or \cite[Section 1.3]{AgrachevGamkrelidze1993_Semigroups}), but not for PDEs (recall \cref{sec:why-null}).}.

\begin{definition}[$L^1$-STLNC]
    \label{def:STLNC-0}
    We say that \eqref{eq:affine-f1} is \emph{small-time locally null controllable} when, for every $T > 0$, there exists $\delta > 0$ such that, for every initial data $y_0 \in B(0,\delta)$, there exists $u \in L^1((0,T);\R)$ with $\norm{u}_{L^1} \leq T$ such that $y(T;u,y_0) = 0$.
\end{definition}

\subsection{Motivating example}
\label{sec:jakub}

We consider the following classical system, due to Jakubczyk (see \cite[eq.\ (6.12)]{Sussmann1983}):
\begin{equation}
    \label{eq:jakubczyk}
    \begin{cases}
        \dot{y}_1 = u \\
        \dot{y}_2 = y_1 \\
        \dot{y}_3 = y_2^2 + y_1^3.
    \end{cases}
\end{equation}
Here $f_0(y) = (0, y_1, y_2^2 + y_1^3)$ so $f_0(0) = 0$ and $f_1(y) = (1,0,0)$.
One checks that the linearized system at $0$ is not controllable, as the third component is stationary.
So one cannot apply \cref{thm:ODE-linear}.

Let us consider various controls on $[0,T]$ and compute the associated final states starting from $y(0) = 0$, keeping only the largest term as $T \to 0$.
Fix a $\sigma \in \{ \pm 1 \}$.
\begin{itemize}
    \item With $u(t) = \sigma$, one gets $y(T) = \sigma T e_1 + \cO(T^2)$.
    \item With $u(t) = \sigma (-4 \frac{t}{T}+2)$, one gets $y_1(t) = \sigma T (-2(\frac t T)^2 + 2\frac t T)$ and $y_2(t) = \sigma T^2 ( - \frac23 (\frac tT)^3 + (\frac tT)^2)$.
    In particular, $y_1(T) = 0$, $y_2(T) = \frac{\sigma}{3} T^2$ and $y_3(T) = \cO(T^4)$.
    Thus $y(T) = \frac{\sigma}{3} T^2 e_2 + \cO(T^3)$.
    \item With $u(t) = \sigma (-10 (\frac{t}{T})^3 + 6 (\frac{t}{T})^2 + 3 \frac{t}{T} - 1)$, one gets $y_1(t) = \frac{\sigma}{2} T (- 5 (\frac{t}{T})^4 + 4 (\frac{t}{T})^3 + 3 (\frac{t}{T})^2 - 2 \frac{t}{T})$ and $y_2(t) = \frac{\sigma}{2} T^2(- (\frac{t}{T})^5 + (\frac{t}{T})^4 + (\frac{t}{T})^3 - (\frac{t}{T})^2)$.
    In particular, $y_1(T) = 0$, $y_2(T) = 0$, $\int_0^T y_2^2 = \cO(T^5)$ and $\int_0^T y_1^3 = \frac{29}{120120} \sigma^3 T^4$.
    Thus $y(T) = \frac{29}{120120} \sigma^3 T^4 e_3 + \cO(T^5)$.
\end{itemize}
Hence, we can move ``approximately'' as $T \to 0$, in the directions $\pm e_1$, $\pm e_2$ and $\pm e_3$.

\bigskip

\emph{Does this imply that system \eqref{eq:jakubczyk} is small-time locally null controllable?}
The following theory gives a formal general framework to answer this question.
The main difficulty is that, since the infinitesimal movements are truly nonlinear, one cannot simply use a linear combination of the associated controls to obtain an exact movement in an arbitrary direction.

\subsection{Concatenation estimates}

In the (nonlinear) context of tangent vectors, linear combinations of controls will be replaced by time concatenation of controls.
To analyze the effects of such operations, we will need the following classical Grönwall-type estimates (similar to \cite[Appendix]{Kawski1987_Survey}).

\begin{lemma}
    \label{lem:gronwall-single}
    There exists $C > 0$ such that, for any $T > 0$, $u \in L^1((0,T);\R)$ and $y_0 \in \R^n$,
    \begin{align}
        \label{eq:gronwall-1}
        | y(T ; u, y_0) - y(T ; u, 0) | & \leq \exp\left(C \left(T + \norm{u}_{L^1}\right)\right) |y_0|, \\
        \label{eq:gronwall-2}
        | y(T ; u, y_0) - y(T ; u, 0) - y_0 | & \leq C \left(T + \norm{u}_{L^1}\right) \exp\left(C \left(T + \norm{u}_{L^1}\right)\right) |y_0|.
    \end{align}
\end{lemma}

\begin{proof}
    Let $L_0$ and $L_1$ be the Lipschitz constants of $f_0$ and $f_1$ and define $C := \max \{ L_0, L_1 \}$.
    Estimate \eqref{eq:gronwall-1} follows from the Lipschitz dependence of the solution with respect to the initial data (see \cref{lem:EDO-Lip}).
    Moreover,
    \begin{equation}
        \begin{split}
            \abs{y(T;u,y_0) - y(T;u,0) - y_0}
            & = \Big|\int_0^T (f_0 + u(t) f_1)(y(t ; u, y_0)) - (f_0 + u(t) f_1)(y(t; u, 0)) \dd t \Big| \\
            & \leq \int_0^T (L_0 + L_1|u(t)|) \abs{y(t;u,y_0)-y(t;u,0)} \dd t \\
            & \leq \int_0^T (L_0 + L_1|u(t)|) \exp \left( L_0 T + L_1 \norm{u}_{L^1} \right) \abs{y_0} \dd t
        \end{split}
    \end{equation}
    using the bound of \cref{lem:EDO-Lip} for $\abs{y(t;u,y_0)-y(t;u,0)}$.
    We obtain \eqref{eq:gronwall-2}.
\end{proof}

\begin{lemma}
    \label{lem:gronwall-iterated}
    For any $N \geq 1$, any times $T_1, \dotsc, T_N > 0$, any controls $u_i \in L^1((0,T_i);\R)$ for $i=1,\dotsc,N$, and any initial condition $y_0 \in \R^n$, let $T = \sum_{i=1}^N T_i$ and let $u \in L^1((0,T);\R)$ be the concatenated control defined by $u := u_1 \diamond \dotsb \diamond u_N$.
    Then, the following estimate holds:
    \begin{equation}
        \label{eq:gronwall-iterated}
        \Big| y(T; u, y_0) - \Big(y_0 + \sum_{i=1}^N y(T_i; u_i, 0)\Big) \Big|
        \leq C R e^{C R} \Big(|y_0| + \sum_{i=1}^N |y(T_i; u_i, 0)| \Big),
    \end{equation}
    where $R := T + \norm{u}_{L^1(0,T)} = \sum_{i=1}^N (T_i + \norm{u_i}_{L^1(0,T_i)})$ and $C$ is the constant of \cref{lem:gronwall-single}.
\end{lemma}

\begin{proof}
    Let us define the flow for the $k$-th segment as $\Theta_k(z) := y(T_k; u_k, z)$ for $z \in \R^n$.
    Let us also define the sequence of states starting from $y_0$ by $y_k := \Theta_k(y_{k-1})$ for $k = 1, \dotsc, N$.
    The final state is thus $y_N = y(T; u, y_0)$.
    The quantity we want to estimate is the norm of the error vector
    \begin{equation}
        E_N := y_N - y_0 - \sum_{i=1}^N \Theta_i(0)
        = \sum_{k=1}^N \left( y_k - y_{k-1} - \Theta_k(0) \right).
    \end{equation}
    By \eqref{eq:gronwall-2}, each term of the sum is bounded as
    \begin{equation}
        \abs{E_N} \leq \sum_{k=1}^N C R_k e^{C R_k} | y_{k-1} |
    \end{equation}
    where $R_k := T_k + \norm{u_k}_{L^1}$.
    The next step is to control the growth of $|y_{k-1}|$.
    Using \eqref{eq:gronwall-1},
    \begin{equation}
        |y_k| = |\Theta_k(y_{k-1})| \leq |\Theta_k(y_{k-1}) - \Theta_k(0)| + |\Theta_k(0)|
        \leq |y_{k-1}| \exp(C R_k) + |\Theta_k(0)|.
    \end{equation}
    By unrolling this recurrence relation from $k-1$ down to $0$:
    \begin{equation}
        |y_{k-1}| \leq \Big( |y_0| + \sum_{i=1}^N |\Theta_i(0)| \Big) \exp (C R_1 + \dotsb + C R_{k-1}).
    \end{equation}
    Using that $R_i \geq 0$ and $R = \sum R_i$ proves \eqref{eq:gronwall-iterated} with the same constant $C$ as in \cref{lem:gronwall-single}.
\end{proof}

\subsection{Tangent vectors}
\label{sec:tangent-def}

We now give a formal definition for the approximate movements highlighted in \cref{sec:jakub}, in the context of general affine systems as in \cref{sec:affine}.
Many slightly different definitions exist in the literature (see \cite[Definition 1]{Frankowska1986}, \cite[Section 2]{BianchiniStefani1986}, \cite[Section 2]{Kawski1987_Survey} or \cite{HermesKawski1987}).

\begin{definition}
    \label{def:tangent}
    Let $\xi \in \R^n$ and $k \in \N^*$.
    We say that $\xi$ is a \emph{tangent vector of order $k$} when there exists a family of \emph{control variations} $\{u_T\}$ such that, for $0 < T \ll 1$, $u_T \in L^1((0,T);\R)$ with $\norm{u_T}_{L^1(0,T)} \leq T$ and the solution starting from the equilibrium satisfies
    \begin{equation}
        y(T;u_T, 0) = T^k \xi + \cO(T^{k+1}).
    \end{equation}
    Let $K_k \subset \R^n$ denote the set of all tangent vectors of order $k$ and $K_\infty := \cup_{k\geq1} K_k$.
\end{definition}

\begin{example}
    For system \eqref{eq:jakubczyk}, we proved in \cref{sec:jakub} that:
    \begin{itemize}
        \item $\pm e_1$ are tangent vectors of order $1$ with associated control variations $u_T(t) = \pm 1$,
        \item $\pm \frac 13 e_2$ are tangent vectors of order $2$ with $u_T(t) = \pm (- 4 \frac t T + 2)$,
        \item $\pm \frac{29}{120120} e_3$ are tangent vectors of order $4$ with $u_T(t) = \pm (- 10 (\frac t T)^3 + 6 (\frac t T)^2 + 3 \frac t T - 1)$.
    \end{itemize}
\end{example}

Since tangent vectors are motions which can be achieved in a truly nonlinear way, there is no reason for $K_k$ (or $K_\infty$) to be a vector space.
However, it enjoys the following elementary properties of monotonicity with respect to $k$ and nonlinear convexity, inspired by \cite[Section 2]{Kawski1987_Survey}.

\begin{lemma}
    \label{lem:K}
    The following facts hold:
    \begin{itemize}
        \item If $k < l$, then $K_k \subset K_l$.
        \item If $\xi_1, \dotsc, \xi_N \in K_k$ and $\lambda_1, \dotsc, \lambda_N \geq 0$ with $\lambda_1 + \dotsb + \lambda_N \leq 1$, then $\lambda_1^k \xi_1 + \dotsb + \lambda_N^k \xi_N \in K_k$.
    \end{itemize}
\end{lemma}

\begin{proof}
    We prove both statements independently.
    \begin{itemize}
        \item Let $k < l$ and $\xi \in K_k$.
        Let $\{u_T\}$ be a family of control variations for $\xi$.
        For $0 < T \ll 1$, let $v_T := 0_{T-T^\alpha} \diamond u_{T^\alpha}$ where $\alpha := \frac l k > 1$.
        Hence $\norm{v_T}_{L^1} \leq \norm{u_{T^\alpha}}_{L^1} \leq T^\alpha \leq T$.
        Since $f_0(0) = 0$,
        \begin{equation}
            y(T;v_T,0) = y(T^\alpha; u_{T^\alpha},0) = T^l \xi + \cO(T^{l+1}).
        \end{equation}

        \item Let $\xi_1, \dotsc, \xi_N \in K_k$ and $\lambda_1, \dotsc, \lambda_N \geq 0$ with $\lambda_1 + \dotsb + \lambda_N \leq 1$.
        Let $\{ u_{i,T} \}$ be families of control variations for the $\xi_i$.
        For $0 < T \ll 1$, let $T_i := \lambda_i T$ and $T^\star := T_1 + \dotsb + T_N \leq T$.
        Define $v_T := 0_{T-T^\star} \diamond u_{1,T_1} \diamond \dotsb \diamond u_{N,T_N}$.
        Then $\norm{v_T}_{L^1} \leq T^\star \leq T$.
        By \cref{lem:gronwall-iterated} with $y_0=0$,
        \begin{equation}
            \begin{split}
                y(T; v_T, 0) & = \left( \sum_{i=1}^N y(T_i; u_{i,T_i}, 0) \right) (1 + \cO(T)) \\
                & = \left( \sum_{i=1}^N \lambda_i^k T^k \xi_i + \cO(T^{k+1}) \right) (1 + \cO(T)) \\
                & = T^k \sum_{i=1}^N \lambda_i^k \xi_i + \cO(T^{k+1}),
            \end{split}
        \end{equation}
        which proves that $\lambda_1^k \xi_1 + \dotsb + \lambda_N^k \xi_N \in K_k$. \qedhere
    \end{itemize}
\end{proof}

These properties give an intuition of why time rescaling and time concatenation can replace linearity. 
In particular, as in \cite[Proposition 1.6]{Frankowska1987}, these properties entail the equivalence of the following formulations of the sufficient condition for controllability of \cref{thm:tangent}.

\begin{corollary}
    \label{cor:equiv-K}
    The following statements are equivalent:
    \begin{enumerate}
        \item \label{it:K-1}
        The origin $0 \in \R^n$ belongs to the interior of $K_\infty$, i.e.\ $0 \in \operatorname{int} K_\infty$.
        \item \label{it:K-2}
        There exist $\xi_1,\dotsc,\xi_N \in K_\infty$ such that $0 \in \operatorname{int} \operatorname{conv} \{ \xi_1,\dotsc,\xi_N \}$.
        \item \label{it:K-3}
        There exists $k \geq 1$ and $r > 0$ such that $\pm r e_1, \dotsc, \pm r e_n \in K_k$.
    \end{enumerate}
\end{corollary}

\begin{proof}
    We will prove the equivalence by showing that \cref{it:K-1} $\Rightarrow$ \cref{it:K-3} $\Rightarrow$ \cref{it:K-2} $\Rightarrow$ \cref{it:K-1}.

    \medskip
    \emph{Proof of \cref{it:K-1} $\Rightarrow$ \cref{it:K-3}.}
    Assume $0 \in \operatorname{int} K_\infty$.
    By definition, there exists $r > 0$ such that $B(0,2r) \subset K_\infty$.
    In particular $\pm r e_1, \dotsc, \pm r e_n \subset K_\infty$.
    Since $K_\infty = \cup_{k \geq 1} K_k$, by the first item in \cref{lem:K}, there exists $k \geq 1$ such that $\pm r e_1, \dotsc, \pm r e_n \in K_k$.

    \medskip
    \emph{Proof of \cref{it:K-3} $\Rightarrow$ \cref{it:K-2}.}
    Let $k \geq 1$ and $r > 0$ such that $\pm r e_1, \dotsc, \pm r e_n \in K_k$.
    For $1 \leq i \leq n$, let $\xi_i := + r e_i$ and $\xi_{n+i} := - r e_i$.
    Then $\xi_1, \dotsc, \xi_{2n} \in K_\infty$ and $\operatorname{conv} \{ \xi_1, \dotsc, \xi_{2n} \}$ is the cross-polytope $\{ |y_1| + \dotsb + |y_n| \leq r \}$.
    By Cauchy--Schwarz, it contains $B(0,r/\sqrt{n})$, so $0 \in \operatorname{int} \operatorname{conv} \{ \xi_1, \dotsc, \xi_{2n} \}$.

    \medskip
    \emph{Proof of \cref{it:K-2} $\Rightarrow$ \cref{it:K-1}.}
    Let $\xi_1, \dotsc, \xi_N \in K_\infty$ such that $0 \in \operatorname{int} \operatorname{conv} \{ \xi_1,\dotsc,\xi_N \}$.
    Since $K_\infty = \cup_{k \geq 1} K_k$, by the first item in \cref{lem:K}, there exists $k \geq 1$ such that $\xi_1, \dotsc, \xi_N \in K_k$.
    By definition, there exists $r > 0$ such that $B(0,r N^k) \subset \operatorname{conv} \{ \xi_1, \dotsc, \xi_N \}$.
    Hence, for any $\xi \in B(0,r)$, there exist $c_1, \dotsc, c_N$ with $0 \leq c_i \leq 1$ such that $\xi = c_1 \xi_1 + \dotsb + c_N \xi_N$ and $c_1 + \dotsb + c_N \leq N^{-k}$.
    Let $\lambda_i := c_i^{\frac 1k}$.
    By Hölder's inequality, $\sum \lambda_i \leq N^{\frac{k-1}{k}} (\sum c_i)^{\frac 1k} \leq 1$.
    By the second item of \cref{lem:K}, $\sum \lambda_i^k \xi_i = \xi \in K_k$.
    Hence $B(0,r) \subset K_k \subset K_\infty$.
\end{proof}

\subsection{Main theorem and time-iteration argument}
\label{sec:tangent-iteration}

Before proving the main result \cref{thm:tangent}, we give a result corresponding to a single step of the time iteration, where we construct a control bringing an initial state closer to the equilibrium.
Our proofs are inspired by \cite[Section 2]{Hermes1980}, \cite[Appendix]{Kawski1987_Survey} and \cite[Lemma 3.1]{Coron1992}.

\begin{proposition}
    \label{prop:single-move}
    Assume that there exists $k \geq 1$ and $r > 0$ such that $\pm r e_1, \dotsc, \pm r e_n \in K_k$.
    Then there exists $C > 0$ such that, for any $z \in B(0,r)$, there exists a time $0 \leq T_z \leq C |z|^{\frac 1k}$ and a control $u_z \in L^1((0,T_z);\R)$ with $\norm{u_z}_{L^1} \leq T_z$ such that $|y(T_z;u_z,z)| \leq C |z|^{1+\frac 1k}$.
\end{proposition}

\begin{proof}
    For $1 \leq i \leq n$, we use the notation $e_{n+i} := - e_i$.
    For $1 \leq i \leq 2n$, let $\{ u_{i,T} \}$ be a family of control variations associated with $r e_i$ for $0 < T \ll 1$.

    For $z \in B(0,r)$, we construct a control $u_z$ that generates a motion approximately equal to $-z$.
    We decompose $- z = z_1 r e_1 + \dotsb + z_{2 n} r e_{2 n}$ with $0 \leq z_i \leq |z|/r$ and, for $1 \leq i \leq n$, $z_i z_{n+i} = 0$.

    Let $T_i := z_i^{\frac 1k}$ and $T_z := \sum_{i=1}^{2 n} T_i \leq n (\frac{|z|}{r})^{\frac 1k}$.
    Define $u_z \in L^1((0,T_z);\R)$ with $\norm{u_z}_{L^1} \leq T_z$ by
    \begin{equation}
        u_z := u_{1,T_1} \diamond \dotsb \diamond u_{2 n,T_{2 n}}.
    \end{equation}
    Using \cref{lem:gronwall-iterated}, the final state $y(T_z; u_z, z)$ satisfies
    \begin{equation}
        \Big| y(T_z; u_z, z) - \Big(z + \sum_{i=1}^{2 n} y(T_i; u_{i,T_i}, 0)\Big) \Big| \leq C R_z e^{C R_z} \Big(|z| + \sum_{i=1}^{2 n} |y(T_i; u_{i,T_i}, 0)| \Big),
    \end{equation}
    where $R_z := T_z + \norm{u_z}_{L^1} \leq 2T_z = \cO(|z|^{\frac 1k})$.
    The displacement from the origin is
    \begin{equation}
        \sum_{i=1}^{2 n} y(T_i; u_{i,T_i}, 0)
        = \sum_{i=1}^{2 n} T_i^k (re_i) + \cO(T_i^{k+1})
        = -z + \cO(|z|^{1+\frac1k}).
    \end{equation}
    Combining these estimates, we obtain $y(T_z; u_z, z) = \cO(|z|^{1+\frac 1k})$, which concludes the proof.
\end{proof}

\bigskip 

\begin{theorem}
    \label{thm:tangent}
    Assume that $0 \in \operatorname{int} K_\infty$.
    Then system \eqref{eq:affine-f1} is $L^1$-STLNC.
\end{theorem}

\begin{proof}
    By \cref{cor:equiv-K}, there exists $k \geq 1$ and $r > 0$ such that $\pm r e_1, \dotsc, \pm r e_n \in K_k$.
    Thus \cref{prop:single-move} provides a map $\mathcal{S}(z) := y(T_z;u_z,z)$ which acts as a one-step contraction towards the origin.
    By iterating this map, we can steer any sufficiently small initial state to zero in finite time.
    The result $|\mathcal{S}(z)| = \cO(|z|^{1+\frac 1k})$ implies that for any contraction factor $q \in (0,1)$, there exists a radius $r_q > 0$ such that for all $z \in B(0,r_q)$, we have $|\mathcal{S}(z)| \leq q |z|$.
    Fix $q=2^{-k}$.

    Let $y_0$ be an initial state in $B(0,r_q)$.
    We define a sequence of states $(y_j)_{j \in \N}$ by the recurrence relation $y_{j+1} := \mathcal{S}(y_j)$ for $j \geq 0$.
    Then $|y_j| \leq 2^{-kj} |y_0|$ and $y_j \to 0$.
    The control driving $y_j$ to $y_{j+1}$ is $u_j := u_{y_j}$, applied over a time interval of duration $T_j := T_{y_j}$.
    \cref{prop:single-move} guarantees that $T_j \leq C|y_j|^{\frac 1k}$ for some constant $C>0$.
    The complete concatenated control is given formally by $u(y_0) := u_0 \diamond u_1 \diamond u_2 \diamond \dotsb$, and its total duration $T(y_0)$ is
    \begin{equation}
        \label{eq:tgt-T-conv}
        T(y_0) := \sum_{j=0}^\infty T_j \leq \sum_{j=0}^\infty C|y_j|^{\frac 1 k} \leq \sum_{j=0}^\infty C \left(2^{-kj} |y_0|\right)^{\frac 1 k} = 2 C|y_0|^{\frac 1 k}.
    \end{equation}
    Moreover, one has $\norm{u(y_0)}_{L^1(0,T(y_0))} \leq T(y_0)$.
    In particular, by \cref{lem:WP}, the associated solution $y \in C^0([0,T(y_0)];\R^n)$ is continuous.
    Since we had $y_j \to 0$, we get that $y(T(y_0);u(y_0),y_0) = 0$.

    As in \cref{def:STLNC-0}, let $T > 0$.
    Let $\delta := \min \{ r_q, (T/2C)^k \}$.
    For $y_0 \in B(0,\delta)$, we can perform the previous construction, and obtain that $\norm{u(y_0)}_{L^1} \leq T(y_0) \leq T$.
    We can therefore use the control $u := u(y_0) \diamond 0_{T-T(y_0)}$ to obtain the small-time local null controllability.
\end{proof}

\begin{example}
    Recall system \eqref{eq:jakubczyk} in $\R^3$.
    In \cref{sec:jakub}, we proved that $\pm e_1$, $\pm \frac 1 3 e_2$ and $\pm c e_3$ with $c \neq 0$ are tangent vectors (of respective orders 1, 2 and 4).
    Since $0 \in \operatorname{int} \operatorname{conv} \{ \pm e_1, \pm e_2, \pm e_3 \}$, \cref{cor:equiv-K} and \cref{thm:tangent} imply that this system is $L^1$-STLNC.
\end{example}

\subsection{How to construct tangent vectors?}

\cref{thm:tangent} deduces controllability from the existence of tangent vectors.
To apply it, one can either perform \emph{ad hoc} system-specific constructions (as in \cref{sec:jakub}), or use general recipes to construct tangent vectors systematically.
We give here a few examples of such arguments.
We start with the easiest systematic motion.

\begin{lemma}
    \label{lem:tangent-f1}
    One has: $\pm f_1(0)$ are tangent vectors of order $1$ with control variations $u_T^\pm(t) = \pm 1$.
\end{lemma}

\begin{proof}
    Using \cref{lem:taylor-uf1}, $y(T;u,0) = (\int_0^T u) f_1(0) + \mathcal{O}((T+\norm{u}_{L^1}) \norm{u}_{L^1})$.
    For $u_T^\pm = \pm 1$, we obtain $y(T;u_T^{\pm},0) = \pm T f_1(0) + \mathcal{O}(T^2)$.
    Thus $\pm f_1(0)$ are tangent vectors of order $1$.
\end{proof}

The following well-known result (see \cite[Theorem 4]{Frankowska1986}, \cite[Theorem 6]{HermesKawski1987} or \cite[Property (P2)]{BianchiniStefani1986}) is of great importance, and can be transferred to PDEs (see \cref{sec:tangent-PDEs}).

\begin{proposition}
    \label{prop:tangent-ad}
    Assume that $f_0 \in C^2(\R^n;\R^n)$.
    Set $A := D f_0(0)$.
    Let $k \in \N^*$.

    If $\pm \xi \in K_k$, then $\pm A \xi \in K_{2k+1}$.
\end{proposition}

\begin{proof}
    We prove that $+ A \xi \in K_{2k+1}$; the proof for $- A \xi$ is analogous.
    By assumption, there exist control variations $\{ u_T^\pm \}$ such that $y(T; u_T^\pm, 0) = \pm T^k \xi + \cO(T^{k+1})$.
    For $0 < T \ll 1$, define
    \begin{equation}
        v_T := u_{T^2}^+ \diamond 0_{T-2T^2} \diamond u_{T^2}^-.
    \end{equation}
    Denote by $y_1$ the state at time $T^2$ after the first pulse, by $y_2$ the state at time $T-T^2$ after the drift segment, and by $y_3$ the final state.
    Using the definition of $K_k$,
    \begin{equation}
        y_1 := y(T^2; u_{T^2}^+, 0) = T^{2k} \xi + \cO(T^{2k+2}).
    \end{equation}
    On $[T^2,T-T^2]$, the control $v_T$ is null and the system evolves according to $\dot{y} = f_0(y)$ from the initial data $y_1$ for a duration $\tau = T-2T^2$.
    By a Grönwall-type estimate (see \cref{lem:autonomous}),
    \begin{equation}
        y_2 = y(\tau;0,y_1) = y_1 + \tau A y_1 + \cO(\tau^2 |y_1| + \tau |y_1|^2) = T^{2k} \xi + T^{2k+1} A \xi + \cO(T^{2k+2}).
    \end{equation}
    Eventually, on the final time interval $[T-T^2,T]$, by \eqref{eq:gronwall-2},
    \begin{equation}
        y_3 = y(T^2;u_{T^2}^-,y_2)
        = y_2 + y(T^2;u^-_{T^2},0) + \mathcal{O}(T^2|y_2|).
    \end{equation}
    Hence, since $y(T^2;u^-_{T^2},0) = - T^{2k} \xi + \cO(T^{2k+1})$,
    \begin{equation}
        y(T;v_T,0) = y_3 = T^{2k+1} A \xi + \cO(T^{2k+2}),
    \end{equation}
    which proves that $+ A \xi \in K_{2k+1}$.
\end{proof}

A comforting application of \cref{prop:tangent-ad} is to recover the linear test of \cref{thm:ODE-linear}.

\begin{corollary}
    Assume that $A := Df_0(0)$ and $b := f_1(0)$ satisfy $\operatorname{rank} \{ b, Ab, \dotsc, A^{n-1}b \} = n$.
    Then the system \eqref{eq:affine-f1} is $L^1$-STLNC.
\end{corollary}

\begin{proof}
    By \cref{lem:tangent-f1}, $\pm f_1(0) = \pm b$ are tangent vectors.
    Applying \cref{prop:tangent-ad} repeatedly proves that $\pm A^k b$ are tangent vectors for all $k \geq 0$.
    The rank assumption thus implies that $0 \in \operatorname{int} K_\infty$, so \eqref{eq:affine-f1} is $L^1$-STLNC by \cref{thm:tangent}.
\end{proof}

Let us give an example of the interest of \cref{prop:tangent-ad} for a concrete system.

\begin{example}
    Consider the following variant of system \eqref{eq:jakubczyk}:
    \begin{equation}
        \label{eq:jakubczyk-int}
        \begin{cases}
            \dot{y}_1 = u \\
            \dot{y}_2 = y_1 \\
            \dot{y}_3 = y_2^2 + y_1^3 \\
            \dot{y}_4 = y_3.
        \end{cases}
    \end{equation}
    The computations of \cref{sec:jakub} remain valid and $\pm e_1$, $\pm e_2$ and $\pm e_3$ are tangent vectors of respective orders 1, 2 and 4 (the motion along $\pm e_3$ being truly nonlinear).
    Since $e_4 = Df_0(0) e_3$, \cref{prop:tangent-ad} implies that $\pm e_4$ are tangent vectors (of order 9).
    Hence \eqref{eq:jakubczyk-int} is $L^1$-STLNC by \cref{thm:tangent}.
    And we did not need to perform any additional explicit computation as in \cref{sec:jakub}.
\end{example}

\begin{example}[On the optimality of the order]
    \cref{prop:tangent-ad} proves that, given $\pm \xi \in K_k$, $\pm A \xi \in K_{2k+1}$.
    One can ask if $2k+1$ is the least possible order.
    Heuristically, one could think that $k+1$ could be enough.
    We give here a construction showing that the result fails with $2k-1$.

    Let $p \in \N$.
    Consider the system in $\R^4$:
    \begin{equation}
        \begin{cases}
            \dot{y}_1 = u \\
            \dot{y}_2 = y_1^{2p+1} \\
            \dot{y}_3 = y_2 \\
            \dot{y}_4 = y_1^{2p+2}.
        \end{cases}
    \end{equation}
    Using the control variations $u_T(t) = \pm (2 - 4 \frac t T)$, one can prove that $\pm r e_2$ are tangent vectors (for a given $r = r(p) > 0$) of order $k := 2p+2$ (with remainder $\cO(T^{2 p + 3})$).

    Assume that $\pm r' e_3 \in K_q$ for some $q \geq 0$.
    By Hölder's inequality,
    \begin{equation}
        \abs{y_3(T)}^{2p+2} = \left(\int_0^T (T-t) y_1^{2p+1}(t) \dd t\right)^{2p+2}
        \leq \left(\int_0^T (T-t)^{2p+2} \dd t\right) \left(\int_0^T y_1^{2p+2}(t) \dd t\right)^{2p+1}.
    \end{equation}
    Thus $\abs{y_3(T)}^{2p+2} \leq T^{2p+3} \abs{y_4(T)}^{2p+1}$.
    If $y(T) = r' T^q e_3 + \cO(T^{q+1})$, we obtain $T^{q (2p+2)} = \cO(T^{2p+3} T^{(q+1)(2p+1)})$ so $q \geq 4p+4 = 2k$.
    So one cannot have $\pm r' e_3 \in K_{2k-1}$.

    There is still a small gap between the positive result in $K_{2k+1}$ and the negative one in $K_{2k}$, which might be settled by more precise investigations.
\end{example}

The constructions presented above are very basic.
One knows many more involved and powerful ways to construct tangent vectors.
Their formulation often involves the language of Lie brackets and properties of the free Lie algebra.
We refer to Hermes \cite[Theorem~3.2]{Hermes1982}, Sussmann \cite[Theorem~7.3]{Sussmann1987}, Agrachev and Gamkrelidze \cite[Theorem~4]{AgrachevGamkrelidze1993_Semigroups} or Krastanov \cite[Theorem 2.7]{Krastanov2009}.

\subsection{About continuous controllability}

For applications, one may need the controls to depend continuously on the initial data.
This is key in order to go from controllability to stabilization (e.g.\ with continuous time-varying feedback laws), as identified by Coron in \cite{Coron1992}.
More precisely, let us introduce the following definition.

\begin{definition}[Continuous $L^1$-STLNC]
    \label{def:cont-STLC}
    We say that \eqref{eq:affine-f1} is \emph{continuously small-time locally null controllable} when, for every $T > 0$, there exists $\delta > 0$ and a continuous map $\mathcal{U} : B(0,\delta) \to L^1((0,T);\R)$ with $\mathcal{U}(0) = 0$, such that, for every initial data $y_0 \in B(0,\delta)$, $y(T;\mathcal{U}(y_0),y_0) = 0$.
\end{definition}

All known examples of $L^1$-STLNC systems are in fact continuously $L^1$-STLNC.
In fact, all known sufficient conditions for $L^1$-STLNC imply the continuous $L^1$-STLNC.
However, whether this holds for all systems is still a deep open question in control theory for ODEs.

\begin{open}
    Let $f_0, f_1 \in C^\omega(\R^n;\R^n)$ with $f_0(0) = 0$.
    Assume that~\eqref{eq:affine-f1} is $L^1$-STLNC.
    Is it also continuously $L^1$-STLNC in the sense of \cref{def:cont-STLC}.
\end{open}

Going back to our example \eqref{eq:jakubczyk}, we remark that our explicit control variations depended continuously on $T$.
This motivates the following definition.

\begin{definition}
    We say that $\xi \in \R^n$ is a \emph{continuous tangent vector of order $k \geq 1$} when it admits a family of control variations $\{ u_T \}$ as in \cref{def:tangent} such that the map $T \mapsto u_T$ is continuous from $(0,\infty)$ to $L^1((0,\infty);\R)$, where we extend each $u_T$ by $0$.
\end{definition}

Defining $K_k^{\mathrm{cont}}$ as the set of all continuous tangent vectors of order $k$ and $K_\infty^{\mathrm{cont}} := \cup_{k\geq1} K_k^{\mathrm{cont}}$, we obtain the same properties as in \cref{sec:tangent-def} (monotonicity with respect to $k$, nonlinear convexity, ...).
All our constructions are continuous.
In particular, one has the following result.

\begin{lemma}
    \label{prop:single-move-C0}
    Assume that there exists $k \geq 1$ and $r > 0$ such that $\pm r e_1, \dotsc, \pm r e_n \in K_k^{\mathrm{cont}}$.
    The maps $z \to T_z$ and $z \to u_z$ of \cref{prop:single-move} are continuous.
\end{lemma}

\begin{proof}
    All the steps in our construction are continuous.
\end{proof}

In the case of continuous tangent vectors, the continuity in \cref{prop:single-move-C0} gives rise to a new proof of \cref{thm:tangent}, which does not require a time iteration.

\begin{proposition}
    Assume that there exists $k \geq 1$ and $r > 0$ such that $\pm r e_1, \dotsc, \pm r e_n \in K_k^{\mathrm{cont}}$.
    The system \eqref{eq:affine-f1} is $L^1$-STLNC.
\end{proposition}

\begin{proof}
    Let $r,r_0 \in (0,1]$ to be chosen later.
    For a fixed $y_0 \in B(0,r_0)$, define a map $\psi_{y_0} : B(0,r) \to B(0,r)$ by $\psi_{y_0}(z) := y(T_z;u_z,y_0) + z$ where $T_z$ and $u_z$ are constructed in \cref{prop:single-move-C0}.
    Since $z \mapsto T_z$ and $z \mapsto u_z$ are continuous, by \cref{lem:EDO-bounds,lem:EDO-Lip}, $z \mapsto y(T_z;u_z,y_0)$ is continuous, and so is $\psi_{y_0}$.
    We claim that $\psi_{y_0}$ maps $B(0,r)$ to itself.
    By Brouwer's theorem, it thus admits a fixed-point $z^\star$.
    For this $z^\star$, one has $y(T_{z^\star};u_{z^\star},y_0) = 0$, so we have proved the $L^1$-STLNC.

    We now prove our claim.
    By \eqref{eq:gronwall-1} and \eqref{eq:gronwall-2},
    \begin{align}
        \abs{y(T_z;u_z,y_0) - y(T_z;u_z,0)} & \leq C \abs{y_0} \leq C r_0, \\
        \abs{y(T_z;u_z,z) - y(T_z;u_z,0) - z} & \leq C (T_z + \norm{u_z}_{L^1}) \abs{z} \leq C \abs{z}^{1+\frac 1k} \leq C r^{1+\frac 1k}.
    \end{align}
    Since, by \cref{prop:single-move}, $\abs{y(T_z;u_z,z)} \leq C \abs{z}^{1+\frac1k}$, we obtain that
    \begin{equation}
        \abs{\psi_{y_0}(z)} \leq C r_0 + C r^{\frac 1k} r.
    \end{equation}
    Choosing $r < 1/(2C)^k$ then $r_0 < r / (2C)$, we obtain that for any $y_0 \in B(0,r_0)$, $\psi_{y_0}$ maps $B(0,r)$ to itself, which concludes the proof.
\end{proof}

The previous proof avoids the time iteration, but it consumes continuous tangent vectors and only delivers (not necessarily continuous) $L^1$-STLNC.
This is inherent to the proof method.

\begin{remark}
    In the previous proof, we found $r, r_0 > 0$ and constructed a parametrized map $\psi_{y_0}(z) := y(T_z;u_z,y_0) + z$ such that, for each $y_0 \in B(0,r_0)$, $\psi_{y_0}$ is continuous and sends $B(0,r)$ to itself.
    For each fixed $y_0$, it thus admits a fixed-point.
    The question here is whether this parameter-dependent Brouwer fixed-point can be chosen to depend continuously on the parameter $y_0$.

    Generally, this is false.
    For $(y,z) \in [-1,1]^2$, define $f_y(z) := y + (1+\abs{y}) z$.
    For each $y \in [-1,1]$, $f_y$ is continuous on $[-1,1]$ and sends it to itself.
    But, for $y < 0$, the unique fixed-point is $-1$, while for $y > 0$, the unique fixed-point is $+1$.
    This prevents a continuous selection of the fixed-point.
\end{remark}

However, as noted in \cite[Proposition 11.22]{Coron2007}, repeating the main time iteration of \cref{sec:tangent-iteration}, one can transfer the continuity of tangent vectors to the controllability.

\begin{theorem}
    Assume that there exists $k \geq 1$ and $r > 0$ such that $\pm r e_1, \dotsc, \pm r e_n \in K_k^{\mathrm{cont}}$.
    Then system \eqref{eq:affine-f1} is continuously $L^1$-STLNC.
\end{theorem}

\begin{proof}
    As in the proof of \cref{thm:tangent}, let $r > 0$ be such that the map $z \mapsto \mathcal{S}(z) := y(T_z;u_z,z)$, which is continuous by \cref{prop:single-move-C0}, satisfies $\abs{\mathcal{S}(z)} \leq 2^{-k} \abs{z}$ for all $z \in B(0,r)$.
    Given $y_0 \in B(0,r)$, define the sequences $y_j = \mathcal{S}^j(y_0)$ as in the proof of \cref{thm:tangent}.
    For each finite $j$, $y_0 \mapsto y_j$ is continuous.
    Since the series of associated times and controls converges uniformly for $y_0 \in B(0,r)$ (see \eqref{eq:tgt-T-conv}), we conclude that $y_0 \mapsto T(y_0)$ and $y_0 \mapsto u(y_0)$ are continuous.
\end{proof}

\subsection{Quantifying controllability}

The relation between the time, the size of the control and the size of the achieved movement is more difficult in this nonlinear context than in \cref{sec:seidman} for linear systems.
However, it gives rise to interesting questions.
For $T > 0$, let us define
\begin{equation}
    \mathcal{N}_T := \left\{ y_0 \in \R^n \mid \exists u \in L^1((0,T);\R), \norm{u}_{L^1} \leq T, \enskip y(T;u,y_0) = 0 \right\}.
\end{equation}
Hence $\mathcal{N}_T$ represents the set of initial states which can be brought back to $0$ in time $T$.

With this notation, $L^1$-STLNC in the sense of \cref{def:STLNC-0} can be equivalently stated as: ``for all $T > 0$, there exists $\delta_T > 0$ such that $B(0,\delta_T) \subset \mathcal{N}_T$''.

Inspecting the end of the proof of \cref{thm:tangent} (and the definition of $\delta$ given there), we actually proved the following stronger statement.

\begin{theorem}
    If $0 \in \operatorname{int} K_k$, there exists $c > 0$ such that, for all $0 < T \leq 1$, $B(0,c T^k) \subset \mathcal{N}_T$.
\end{theorem}

Hence, controllability obtained from tangent vectors is automatically ``quantified'', in the sense that we have a polynomial estimate on $\delta_T$ as $T \to 0$.
In fact, all known $L^1$-STLNC systems exhibit this property.
This leads to the following deep open question.

\begin{open}
    Assume that system \eqref{eq:affine-f1} is $L^1$-STLNC.
    Does there exist $c > 0$ and $k \in \N^*$ such that, for all $0 < T \leq 1$, $B(0,c T^k) \subset \mathcal{N}_T$?
\end{open}

There are many equivalent ways to present this issue.
From an optimal control perspective, a reformulation is to ask whether small-time local controllability implies that the value function of the minimum-time problem is Hölder continuous at the origin (see \cite[Open Problem 1]{Kawski2006}).

\subsection{Generalization to PDEs}
\label{sec:tangent-PDEs}

For PDEs, there is a wide range of examples of situations in which the linearized system at an equilibrium ``misses'' a finite number of directions (and is therefore not controllable).

\paragraph{Obstructions.}
In some cases, one can prove that the nonlinear system is not controllable either, due to the presence of some (typically quadratic) ``drift'', reminiscent of \cref{ex:obs-W1} for ODEs, where some given projection of the state cannot be made negative.
This has for example been observed for Schrödinger with Dirichlet \cite{Bournissou2023_Quad,Coron2006} or Neumann \cite{BeauchardMarbachPerrin2025} boundary conditions, for a water-tank system \cite{CoronKoenigNguyen2024}, for Korteweg--de Vries \cite{CoronKoenigNguyen2022,NiuXiang2025}, for Burgers \cite{Marbach2018}.

\paragraph{Tangent vectors.}
In other cases, one can use the nonlinear terms to recover the finite number of directions missing at the linear order, as in \cref{ex:cubic} with a purely cubic term, or as in \cref{sec:jakub} despite a competition with a quadratic one.
One can use manipulations reminiscent of tangent vectors for ODEs presented in this section.

For PDEs, the current works use slightly different notions of tangent vectors, in order to avoid the exploding cost $\exp(1/T)$ of the control of the linear part (see for example \cite[Section 2]{Bournissou2024}).

Nevertheless, some ideas are very similar.
For example, the idea of \cref{prop:tangent-ad} to generate new accessible directions has been adapted to several PDEs (see \cite[Section 6.4]{Bournissou2024} or \cite[Proposition 4.8]{Gherdaoui2025} for Schrödinger with Dirichlet boundary conditions, \cite[Section 4.3]{BeauchardMarbachPerrin2025} for Schrödinger with Neumann boundary conditions, or \cite[Section 4.3]{ChowdhuryErvedoza2019} for Navier--Stokes).

\newpage
\appendix
\section{Some inequalities}

\subsection{Remez inequality in a segment}

Our proofs of both spectral inequalities of \cref{p:turan,thm:LS} rely fundamentally on \emph{Remez inequalities} (see the historical proof~\cite{Remes1936} with optimal constant attained by Chebyshev polynomials).
The interest of this inequality for control theory is well-identified (see e.g.\ \cite{BeauchardJamingPravdaStarov2021,GreenLeBalchMartinOrsoni2025,HuangWangWang2024}).

We give here a direct proof similar to \cite[Proposition 3.4]{GreenLeBalchMartinOrsoni2025} with a suboptimal constant, but valid for all regular functions.
The usual polynomial case is of course the particular case with vanishing derivative of order $n+1$.

\begin{lemma}
    \label{lem:remez-f}
    Let $\omega \subset [0,1]$ be measurable with $|\omega| > 0$.
    For any $n \in \N$ and $f \in C^{n+1}([0,1];\C)$,
    \begin{equation}
        \label{eq:remez-f}
        \sup_{[0,1]} |f| \leq \left( \frac{8e}{|\omega|} \right)^n \sup_{\omega} |f| + \frac{1}{(n+1)!} \sup_{[0,1]} | f^{(n+1)} |.
    \end{equation}
\end{lemma}

\begin{proof}
    Let $\delta := \frac{|\omega|}{2} \frac{1}{n+1}$.
    We call a finite set $X \subset \omega$ ``$\delta$-separated'' if $|x-x'| \geq \delta$ for all $x \neq x' \in X$.
    Take a maximal (for inclusion) $\delta$-separated set $X$ (for example constructed greedily starting from any $\bar{x} \in \omega$).
    By maximality, $\omega \subset \cup_{x \in X} [x - \delta, x+\delta]$.
    Comparing measures, we obtain that $X$ contains at least $n+1$ elements.
    Thus we found $x_0 < x_1 < \dotsb < x_n \in \omega$ such that $|x_i-x_j| \geq \delta |i-j|$.

    Write the Lagrange interpolation polynomial for $f$ at these nodes:
    \begin{equation}
        P_f(x) := \sum_{j=0}^n f(x_j) \prod_{k \neq j} \frac{x-x_k}{x_j-x_k}.
    \end{equation}
    For all $x \in [0,1]$, $|x-x_k| \leq 1$ and thus, using $n! \geq (n/e)^n$,
    \begin{equation}
        \frac{|P_f(x)|}{\sup_\omega |f|}
        \leq \sum_{j=0}^n \prod_{k \neq j} \frac{1}{\delta |j-k|}
        = \frac{1}{\delta^n} \sum_{j=0}^n \frac{1}{j! (n-j)!}
        = \frac{1}{\delta^n} \frac{2^n}{n!}
        \leq \left( \frac{8 e}{|\omega|} \right)^n.
    \end{equation}
    Finally, by the Lagrange remainder formula, for each $x \in [0,1]$, there exists $\xi_x \in [0,1]$ such that
    \begin{equation}
        f(x) - P_f(x) = \frac{f^{(n+1)}(\xi_x)}{(n+1)!} \prod_{j = 0}^n (x-x_j).
    \end{equation}
    Using again that $|x-x_j| \leq 1$, this concludes the proof of \eqref{eq:remez-f}.
\end{proof}

\subsection{Tur\'an inequality}
\label{s:turan}

\begin{proof}[Proof of \cref{p:turan}]
    Using the substitution $y = \cos(\pi x)$, write $f(x) = \sin(\pi x) P(\cos (\pi x))$ for a polynomial $P \in \R_{n-1}[y]$.
    Note that $\norm{f}_{L^2(0,1)} \leq \sup_{[-1,1]} |P|$.

    Given $\omega \subset [0,1]$ with $|\omega| = 3 \delta > 0$, set $\omega' := \omega \cap [\delta,1-\delta]$.
    Hence $|\omega'| \geq \delta$.
    The map $\phi(x) = \cos(\pi x)$ is a diffeomorphism from $[\delta, 1-\delta]$ onto its image.
    On $[\delta,1-\delta]$, we have $|\phi'| \geq 2 \pi \delta$.
    Let $E := \phi(\omega')$.
    It follows that $|E| \geq 2 \pi \delta^2$.
    Let $Z := \{ y \in E ; |P(y)|^2 \leq 2 \norm{P}_{L^2(E)}^2 / |E| \}$.
    One has $|Z| \geq |E| / 2 \geq \pi \delta^2$.
    By the Remez inequality of \cref{lem:remez-f} applied to $P$ and $Z$,
    \begin{equation}
        \sup_{[-1,1]} |P|
        \leq \left(\frac{8e}{|Z|}\right)^{n-1} \sup_Z |P|
        \leq \left(\frac{8e}{\pi \delta^2}\right)^{n-1} \frac{1}{\pi^{\frac12} \delta} \norm{P}_{L^2(E)}.
    \end{equation}
    Finally, the change of variable $y = \cos(\pi x)$ on $\omega'$ yields
    \begin{equation}
        \int_{E} P^2(y) \dd y
        = \pi \int_{\omega'} \frac{f^2(x)}{\sin(\pi x)} \dd x
        \leq \frac{\pi}{\sin \pi \delta} \norm{f}_{L^2(\omega)}^2,
    \end{equation}
    which concludes the proof.
\end{proof}

\subsection{Remez inequality in multiple variables}

The one-dimensional Remez inequality of \cref{lem:remez-f} easily generalizes to higher dimensions.

\begin{lemma}
    \label{lem:remez-f-Rd}
    Let $\omega \subset [0,1]^d$ be measurable with $|\omega| > 0$.
    For any $n \in \N$ and $f \in C^{n+1}([0,1]^d;\C)$,
    \begin{equation}
        \label{eq:remez-f-Rd}
        \sup_{[0,1]^d} |f|
        \leq \left(\frac{8 e d}{|\omega|}\right)^n \sup_{\omega} |f| + d^{\frac{n+1}{2}} \sum_{|\alpha| = n+1} \sup_{[0,1]^d} \frac{| D^\alpha f |}{\alpha!}.
    \end{equation}
\end{lemma}

\begin{proof}
    Let $x_0$ be a point where $|f|$ attains its maximum $M$ on $[0,1]^d$, $m := \sup_\omega |f|$ and $M'$ the second term in the right-hand side of \eqref{eq:remez-f-Rd}.
    Assume without loss of generality that $M > M'$.

    Using spherical coordinates centered at $x_0$, let $L_\theta$ be the ray from $x_0$ to the boundary of $[0,1]^d$ along a direction $v_\theta \in S^{d-1}$ and $0 \leq R_\theta \leq \sqrt{d}$ its length.
    For $r \in [0,R_\theta]$, let $g_\theta(r) := f(x_0 + r \theta)$.
    By the (rescaled) 1D Remez inequality of \cref{lem:remez-f} applied to $g_\theta$,
    \begin{equation}
        M = |g_\theta(0)| 
        \leq \left(\frac{8 e d R_\theta}{|L_\theta \cap \omega|_1}\right)^n \sup_{L_\theta \cap \omega} |f| + \frac{R_\theta^{n+1}}{(n+1)!} \sup_{[0,R_\theta]} | g_\theta^{(n+1)} | 
        \leq \left(\frac{8 e d R_\theta}{|L_\theta \cap \omega|_1}\right)^n m + M'
    \end{equation}
    Rearranging this inequality, we obtain $|L_\theta \cap \omega|_1 \leq 8 e R_\theta (m/(M-M'))^{\frac 1 n}$.
    
    Since $r \leq R_\theta$ along $L_\theta$,
    \begin{equation}
        |\omega|
        = \int_{S^{d-1}} \int_{L_\theta \cap \omega} r^{d-1} \dd r \dd \sigma
        \leq \int_{S^{d-1}} R_\theta^{d-1} |L_\theta \cap \omega|_1 \dd \sigma.
    \end{equation}
    Substituting the 1D bound and noting that $\int_{S^{d-1}} R_\theta^d \dd \sigma = d \cdot \operatorname{Vol}([0,1]^d) = d$, we obtain:
    \begin{equation}
        |\omega|
        \leq 8e \left(\frac{m}{M-M'}\right)^{\frac{1}{n}} \int_{S^{d-1}} R_\theta^d \dd \sigma
        = 8e d \left(\frac{m}{M-M'}\right)^{\frac{1}{n}}.
    \end{equation}
    Rearranging for $M$ yields the result.
\end{proof}

\begin{lemma}
    \label{cor:remez-f-Rd-L2}
    Let $\omega \subset [0,1]^d$ be measurable with $|\omega| > 0$.
    For any $n \in \N^*$ and $f \in C^n([0,1]^d;\C)$,
    \begin{equation}
        \sup_{[0,1]^d} |f|
        \leq \left(\frac{16 e d}{|\omega|}\right)^n \norm{f}_{L^2(\omega)} + d^{\frac{n}{2}} \sum_{|\alpha| = n} \sup_{[0,1]^d} \frac{| D^\alpha f |}{\alpha!}.
    \end{equation}
\end{lemma}

\begin{proof}
    Let $M := \sqrt{2} |\omega|^{-\frac12} \norm{f}_{L^2(\omega)}$.
    We apply \cref{lem:remez-f-Rd} with $n \gets n-1$ and a modified subset $\omega' := \left\{ x \in \omega ; |f(x)| \leq M \right\}$.
    First $\sup_{\omega'} |f| \leq M$.
    Second, since $\frac 12 |\omega| M^2 = \int_\omega f^2 \geq \int_{\omega \setminus \omega'} f^2 \geq (|\omega|-|\omega'|) M^2$.
    Hence $|\omega'| \geq \frac 12 |\omega|$.
    Combining these estimates concludes the proof.
\end{proof}

\subsection{Bernstein inequality}

For $f \in L^2(\R^d;\C)$ and $\xi \in \R^d$, let $\widehat{f}(\xi) := \int_{\R^d} f(x) e^{- i x \cdot \xi} \dd x$.
We use the standard multi--index notation: for $\alpha = (\alpha_1,\dots,\alpha_d) \in \N^d$ we set $|\alpha| := \alpha_1 + \cdots + \alpha_d$ and $D^\alpha := \partial_{x_1}^{\alpha_1} \cdots \partial_{x_d}^{\alpha_d}$.
For $n \in \N$,
\begin{equation}
    \label{eq:cards}
    \# \left\{ \alpha \in \N^d ; |\alpha| \leq n \right\} \leq (n+1)^d.
\end{equation}

\begin{lemma}
    \label{lem:bernstein-Rd}
    Let $f \in L^2(\R^d)$ with $\supp \widehat{f} \subset B_{\R^d}(0,N)$.
    For any multi--index $\alpha \in \N^d$,
    \begin{equation}
        \norm{D^\alpha f}_{L^2(\R^d)}
        \leq N^{|\alpha|} \norm{f}_{L^2(\R^d)}.
    \end{equation}
\end{lemma}

\begin{proof}
    Since $\widehat{D^\alpha f}(\xi) = (i\xi)^\alpha \widehat{f}(\xi)$, the estimate follows from Plancherel's theorem.
\end{proof}

\subsection{Logvinenko--Sereda inequality}
\label{sec:LS-proof}

We prove \cref{thm:LS}.
The thickness condition on the observation subset $\omega$ hints at partitioning $\R^d$ into cells of equal sizes.
Within each cell, we apply our Remez inequality.
We conclude using the fact that the derivatives of band-limited functions grow only geometrically.
Unlike Kovrijkine's proof of \cite{Kovrijkine2001} or \cite{WangWangZhangZhang2019}, we neither distinguish good and bad cells, nor use Harnack's inequality.

\begin{proof}[Proof of \cref{thm:LS}]
    Let $\gamma \in (0,1]$ and $a > 0$ such that \eqref{eq:thick} holds.
    Up to working with $g(x) := f(a x)$, which satisfies $\supp \widehat{g} \subset B_{\R^d}(0,N a)$, we can assume that $a = 1$.

    For each $j \in \Z^d$, set $E_j := j + [0,1]^d$ and $\omega_j := \omega \cap E_j$.
    By \eqref{eq:thick}, $|\omega_j| \geq \gamma$.

    Let $n \in \N^*$ to be chosen later.
    By the Remez inequality of \cref{cor:remez-f-Rd-L2} and $|\omega_j|\geq\gamma$,
    \begin{equation}
        \norm{f}_{L^2(E_j)} \leq C_\gamma^n \norm{f}_{L^2(\omega_j)} + d^{\frac n 2} \sum_{|\alpha| = n} \frac{1}{\alpha!} \norm{D^\alpha f}_{L^\infty(E_j)},
    \end{equation}
    where $C_\gamma := 16 e d / \gamma$.
    Squaring, using \eqref{eq:cards} and the Sobolev embedding of \cref{lem:sobolev},
    \begin{equation}
        \norm{f}_{L^2(E_j)}^2 \leq 2 C_\gamma^{2n} \norm{f}_{L^2(\omega_j)}^2 + \frac{2 d^{3n}}{n!^2} (n+1)^d 2^d \sum_{|\alpha| = n} \sum_{|\beta| \leq d} \norm{D^{\alpha+\beta} f}_{L^2(E_j)}^2.
    \end{equation}
    Summing over $j \in \Z^d$, we obtain
    \begin{equation}
        \label{eq:global-splitting}
        \norm{f}_{L^2(\R^d)}^2
        \leq 2 C_\gamma^{2n} \norm{f}_{L^2(\omega)}^2
        + \frac{2 d^{3n}}{n!^2} (n+1)^d 2^d \sum_{|\alpha| = n} \sum_{|\beta| \leq d} \norm{D^{\alpha+\beta} f}_{L^2(\R^d)}^2.
    \end{equation}
    Since $\supp \widehat{f} \subset B(0,N)$, by Bernstein’s inequality of \cref{lem:bernstein-Rd} and \eqref{eq:cards}, we obtain
    \begin{equation}
        \label{eq:global-tail}
        \norm{f}_{L^2(\R^d)}^2
        \leq 2 C_\gamma^{2n} \norm{f}_{L^2(\omega)}^2
        + 2 \varepsilon_N^2(n) \norm{f}_{L^2(\R^d)}^2
    \end{equation}
    where
    \begin{equation}
        \varepsilon_N(n) := 2^d \frac{(ed^2)^n}{n^n} (n+1)^d (d+1)^d N^{d + n}.
    \end{equation}
    Choosing $n_N := \lceil 2 e d^2 N \rceil$, we obtain
    \begin{equation}
        \varepsilon_N(n_N)
        \leq 2^d (d+1)^d \left( \frac{e d^2 N}{n_N} \right)^{n_N} (n_N + 1)^{2d}
        \leq C_d 2^{- 2 e d^2 N} (2 e d^2 N + 2)^{2d}.
    \end{equation}
    Hence, for $N \geq N_0(d)$, $\varepsilon_N(n_N) \leq \frac 12$.
    Returning to \eqref{eq:global-tail}, we conclude that
    \begin{equation}
        \norm{f}_{L^2(\R^d)} \leq 2 C_\gamma^{n_N} \norm{f}_{L^2(\omega)}
        \leq 2 C_\gamma^{2 e d^2 N +1} \norm{f}_{L^2(\omega)},
    \end{equation}
    which is the claimed estimate \eqref{eq:logvinenko-sereda}.
    We recover the optimal dependency on $\gamma$ discussed in \cite{Kovrijkine2001}.
\end{proof}

\subsection{Sobolev inequality}

In the proof of \cref{thm:LS}, we used the following highly-subcritical Sobolev embedding.

\begin{lemma}
    \label{lem:sobolev}
    Let $d \geq 1$.
    For all $f \in C^\infty([0,1]^d;\R)$,
    \begin{equation}
        \label{eq:sobolev}
        \norm{f}_{L^\infty([0,1]^d)}^2 \leq 2^d \sum_{|\alpha| \leq d} \norm{D^\alpha f}_{L^2([0,1]^d)}^2.
    \end{equation}
\end{lemma}

\begin{proof}
    We proceed by induction on $d \geq 1$.
    For $d = 1$ and $x, x' \in [0,1]$, by Cauchy--Schwarz
    \begin{equation}
        |f^2(x) - f^2(x')| = \left| \int_{x'}^{x} 2 f f' \right| \leq \norm{f}_{L^2(0,1)}^2 + \norm{f'}_{L^2(0,1)}^2.
    \end{equation}
    Integrating over $x' \in [0,1]$ leads to \eqref{eq:sobolev} for $d=1$.

    Now let $d \geq 2$ and assume that the inequality holds for dimension $d-1$.
    Let $x \in [0,1]^d$ be written as $x = (x_1, x')$ where $x_1 \in [0,1]$ and $x' \in [0,1]^{d-1}$.

    Fix $x' \in [0,1]^{d-1}$.
    Applying the base case $d=1$ to the function $g(x_1) := f(x_1, x')$ yields:
    \begin{equation}
        \sup_{x_1 \in [0,1]} |f(x_1, x')|^2 \leq 2 \int_0^1 \left( |f(x_1, x')|^2 + |\partial_1 f(x_1, x')|^2 \right) \dd x_1.
    \end{equation}
    To bound the $L^\infty$ norm over the full domain $[0,1]^d$, we take the supremum over $x'$:
    \begin{equation}
        \begin{split}
            \norm{f}_{L^\infty([0,1]^d)}^2
            & \leq 2 \sup_{x'} \int_0^1 \left( |f(x_1, x')|^2 + |\partial_1 f(x_1, x')|^2 \right) \dd x_1 \\
            & \leq 2 \int_0^1 \left(\sup_{x'} |f(x_1, x')|^2 + \sup_{x'} |\partial_1 f(x_1, x')|^2\right) \dd x_1 \\
            & \leq 2 \int_0^1 2^{d-1} \sum_{|\beta| \leq d-1} \left(\norm{D^\beta_{x'} f(x_1,\cdot)}_{L^2(x')}^2 + \norm{D^\beta \partial_1 f(x_1, \cdot)}_{L^2(x')}^2\right) \dd x_1 \\
            & \leq 2^d \sum_{|\beta| \leq d} \norm{D^\beta f}_{L^2}^2
        \end{split}
    \end{equation}
    where we applied the inductive hypothesis to $x' \mapsto f(x_1, x')$ and $x' \mapsto \partial_1 f(x_1, x')$ for fixed~$x_1$.
\end{proof}

\subsection{Gautschi bound for inverse Vandermonde matrix}

Given distinct $z_1, \dotsc, z_n \in \C$, define the square Vandermonde matrix $V := (z_j^{k-1})_{1 \leq j,k \leq n}$.
The following estimate of the norm of the inverse of $V$ is proved in \cite[Theorem 1]{Gautschi1962}.

\begin{lemma}
    \label{lem:gautschi}
    For any $n \geq 1$ and distinct $z_1, \dotsc, z_n \in \C$, one has
    \begin{equation}
        \label{eq:gautschi}
        \norm{V^{-1}} \leq \sqrt{n} \max_{1 \leq j \leq n} \prod_{k \neq j} \frac{1+|z_k|}{|z_j-z_k|}.
    \end{equation}
\end{lemma}

\begin{proof}
    For a matrix $A = (a_{jk})_{1 \leq j,k \leq n}$, we introduce its $1$-norm or \emph{column-sum} norm
    \begin{equation}
        \norm{A}_1 := \max_{\abs{x}_1=1} \abs{Ax}_1 = \max_{1 \leq k \leq n} \sum_{j=1}^n |a_{jk}|.
    \end{equation}
    In particular, by Cauchy--Schwarz, $\norm{A} = \max_{\abs{x}_2=1} \abs{Ax}_2 \leq \sqrt{n} \norm{A}_1$.

    Let $L_j$ for $1 \leq j \leq n$ be the Lagrange basis polynomials for the nodes $z_1, \dotsc, z_n$:
    \begin{equation}
        L_j(z) := \prod_{k \neq j} \frac{z-z_k}{z_j-z_k}
        \quad \text{so that} \quad
        L_j(z_k) = \mathbf{1}_{j=k}.
    \end{equation}
    Writing $L_j(z) = \ell^{(j)}_{0} + \ell^{(j)}_{1} z + \dotsb + \ell^{(j)}_{n-1} z^{n-1}$, the vector $\ell^{(j)}$ is the $j$-th column of $V^{-1}$ since $V \ell^{(j)} = e_j$.
    For a polynomial $Q(z) = \sum_{m=0}^{n-1} q_m z^m$, define $\norm{Q}_1 := |q_0| + \dotsb + |q_{n-1}|$.
    Then $\norm{z - z_0}_1 = 1 + |z_0|$ and $\norm{P Q}_1 \leq \norm{P}_1 \norm{Q}_1$.
    Thus
    \begin{equation}
        \norm{L_j}_1 \leq \prod_{k \neq j} \frac{1+|z_k|}{|z_j-z_k|}.
    \end{equation}
    By definition $\norm{V^{-1}}_1 = \max_{1 \leq j \leq n} \abs{\ell^{(j)}}_1$, so \eqref{eq:gautschi} follows by Cauchy--Schwarz.
\end{proof}

\newpage

\section{Well-posedness results}

\subsection{Control-affine ODEs}

Let $n \in \N^*$ and $f_0, f_1 \in C^1(\R^n;\R^n)$ vector fields with $f_0(0) = 0$.
We assume that $f_0$ and $f_1$ are globally Lipschitz, with constants $L_0$ and $L_1$, and globally bounded, with constants $M_0$ and $M_1$.

For $T > 0$, $u \in L^1((0,T);\R)$ and $y_0 \in \R^n$, let $t \mapsto y(t;u,y_0)$ be the solution to \eqref{eq:affine-f1} with initial data $y_0$ constructed in \cref{lem:WP}.
We collect standard estimates on this solution.

\subsubsection{A priori bounds}

\begin{lemma}
    \label{lem:EDO-bounds}
    For all $T > 0$, $u \in L^1((0,T);\R)$, $y_0 \in \R^n$ and $t \in [0,T]$,
    \begin{align}
        \label{eq:EDO-bound-1}
        \abs{y(t;u,y_0)} & \leq e^{L_0 t} \left(\abs{y_0} + M_1 \norm{u}_{L^1(0,t)}\right), \\
        \label{eq:EDO-bound-2}
        \abs{y(t;u,y_0)-y_0} & \leq e^{L_0 t} \left(L_0 t \abs{y_0} + M_1 \norm{u}_{L^1(0,t)}\right).
    \end{align}
\end{lemma}

\begin{proof}
    Since $f_0(0) = 0$, by the integral formulation \eqref{eq:mild},
    \begin{equation}
        \abs{y(t)} \leq \abs{y_0} + \int_0^t \abs{f_0(y) + u f_1(y)}
        \leq \abs{y_0} + \int_0^t L_0 \abs{y} + M_1 \int_0^t \abs{u}
    \end{equation}
    Thus \eqref{eq:EDO-bound-1} follows by Grönwall's inequality.
    Similarly,
    \begin{equation}
        \abs{y(t) - y_0} \leq \int_0^t \abs{f_0(y) + u f_1(y)}
        \leq \int_0^t L_0 \abs{y-y_0} + L_0 \abs{y_0} + M_1 \abs{u},
    \end{equation}
    which entails \eqref{eq:EDO-bound-2} by Grönwall's inequality.
\end{proof}

\subsubsection{Lipschitz dependence on the parameters}

\begin{lemma}
    \label{lem:EDO-Lip}
    For all $T > 0$, $u, v \in L^1((0,T);\R)$, $y_0, z_0 \in \R^n$ and $t \in [0,T]$,
    \begin{equation}
        \label{eq:EDO-Lip}
        \abs{y(t;u,y_0) - y(t;v,z_0)} \leq e^{L_0 T + L_1 \norm{u}_{L^1}} \left(\abs{y_0-z_0} + M_1 \norm{u-v}_{L^1} \right).
    \end{equation}
\end{lemma}

\begin{proof}
    Let $y$ denote the trajectory for the control $u$ and initial data $y_0$, $z$ the trajectory for the control $v$ and initial data $z_0$.
    Then $w := y - z$ satisfies the mild formulation:
    \begin{equation}
        w(t) = (y_0 - z_0) + \int_0^t (f_0 + u(s) f_1)(y(s)) - (f_0 + v(s) f_1)(z(s)) \dd s.
    \end{equation}
    By the triangle inequality,
    \begin{equation}
        \begin{split}
            \abs{w(t)} & \leq \abs{y_0-z_0} + \int_0^t \abs{f_0(y) - f_0(z)} + \abs{u (f_1(y) - f_1(z))} + \abs{(u-v) f_1(z)} \\
            & \leq \abs{y_0-z_0} + \int_0^t (L_0 + \abs{u} L_1) \abs{w} + M_1 \int_0^t \abs{u-v}.
        \end{split}
    \end{equation}
    Grönwall's inequality entails \eqref{eq:EDO-Lip}.
\end{proof}

\subsubsection{Taylor expansions}

\begin{lemma}
    \label{lem:taylor-uf1}
    For $T \in [0,1]$ and $u \in L^1((0,T);\R)$,
    \begin{equation}
        \label{eq:taylor-uf1}
        y(T;u,0) = \Big(\int_0^T u\Big) f_1(0) + \cO((T+\norm{u}_{L^1}) \norm{u}_{L^1}).
    \end{equation}
\end{lemma}

\begin{proof}
    Let $w(T) := y(T;u,0) - (\int_0^T u) f_1(0)$.
    By the mild formulation \eqref{eq:mild}, since $f_0(0) = 0$,
    \begin{equation}
        \begin{split}
            \abs{w(T)}
            & \leq \int_0^T \abs{f_0(y(t;u,0)) + u(t) f_1(y(t;u,0)) - u(t) f_1(0)} \dd t \\
            & \leq \int_0^T (L_0 + L_1 \abs{u(t)}) \abs{y(t;u,0)} \dd t.
        \end{split}
    \end{equation}
    Using \eqref{eq:EDO-bound-1}, we obtain
    \begin{equation}
        \abs{w(T)} \leq (L_0 T + L_1 \norm{u}_{L^1(0,T)}) e^{L_0 T} M_1 \norm{u}_{L^1(0,T)}.
    \end{equation}
    This entails \eqref{eq:taylor-uf1} for $T \leq 1$.
\end{proof}

\begin{lemma}
    \label{lem:autonomous}
    Assume that $f_0 \in C^2(\R^n;\R^n)$.
    Let $A := Df_0(0)$.
    For $t \in [0,1]$ and $y_0 \in \R^n$,
    \begin{equation}
        y(t;0,y_0) = y_0 + t A y_0 + \cO(t^2 |y_0| + t |y_0|^2).
    \end{equation}
\end{lemma}

\begin{proof}
    Let $w(t) := y(t;0,y_0) - y_0 - t A y_0$.
    By the mild formulation \eqref{eq:mild},
    \begin{equation}
        \begin{split}
            \abs{w(t)} & \leq \int_0^t \abs{f_0(y(s;0,y_0)) - A y_0} \dd s \\
            & \leq \int_0^t \abs{f_0(y_0) - A y_0} + \abs{f_0(y(s;0,y_0)) - f_0(y_0)} \dd s \\
            & \leq t \abs{f_0(y_0) - A y_0} + L_0 \int_0^t \abs{y(s;0,y_0)-y_0} \dd s \\
            & \leq t \abs{f_0(y_0) - A y_0} + (L_0 t)^2 e^{L_0 t} \abs{y_0}
        \end{split}
    \end{equation}
    since, by \eqref{eq:EDO-bound-2} with $u = 0$, $\abs{y(s;0,y_0)-y_0} \leq e^{L_0 s} L_0 s \abs{y_0}$.

    Since $f_0 \in C^2(\R^n;\R^n)$, there exists $C_0 > 0$ such that $|f_0(y_0) - Ay_0| \leq C_0 |y_0|^2$ when $|y_0| \leq 1$.
    For $|y_0| > 1$, write $|f_0(y_0) - Ay_0| \leq |f_0(y_0)-f_0(0)| + |Ay_0| \leq L_0 |y_0| + \norm{A} |y_0| \leq (L_0 + \norm{A}) |y_0|^2$.
\end{proof}

\subsection{Burgers equation}
\label{sec:WP-Burgers}

We prove \cref{prop:WP-Burgers} using a classical compactness argument.
To prove the existence of solutions, we will rely on Schaefer's fixed-point theorem \cite{Schaefer1955}, which is an easy consequence of Schauder's fixed-point theorem (see e.g.\ \cite[Chapter 3, Theorem 4.4]{Cronin1964} for a short proof), sometimes easier to apply to prove the existence of solutions to PDEs.

\begin{lemma}[Schaefer's fixed-point theorem]
    \label{lem:schaefer}
    Let $E$ be a Banach space and $f$ a continuous compact map from $E$ to $E$.
    Assume that the set
    \begin{equation}
        \label{eq:schaefer-set}
        \left\{ y \in E \mid \exists \lambda \in [0,1], y = \lambda f(y) \right\}
    \end{equation}
    is bounded.
    Then $f$ has a fixed point.
\end{lemma}

We will therefore need the following \emph{a priori} estimates.

\begin{lemma}
    \label{lem:Burgers-a-priori}
    There exists a constant $C > 0$ such that, for any $\lambda \in [0,1]$, $T > 0$, $y_0 \in L^2(0,1)$ and $u \in L^2((0,T) \times \omega)$, if $y \in \mathcal{Y}$ is a solution to
    \begin{equation}
        \label{eq:Burgers-lambda}
        \begin{cases}
            y_t - y_{xx} = \lambda (u \mathbf{1}_\omega - y y_x) & \text{in } (0,T) \times (0,1), \\
            y(t,0) = y(t,1) = 0 & \text{in } (0,T), \\
            y(0,x) = \lambda y_0(x) & \text{in } (0,1),
        \end{cases}
    \end{equation}
    then
    \begin{equation}
        \norm{y - \lambda z}_{\mathcal{Y}} \leq C \left(\norm{y_0}_{L^2} + \norm{u}_{L^2} \right)^2,
    \end{equation}
    where $z \in \mathcal{Y}$ is the solution to the linear heat equation \eqref{eq:heat-1D-internal}.
\end{lemma}

\begin{proof}
    First, let us obtain \emph{a priori} estimates on $y$.
    Multiplying \eqref{eq:Burgers-lambda} by $y$, integrating over $(0,1)$ and using integration by parts, the boundary conditions and Poincaré, we obtain
    \begin{equation}
        \frac 12 \frac{\dd}{\dd t} \int_0^1 y^2 + \int_0^1 y_x^2 = \lambda \int_0^1 u y
        \leq \frac 12 \int_0^1 u^2 + \frac 12 \int_0^1 y_x^2.
    \end{equation}
    Hence,
    \begin{equation}
        \norm{y}_{L^\infty L^2}^2 + \norm{y}_{L^2 H^1}^2
        \leq \norm{y_0}_{L^2}^2 + \norm{u}_{L^2}^2.
    \end{equation}
    Moreover, $\norm{(y^2)_x}_{L^2 H^{-1}} \leq \norm{y}_{L^4 L^4}^2 \leq \norm{y}_{L^\infty L^2} \norm{y}_{L^2 H^1}$.
    Hence, using the PDE, we can estimate $y_t$ in $L^2 H^{-1}$.
    We conclude that $\norm{y}_{\mathcal{Y}} \leq C (\norm{y_0}_{L^2}^2 + \norm{u}_{L^2}^2)$.

    Let $w := y - \lambda z \in \mathcal{Y}$.
    Then $w$ is a solution to
    \begin{equation}
        \begin{cases}
            w_t - w_{xx} = - \lambda yy_x & \text{in } (0,T) \times (0,1), \\
            w(t,0) = w(t,1) = 0 & \text{in } (0,T), \\
            w(0,x) = 0 & \text{in } (0,1).
        \end{cases}
    \end{equation}
    By \cref{prop:heat-f-WP}, $\norm{w}_{\mathcal{Y}} \leq C \norm{\lambda y y_x}_{L^2 H^{-1}} \leq C \norm{y}_{\mathcal{Y}}^2$, which concludes the proof.
\end{proof}

We can now turn to the proof of the well-posedness.

\begin{proof}[Proof of \cref{prop:WP-Burgers}]
    Fix $T > 0$, $y_0 \in L^2(0,1)$ and $u \in L^2((0,T) \times \omega)$.

    \medskip
    \step{Existence}
    Let $E := L^4((0,T);L^4(0,1))$.
    The map $w \mapsto w w_x = \frac 12 \partial_x (w^2)$ is continuous from $E$ to $L^2((0,T);H^{-1}(0,1))$.
    Let $f : E \to E$ be the map which associates to $w \in E$ the solution $y \in \mathcal{Y}$ given by \cref{prop:heat-f-WP} to
    \begin{equation}
        \begin{cases}
            y_t - y_{xx} = u \mathbf{1}_\omega - w w_x & \text{in } (0,T) \times (0,1), \\
            y(t,0) = y(t,1) = 0 & \text{in } (0,T), \\
            y(0,x) = y_0(x) & \text{in } (0,1).
        \end{cases}
    \end{equation}
    By continuity of $w \mapsto w w_x$ from $E$ to $L^2((0,T);H^{-1}(0,1))$ and \cref{prop:heat-f-WP}, $f$ is continuous from $E$ to $\mathcal{Y}$.
    Moreover, $\mathcal{Y} \hookrightarrow E$ and this embedding is compact (this follows for example from the Aubin--Lions--Simon lemma, see e.g.\ \cite[Section 9, Corollary 6]{Simon1987}).
    Hence $f$ is compact as a map from $E$ to $E$.
    To apply \cref{lem:schaefer}, it remains to check that the set \eqref{eq:schaefer-set} is bounded.
    Let $y \in E$ and $\lambda \in [0,1]$ such that $y = \lambda f(y)$.
    Since $f \in C^0(E;\mathcal{Y})$, $y \in \mathcal{Y}$ and satisfies \eqref{eq:Burgers-lambda}.
    By \cref{lem:Burgers-a-priori} and \cref{prop:heat-f-WP}, $\norm{y}_{\mathcal{Y}} \leq C (M + M^2)$, where $M := \norm{y_0}_{L^2} + \norm{u}_{L^2}$.
    Hence, by Schaefer's fixed-point theorem, there exists $y \in E$ such that $y = f(y)$.
    In particular, $y \in \mathcal{Y}$.

    \step{Uniqueness}
    Consider two solutions $y, \tilde{y} \in \mathcal{Y}$ to \eqref{eq:Burgers}.
    Let $w := y - \tilde{y}$ and $z := y + \tilde{y}$.
    Then $w \in \mathcal{Y}$ is a solution to
    \begin{equation}
        \begin{cases}
            w_t - w_{xx} = - \frac 12 (w z)_x & \text{in } (0,T) \times (0,1), \\
            w(t,0) = w(t,1) = 0 & \text{in } (0,T), \\
            w(0,x) = 0 & \text{in } (0,1).
        \end{cases}
    \end{equation}
    Multiplying by $w$ and integrating by parts yields:
    \begin{equation}
        \frac 12 \frac{\dd}{\dd t} \int_0^1 w^2 + \int_0^1 w_x^2 = - \frac 14\int_0^1 z_x w^2.
    \end{equation}
    By Cauchy--Schwarz, Poincaré and Young, for each $t \in [0,T]$,
    \begin{equation}
        \left| \int_0^1 z_x w^2 \right|
        \leq \norm{z_x}_{L^2} \norm{w}_{L^4}^2
        \leq \norm{z_x}_{L^2} \norm{w}_{L^2} \norm{w_x}_{L^2}
        \leq \norm{z_x}_{L^2}^2 \norm{w}_{L^2}^2 + \norm{w_x}_{L^2}^2.
    \end{equation}
    Hence, setting $M(t) := \norm{z_x(t)}_{L^2}^2$,
    \begin{equation}
        \frac{\dd}{\dd t} \norm{w(t)}_{L^2}^2 \leq M(t) \norm{w(t)}_{L^2}^2.
    \end{equation}
    Since $z \in \cY$, $M \in L^1(0,T)$.
    Since $w(0) = 0$, Grönwall's lemma implies that $w \equiv 0$ on $[0,T]$.
\end{proof}

\section*{Acknowledgments}

I am indebted to many colleagues for the numerous discussions that helped shape this course.

I am particularly grateful to Karine Beauchard for our ongoing collaboration on the controllability of nonlinear systems, and for extensive exchanges on almost all subsections of this course.
I also thank Franck Sueur for stimulating conversations on the abstraction of \cref{thm:LTT} as part of our research on the Navier--Stokes equations, and Thomas Perrin for \cref{sec:WP-Burgers}.

I would like to thank Jérémi Dardé and Julien Royer for organizing the EUR-MINT 2025 summer school, the other lecturers, Lucie Beaudouin, Lucas Chesnel, and David Krejčiřík, as well as all the students, for the pleasant atmosphere and the fruitful exchanges we enjoyed together.

\bibliographystyle{plain}
\bibliography{control}

\end{document}